\documentclass[11pt]{article}

\usepackage[utf8]{inputenc}
\usepackage[T1]{fontenc}
\usepackage{lmodern}
\usepackage{amsmath,amsfonts,amssymb,dsfont,bm,mathtools}
\usepackage{amsthm}
\RequirePackage[colorlinks,citecolor=blue,urlcolor=blue]{hyperref}
\usepackage{graphicx}
\usepackage[plain]{fullpage} 
\usepackage{authblk}
\usepackage{xcolor}
\usepackage{csquotes}
\usepackage{enumitem}
\usepackage{algorithm}
\usepackage{algorithmic}
\usepackage{yhmath}

\newcommand{\R}{\ensuremath{\mathbb{R}}}
\newcommand{\norm}[1]{\ensuremath{\left\|#1\right\|}}

\newcommand{\regSp}{\mathcal{C}}
\renewcommand{\L}{\mathbb{L}} 

\DeclareMathOperator{\Var}{Var}
\DeclareMathOperator{\cov}{Cov}

\DeclareMathOperator{\HSIC}{HSIC}
\DeclareMathOperator{\HilbSchmidt}{HS}

\DeclareMathOperator{\HH}{HSIC}
\DeclareMathOperator{\E}{\mathbb{E}} 
\usepackage{scalerel,stackengine} 
\stackMath
\newcommand\reallywidehat[1]{%
\savestack{\tmpbox}{\stretchto{%
  \scaleto{%
    \scalerel*[\widthof{\ensuremath{#1}}]{\kern-.6pt\bigwedge\kern-.6pt}%
    {\rule[-\textheight/2]{1ex}{\textheight}}
  }{\textheight}%
}{0.5ex}}%
\stackon[1pt]{#1}{\tmpbox}%
}

\newcommand{\pa}[1]{\left(#1\right)}
\newcommand{\cro}[1]{\left[#1\right]}
\newcommand{\ac}[1]{\left\{#1\right\}}
\newcommand{\abs}[1]{\left| #1 \right|}
\newcommand{\rperm}{\tau} 					
\newcommand{\rpermb}{\kappa} 					
\newcommand{\echperm}{\pi}					
\newcommand{\perm}{\sigma}					
\DeclareMathOperator{\id}{id}				
\newcommand{\ZZ}{\mathbb{Z}}				
\newcommand{\proba}[1]{\mathbb{P}\!\pa{#1}}	
\newcommand{\esp}[1]{\mathbb{E}\!\cro{#1}}	
\newcommand{\esps}[2]{\mathbb{E}_{#1}\!\cro{#2}}	
\newcommand{\I}{I}							
\DeclareMathOperator{\Card}{Card}

\newcommand{\W}{\mathcal{W}}

\newtheorem{prop}{Proposition}

\newtheorem{corr}{Corollary}
\newtheorem{lemm}{Lemma}

\newtheorem{theorem}{Theorem}

\numberwithin{equation}{section}
\theoremstyle{plain}

\title{\textsc{Adaptive test of independence based on HSIC measures.}}
\date{}
\author[,1]{M\'elisande Albert\thanks{Electronic address: melisande.albert@insa-toulouse.fr}}
\author[,1]{Béatrice Laurent\thanks{Electronic address: beatrice.laurent@insa-toulouse.fr}}
\author[,2]{Amandine Marrel\thanks{Electronic address: amandine.marrel@cea.fr}}
\author[,1,2]{Anouar Meynaoui\thanks{Electronic address: anouar.meynaoui@gmail.com}}

\affil[1]{\small Institut de Mathématiques de Toulouse ; UMR5219, Université de Toulouse ; CNRS, INSA, F-31077 Toulouse, France.}
\affil[2]{\small CEA, DEN, DER, F-13108 Saint-Paul-lez-Durance, France.}

\begin{document}

\maketitle

\vspace{-18pt}
\begin{small}{\bf Abstract:}~The Hilbert-Schmidt Independence Criterion (HSIC) is a dependence measure based on reproducing kernel Hilbert spaces that is widely used to test independence between two random vectors. Remains the delicate choice of the kernel. 
In this work, we develop a new HSIC-based aggregated procedure which avoids such a kernel choice, and provide theoretical guarantees for this procedure. 
To achieve this, on the one hand, we introduce non-asymptotic single tests based on Gaussian kernels with a given bandwidth, which are of prescribed level. 
Then, we aggregate several single tests with different bandwidths, and prove sharp upper bounds for the uniform separation rate of the aggregated procedure over Sobolev balls. 
On the other hand, we provide a lower bound for the non-asymptotic minimax separation rate of testing over Sobolev balls, and deduce that the aggregated procedure is adaptive in the minimax sense over such regularity spaces. 
Finally, from a practical point of view, we perform numerical studies in order to assess the efficiency of our aggregated procedure and compare it to existing tests in the literature.

\medskip


\smallskip


\end{small}

\section{Introduction to independence testing}
\label{Introduction}

\begin{sloppypar}
Many nonparametric approaches to test independence between two continuous random vectors have been explored in the last few decades.
Among them, \cite{hoeffding1948non} introduces a test based on the difference between the joint distribution function and the product of the marginal distribution functions. 
This test has good properties in the asymptotic framework since it is consistent. 
Yet, it only applies to univariate random variables. 
Lately, \cite{weihs2018symmetric} extend Hoeffding's test to the case of multivariate random variables, but still in an asymptotic framework. 
Another classical method for testing independence is based on comparing the joint density and the product of the marginal densities \cite{rosenblatt1975quadratic,ahmad1997testing}. For this, an intermediate step is to estimate these densities using, e.g., the kernel-based method of Parzen-Rosenblatt \cite{parzen1962estimation}. 
More recently, many approaches based on reproducing kernel Hilbert spaces (RKHS) have been developed (see \cite{aronszajn1950theory} for more details). 
One of the first RKHS measures is the kernel canonical correlation (KCC) \cite{bach2002kernel}. 
Yet, the estimation of the KCC is not practical since it requires an extra regularization which has to be adjusted. 
Other dependence measures, easier to estimate have been studied later. For instance, the kernel mutual information (KMI) \cite{gretton2003kernel,gretton2005kernel} and the constrained covariance (COCO) \cite{gretton2005kernelConstrained,gretton2005kernel} are widely used, since they are relatively easy to interpret and to implement. 
Finally, one of the most interesting kernel dependence measure is the Hilbert-Schmidt independence criterion (HSIC) \cite{gretton2005measuring}. 
The HSIC has a very low computational cost and seems to numerically outperform all previous RKHS measures \cite{gretton2005measuring}.
A first independence test based on the HSIC is developed using large deviation inequalities \cite{gretton2005measuring}. Then, other tests are constructed in \cite{gretton2008kernel,li2019optimality} using an approximation of the null distribution of the HSIC estimator either by an asymptotic Gamma distribution or by a permutation approach.
A generalization to joint and mutual independence testing is presented in \cite{pfister2018kernel}. 
We also mention the RKHS-based test \cite{poczos2012copula}, using the copula-based kernel dependency measure. 
Yet, this test is more conservative than the test of \cite{gretton2008kernel}, since it is based on large deviation inequalities. 
Lately, based on characteristic functions, \cite{szekely2007measuring} introduce the distance covariance which has good properties and can be used in high dimensional frameworks \cite{szekely2013distance, yao2018testing}. Furthermore, it has been shown that the distance covariance coincides with the HSIC for a specific choice of kernels. 
Other tests have emerged based for instance on a sample space partitioning \cite{heller2016consistent} or based on binary expansion \cite{zhang2019bet} and very recently extended to any arbitrary dimension \cite{lee2019testing}.
Finally, the authors of \cite{berrett2019nonparametric} introduce a new test based on nearest neighbour methods and kernel mutual information which seems to achieve comparable results with the classical tests based on HSIC. In this paper, we focus on HSIC measures to test independence.
\end{sloppypar}

\subsection{Adaptive independence tests} 

\begin{sloppypar}
To study the non-asymptotic performances of testing, we consider the \emph{uniform separation rate} as defined in \cite{baraud2002non}. For any $\alpha$-level test $ \Delta_{\alpha}$ with values in $\{0,1\}$, which rejects independence when $\Delta_{\alpha}=1$, the uniform separation rate $\rho \left( \Delta_{\alpha} , \regSp_\delta , \beta \right)$ of $ \Delta_{\alpha}$, over a class $\regSp_\delta$ of regular alternatives $f$ (such that the difference between the density $f$ and the product of its marginals $f_1\otimes f_2$ satisfies smoothness assumptions), with respect to (w.r.t.) the $\L_2$-norm, is defined for all $\beta$ in $(0,1)$ by
\begin{equation}\label{seprate}
\rho \left( \Delta_{\alpha}, \regSp_\delta , \beta \right) = \inf \ac{ \rho > 0; \sup_{f \in \mathcal{F}_\rho(\regSp_\delta)} {P}_{f} \left( \Delta_{\alpha} = 0 \right) \leq \beta },
\end{equation} 
where $ \mathcal{F}_\rho(\regSp_\delta) = \left\{ f; f-f_1\otimes f_2 \in \regSp_\delta, \norm{f-f_1\otimes f_2}_2 > \rho \right\}$ and $\norm{\cdot}_2$ designates the usual $\L_2$-norm. 
The uniform separation rate is then the smallest value in the sense of the $\L_2$-norm of $f - f_1 \otimes f_2$ allowing to control the second kind error of the test by $\beta$. 
A test of level $\alpha$ having the optimal performances, should then have a uniform separation rate as small as possible over $\regSp_\delta$. 
To quantify this, let us define, as in \cite{baraud2002non}, the \emph{non-asymptotic minimax rate of testing} by 
\begin{equation}\label{minimaxrate}
\rho\left( \regSp_\delta , \alpha, \beta \right) = \inf_{\Delta_{\alpha}} \rho \left( \Delta_{\alpha}, \regSp_\delta , \beta \right),
\end{equation}
where the infimum is taken over all $\alpha$-level tests. 
If the uniform separation rate of a test is upper bounded up to a constant by the non-asymptotic minimax rate of testing, then this test is said to be \emph{optimal in the minimax sense}. 
The problem of non-asymptotic minimax rate of testing was raised in many papers over the past years. Among them, we mention for example \cite{ingster1998minimax, laurent2012non} for minimax signal detection testing. 
Concerning independence testing, optimality in the minimax sense defined as above is closely related to \emph{the asymptotic minimax rate} as introduced in the notable works of Ingster \cite{ingster1989asymptotically, ingster1993minimax}, of Yod\'e \cite{yode2004asymptotically, yode2011adaptive} or very recently of \cite{li2019optimality}. Lately, \cite{berrett2019nonparametric} study upper bounds w.r.t. mutual information, and \cite{ramdas2016minimax} obtain minimax lower bounds for linear independence testing. 
In the non-asymptotic framework considered in this paper, \cite{albert2015tests} obtains upper bounds w.r.t. the $\L_2$ distance over weak Besov spaces. Concurrent with our work, and independently, \cite{berrett2020optimal} and \cite{kim2020minimax} also obtain minimax separation rates for independence tests based on permuted $U$-statistics. 
\end{sloppypar}

\begin{sloppypar}
Furthermore, beyond the problem of minimax optimality, the straightforward practical construction of a minimax test usually depends on the unknown smoothness parameter $\delta$ of the regularity space $\regSp_\delta$. The objective is then to construct a minimax test which does not need any smoothness assumption to be implemented. These tests are called \emph{minimax adaptive}. 
The problem of adaptivity has received a great attention in the literature. We mention for instance the works of \cite{baraud2003adaptive} for linear regression model testing with Gaussian noise and of \cite{ingster2000adaptive,fromont2006adaptive,balakrishnan2019hypothesis} for goodness-of-fit testing. 
The authors of \cite{fromont2013two} consider an interesting approach for testing the equality of two Poisson processes intensities, which consists in aggregating several single kernel-based tests, and prove that it is adaptive over several regularity spaces. 
This paper lies in the lineage of these works.
\end{sloppypar}

\subsection{Mathematical framework and notation}
\label{sect:notation}

\begin{sloppypar}
In this work, we study the problem of testing the independence between two continuous real random vectors $X = ( X^{(1)}, \ldots ,X^{(p)} )$ and $Y = (Y^{(1)}, \ldots ,Y^{(q)})$. 
The couple $(X,Y)$ is assumed to have a joint density $f$ w.r.t. Lebesgue measure on $\R^p\times \R^q$, with marginal density functions $f_1$ and $f_2$. To avoid any misunderstanding, let us highlight that $f_1$ and $f_2$ are assumed to be unknown and are not fixed \emph{a priori}. 
We denote by $f_1\otimes f_2:(x,y)\in\R^p\times \R^q \mapsto f_1(x)f_2(y)$ the product of the marginal densities.
We also assume that we observe a $n$-sample $(X_1,Y_1), \ldots ,(X_n,Y_n)$ of independent and identically distributed (i.i.d.) random variables with common density $f$. 
The probability measure associated to this $n$-sample is denoted $P_f$. By analogy, $P_{f_1 \otimes f_2}$ designates the probability measure associated to a $n$-sample with common density $f_1 \otimes f_2$. 
\end{sloppypar}

We address here the question of testing the null hypothesis $(\mathcal{H}_0)$: \enquote{$X$ and $Y$ are independent} against the alternative $(\mathcal{H}_1)$: \enquote{$X$ and $Y$ are dependent}. That is equivalent to test 
\begin{center}
$(\mathcal{H}_0)$: \enquote{$f = f_1 \otimes f_2$} \quad against \quad $(\mathcal{H}_1)$: \enquote{$f \neq f_1 \otimes f_2$}. 
\end{center} 
Throughout this document, we consider the following assumption:
\begin{align*}
\boldsymbol{\mathcal{A}_1}:\ & \mbox{the density $f$, and its marginal densities $f_1$ and $f_2$ are bounded,}\\
& \mbox{and denote }M_f=\max\ac{\norm{f}_{\infty}, \norm{f_1}_{\infty}, \norm{f_2}_{\infty}}. 
\end{align*}

Moreover, the generic notation $C(a,b, \ldots)$ denotes a positive constant depending only on its arguments $(a,b, \ldots)$ and that may vary from line to line. 
Finally, the dimensions $p$ and $q$ are assumed to be fixed, and do not depend on the sample size. 

\subsection{Review on HSIC measures}
\label{sect:HSICmeasures}

\begin{sloppypar}
The definition of the HSIC is derived from the notion of cross-covariance operator \cite{baker1973joint, fukumizu2004dimensionality}, which can be seen as a generalization of the classical covariance, measuring many forms of dependence between $X$ and $Y$ (not only linear ones). For this, \cite{gretton2005measuring} associate to $X$ an RKHS $\mathcal{F}$ composed of functions mapping from $\mathbb{R}^p$ to $\mathbb{R}$ ($\mathcal{F}$ is a set of transformations for $X$), and characterized by a scalar product $\langle \cdot , \cdot \rangle_\mathcal{F}$. The same operation is carried out for $Y$, considering an RKHS denoted $\mathcal{G}$ and a scalar product $\langle \cdot , \cdot \rangle_\mathcal{G}$. The cross-covariance operator $C_{X,Y}$ associated to $\mathcal{F}$ and $\mathcal{G}$ is the operator mapping from $\mathcal{G}$ to $\mathcal{F}$ and verifying for all $(F,G) \in \mathcal{F} \times \mathcal{G}$,
\begin{equation*}
\langle F , C_{X,Y} (G) \rangle_\mathcal{F} = \cov \left( F(X),G(Y) \right). 
\end{equation*}
Designating by $(u_i)_i$ and $(v_j)_j$ respectively orthonormal bases of $\mathcal{F}$ and $\mathcal{G}$, the HSIC between $X$ and $Y$ is the square Hilbert-Schmidt norm of the operator $C_{X,Y}$  defined as in \cite{gretton2005measuring} by
\begin{equation*}
\HSIC (X,Y) = \norm{C_{X,Y}}_{\HilbSchmidt}^2 = \sum_{i,j} \langle u_i , C_{X,Y} (v_j) \rangle_\mathcal{F}^2 = \displaystyle \sum_{i,j} \cov \left( u_i(X), v_j(Y) \right)^2.
\end{equation*}
The fundamental idea behind this definition is that $\HSIC (X,Y)$ equals zero if and only if  
$\cov \left( F(X), G(Y) \right) = 0$ for all $(F,G)$ in $\mathcal{F} \times \mathcal{G}$. 
Furthermore, $X$ and $Y$ are independent if and only if $\cov \left( F(X), G(Y) \right) = 0$ for all bounded and continuous functions $F$ and $G$ (see e.g. \cite{jacod2012probability}). 
It follows that, for well chosen RKHS, the nullity of the HSIC characterizes independence. 
Authors of \cite{gretton2005kernelConstrained} show that a sufficient condition so that the nullity of the associated HSIC characterizes independence is that the RKHS $\mathcal{F}$ (resp. $\mathcal{G}$) induced by a kernel $k$ (resp. $l$) is dense in the space of bounded and continuous functions mapping from $\mathbb{R}^p$ (resp. $\mathbb{R}^q$) to $\mathbb{R}$. 
Such kernels are called \emph{universal} \cite{micchelli2006universal}. However, the universality is a very limiting condition and only adapted to compact domains. Recently, a wider class of kernels called \emph{characteristic kernels} has been introduced in \cite{fukumizu2008kernel, sriperumbudur2010hilbert}. These kernels characterize independence on compact as well as non-compact sets. Among them, one of the most commonly used is the Gaussian kernel \cite{steinwart2001influence}, which we consider in this paper. It is defined as follows. Let $g_d$ be the density of the standard Gaussian distribution on $\R^d$ defined for all $x=(x^{(1)},\ldots,x^{(d)})$ in $\R^d$ by
\begin{equation}\label{gs}
g_d(x)= \frac1{(2\pi)^{d/2}} \exp \left( - \frac12\sum_{i=1}^d \cro{x^{(i)}}^2 \right).
\end{equation}
For any bandwidths $\lambda = \left( \lambda_1,\ldots, \lambda_p \right)$ in $(0,+ \infty)^p$ and $\mu = \left( \mu_1,\ldots, \mu_q \right)$ in $(0,+ \infty)^q$, we denote for any $x$ in $\R^p$ and $y$ in $\R^q$, 
\begin{equation}\label{varphilambdaphimu}
\varphi_{\lambda}(x) = \frac1{\lambda_1 \ldots \lambda_p } g_p\pa{\frac{x^{(1)}}{\lambda_1},\ldots , \frac{x^{(p)}}{\lambda_p}}, \quad \phi_{\mu}(y) = \frac1{\mu_1 \ldots \mu_q } g_q\pa{\frac{y^{(1)}}{\mu_1},\ldots , \frac{y^{(q)}}{\mu_q}}.
\end{equation}
Finally, the Gaussian kernels are defined for $x, x'$ in $\R^p$ and $y, y'$ in $\R^q$ by
\begin{equation*}
 k_{\lambda} (x,x')= \varphi_{\lambda}(x - x'), \quad \mbox{and}\quad l_{\mu}(y,y')= \phi_{\mu}(y-y').
\end{equation*}
\end{sloppypar}

A very convenient form of $\HSIC(X,Y)$ is expressed in \cite{gretton2005measuring} using kernels $k$ and $l$ respectively associated to $\mathcal{F}$ and $\mathcal{G}$,
\begin{align}
\HSIC (X,Y) &= \esp{k (X , X') l(Y , Y')} + \esp{k \left(X , X' \right)} \esp{l \left( Y , Y' \right)} \nonumber \\ &- 2 \mathbb{E} \Big[ \esp{k \left(X , X' \right) \mid X} \esp{l \left( Y , Y' \right) \mid Y} \Big],
\label{formHSIC}
\end{align}
where $(X',Y')$ is an i.i.d. copy of $(X,Y)$. Note that $\HSIC(X,Y)$ only depends on the density $f$ of $(X,Y)$. 
Hence, in the following, we denote by $\HSIC_{\lambda,\mu}(f)$ the HSIC measure defined in (\ref{formHSIC}), where the kernels $k$ and $l$ are respectively the Gaussian kernels $k_{\lambda}$ and $l_{\mu}$. \\

Given an i.i.d. $n$-sample $(X_i,Y_i)_{1 \leq i \leq n}$ with common density $f$, an estimator of $\HSIC_{\lambda,\mu}(f)$ can be obtained by estimating each expectation of Equation \eqref{formHSIC}. For this, we introduce the following $U$-statistics, respectively of order 2, 3 and 4, 
\begin{eqnarray*}
\widehat{\HSIC}_{\lambda,\mu}^{(2)} &=& \frac{1}{n(n-1)} \sum_{(i,j) \in \mathbf{i}^n_2} k_\lambda\pa{X_i,X_j} l_\mu\pa{Y_i,Y_j},\\
\widehat{\HSIC}_{\lambda,\mu}^{(3)} &=& \frac{1}{n(n-1)(n-2)} \sum_{(i,j,r) \in \mathbf{i}^n_3} k_\lambda\pa{X_i,X_j} l_\mu\pa{Y_j,Y_r},\\
\widehat{\HSIC}_{\lambda,\mu}^{(4)} &=& \frac{1}{n(n-1)(n-2)(n-3)} \sum_{(i,j,r,s) \in \mathbf{i}^n_4} k_\lambda\pa{X_i,X_j} l_\mu\pa{Y_r,Y_s},
\end{eqnarray*}
where 
$\mathbf{i}^n_r$ is the set of all r-tuples drawn without replacement from $\{ 1,\ldots, n\}$.
We estimate $\HSIC_{\lambda,\mu}(f)$ by the $U$-statistic
\begin{equation}\label{esHSIC}
\widehat{\HSIC}_{\lambda,\mu} = \widehat{\HSIC}_{\lambda,\mu}^{(2)} + \widehat{\HSIC}_{\lambda,\mu}^{(4)} - 2 \widehat{\HSIC}_{\lambda,\mu}^{(3)}.
\end{equation}
\begin{sloppypar}
Similar estimators have been used to construct independence tests (see e.g. \cite{gretton2008kernel}). 
Yet, only a heuristic choice of the bandwidths $\lambda$ and $\mu$ is considered with no theoretical guarantees.
To avoid this choice, following the work of \cite{fromont2013two}, we introduce in this paper an aggregated procedure based on Gaussian kernel HSIC measures and prove that it is minimax adaptive over Sobolev balls. Note that in the continuity of our work, \cite{kim2020minimax} obtain minimax adaptive results over H\"older spaces for two-sample and independence tests based on permutations.\\
\end{sloppypar}

The structure of this paper is as follows. In Section \ref{Single kernel}, we first present a theoretical non-asymptotic HSIC-based test of prescribed level $\alpha$ as well as a permutation-based HSIC-test that is implemented in practice. 
We then provide theoretical conditions based on concentration inequalities for $U$-statistics, allowing to control the second kind error of the theoretical test by a given $\beta$. 
This last step leads us to sharp upper bounds of the uniform separation rate over Sobolev balls, and an optimal bandwidth choice (depending on the regularity parameter) in order to obtain a minimax optimal test. 
In Section \ref{Multiple}, we introduce an aggregated procedure avoiding the bandwidth choice. We prove both an oracle-type inequality and sharp upper bounds for the uniform separation rate over Sobolev balls. 
Lower bounds over Sobolev spaces are obtained in Section \ref{LowerBoundSobolev}. Finally, a comparison of the permutation-based test with the theoretical test first, and then with other existing tests, is presented in a simulation study in Section \ref{NumericalSimulations}.

\section{Single HSIC-based tests of independence}
\label{Single kernel}

The aim of this section is to sharply upper bound the (non-asymptotic) uniform separation rate of HSIC-based tests over Sobolev balls which are well adapted to kernel-based tests. For this, theoretical conditions allowing to control the second kind error are first given in terms of $\HSIC_{\lambda,\mu}(f)$ and then in terms of the $\L_2$-norm of $f - f_1 \otimes f_2$. 
In this section, we consider fixed bandwidths~$(\lambda,\mu)$.

\subsection{The testing procedures} 
\label{Single-tests}

Consider the notations introduced in Sections \ref{sect:notation} and \ref{sect:HSICmeasures}. 

\paragraph{A theoretical test of independence} Since Gaussian kernels are characteristic, testing the independence between $X$ and $Y$ is equivalent to testing 
\begin{equation*}
(\mathcal{H}_0): \HSIC_{\lambda,\mu}(f)~=~0 \quad\quad \text{against} \quad\quad (\mathcal{H}_1): \HSIC_{\lambda,\mu}(f)~>~0.
\end{equation*}
The statistic $\widehat{\HSIC}_{\lambda,\mu}$ defined in Equation \eqref{esHSIC} is then a natural test statistic since it is an unbiased estimator of $\HSIC_{\lambda,\mu}(f)$. 
For a prescribed level $\alpha$ in $(0,1)$, we consider the theoretical statistical test $\Delta_{\alpha}^{\lambda,\mu}$ defined by 
\begin{equation}\label{Deltalambdamu}
\Delta_{\alpha}^{\lambda,\mu} = \mathds{1}_{\widehat{\HSIC}_{\lambda,\mu} \; > \; q_{1-\alpha}^{\lambda,\mu}}, 
\end{equation}
where $q_{1-\alpha}^{\lambda,\mu}$ denotes the $(1-\alpha)$-quantile of $\widehat{\HSIC}_{\lambda,\mu}$ under $P_{f_1\otimes f_2}$. 
We reject $(\mathcal{H}_0)$ if $\Delta_{\alpha}^{\lambda,\mu} = 1$. 
By definition of the quantile, this theoretical test is of non-asymptotic level $\alpha$, that is for all densities $f_1$ and $f_2$, 
$P_{f_1 \otimes f_2}(\Delta^{\lambda,\mu}_\alpha = 1) \leq \alpha. $
Note that the analytical computation of the quantile $q_{1-\alpha}^{\lambda,\mu}$ is not possible since its value depends on the unknown marginal densities $f_1$ and $f_2$. In practice, this quantile is approached by permutation with a Monte Carlo approximation as described in the following paragraph.

\paragraph{A permutation-based test of independence} 

Let $\ZZ_n=(X_i,Y_i)_{1\leq i\leq n}$ denote the original sample and compute the test statistic $\widehat{\HSIC}_{\lambda,\mu}\pa{\ZZ_n}$ defined by Equation~\eqref{esHSIC}. Then, let $\rperm_1, \ldots , \rperm_B$ be $B$ i.i.d. random permutations of $\{1,\ldots,n\}$, independent of $\ZZ_n$. We define for each permutation $\rperm_b$ the corresponding permuted sample $\ZZ_n^{\rperm_b}=(X_i,Y_{\rperm_b(i)})_{1\leq i\leq n}$ and compute the permuted test statistic on this new sample
$$\widehat{H}^{\star b}_{\lambda,\mu}=\widehat{\HSIC}_{\lambda,\mu}\pa{\ZZ_n^{\rperm_b}}.$$ 

\begin{sloppypar}
Under $P_{f_1\otimes f_2}$, each permuted sample $\ZZ_n^{\rperm_b}$ has the same distribution as the original sample $\ZZ_n$. Hence, the random variables $\{\widehat{H}^{\star b}_{\lambda,\mu}\}_{1 \leq b \leq B}$, have the same distribution as $\widehat{\HSIC}_{\lambda,\mu}$. We apply a trick, based on \cite[Lemma 1]{romano2005exact}, which consists in adding the original sample to the Monte Carlo sample in order to obtain a test of non-asymptotic level $\alpha$. To do so, denote 
$$\widehat{H}_{\lambda,\mu}^{\star B+1} = \widehat{\HSIC}_{\lambda,\mu}
\quad \mbox{and}\quad 
\widehat{H}_{\lambda,\mu}^{\star (1)} \leq \widehat{H}_{\lambda,\mu}^{\star (2)}\leq \ldots \leq \widehat{H}_{\lambda,\mu}^{\star (B+1)}$$
the order statistic. Then,
the permuted quantile with Monte Carlo approximation $\hat{q}_{1-\alpha}^{\lambda,\mu}$ is thus defined as
\begin{equation}\label{eq:permquant}
\hat{q}_{1-\alpha}^{\lambda,\mu} = \widehat{H}_{\lambda,\mu}^{\star (\lceil (B+1)(1-\alpha)\rceil)}.
\end{equation}
where $\lceil \cdot \rceil$ denotes the ceiling function. The permuted test with Monte Carlo approximation $\widehat{\Delta}_{\alpha}^{\lambda,\mu}$ performed in practice is then defined as 
\begin{equation}\label{approxtest}
\widehat{\Delta}^{\lambda,\mu}_\alpha = \mathds{1}_{\widehat{\HSIC}_{\lambda,\mu} \, > \, \hat{q}^{\lambda,\mu}_{1-\alpha}}.
\end{equation}
\end{sloppypar}

\begin{prop}\label{levelapprxtest}
Let $\alpha$ in $(0,1)$ and consider the permuted test with Monte Carlo approximation $\widehat{\Delta}^{\lambda,\mu}_\alpha$ defined by Equation \eqref{approxtest}. 
Then, for all $B$,
$P_{f_1\otimes f_2} \left( \widehat{\Delta}^{\lambda,\mu}_\alpha = 1 \right) \leq \alpha.$
\end{prop}

Hence, both the theoretical test $\Delta_{\alpha}^{\lambda,\mu}$ and the permuted test $\widehat{\Delta}^{\lambda,\mu}_\alpha$ are of prescribed non-asymptotic level $\alpha$. 
A comparison in terms of power is done on simulated data in Section \ref{Single_permuted_tests} justifying the restriction of the following theoretical study to the theoretical test.

\subsection{Control of the second kind error in terms of HSIC}
\label{Second-error-control}

For an arbitrarily small $\beta$ given in $(0,1)$, Lemma \ref{powerful-test} provides a first non-asymptotic condition on the alternative $f$ ensuring that the probability under $P_f$ of second kind error of the theoretical test $\Delta_{\alpha}^{\lambda,\mu}$ defined in Equation \eqref{Deltalambdamu} is at most equal to $\beta$. This condition is given for the value of $\HSIC_{\lambda,\mu}(f)$. It involves the variance of the estimator $ \widehat{\HSIC}_{\lambda,\mu}$ which is finite since this estimator is bounded. 

\begin{lemm} \label{powerful-test}
Let $\alpha$, $\beta$ in $(0,1)$ and $(X_i,Y_i)_{1\leq i\leq n}$ be an i.i.d. sample with distribution $P_f$. Consider the test statistic $\widehat{\HSIC}_{\lambda,\mu}$ defined by \eqref{esHSIC} and denote $q_{1-\alpha}^{\lambda,\mu}$ its $(1-\alpha)$-quantile under $P_{f_1\otimes f_2}$. Then $P_{f} (\widehat{\HSIC}_{\lambda,\mu} \leq q_{1-\alpha}^{\lambda,\mu}) \leq \beta$ as soon as 
\begin{equation*}
\HSIC_{\lambda,\mu}(f) \geq \displaystyle \sqrt{\frac{\Var_f(\widehat{\HSIC}_{\lambda,\mu})}{\beta}} + q_{1-\alpha}^{\lambda,\mu}.
\end{equation*}
\end{lemm}
Lemma \ref{powerful-test} gives a threshold for $\HSIC_{\lambda,\mu}(f)$ from which the dependence between $X$ and $Y$ is detectable with probability greater than $1-\beta$. 
In order to express the order of magnitude of this threshold w.r.t. $n$, $\lambda$ and $\mu$, we establish sharp upper bounds for both the variance $\Var_f(\widehat{\HSIC}_{\lambda,\mu})$ and the quantile $q_{1-\alpha}^{\lambda,\mu}$. 
Proposition \ref{LTV} gives an upper bound for the variance.

\begin{prop} \label{LTV}
Let $f$ be a density satisfying Assumption $\boldsymbol{\mathcal{A}_1}$, and $(X_i,Y_i)_{1\leq i\leq n}$ be an i.i.d. sample with distribution $P_f$. Consider the test statistic $\widehat{\HSIC}_{\lambda,\mu}$ defined by \eqref{esHSIC}. 
Then, 
\begin{equation*}
\Var_f(\widehat{\HSIC}_{\lambda,\mu}) \leq C\pa{M_f, p, q} \ac{\frac{1}{n} + \frac{1}{\lambda_1 \ldots \lambda_p \mu_1 \ldots \mu_q n^2}}.
\end{equation*}
\end{prop}

Propostion \ref{LTQ} provides an upper bound for the quantiles. It requires the following assumptions on the bandwidths $(\lambda,\mu)$:
$$\boldsymbol{\mathcal{A}_2(\alpha)}:\ \max\ac{\prod_{i=1}^p\lambda_i, \prod_{j=1}^q\mu_j} < 1
\quad\mbox{and}\quad 
n \sqrt{\lambda_1\ldots\lambda_p\mu_1\ldots\mu_q} > \log \left(\frac{1}{\alpha} \right) > 1.$$
Note that larger sample sizes allow for smaller bandwidths.

\begin{prop}\label{LTQ} 
Let $\alpha$ in $(0,1)$. Let $f$ be a density satisfying $\boldsymbol{\mathcal{A}_1}$ and $(X_i,Y_i)_{1\leq i\leq n}$ be an i.i.d. sample with distribution $P_f$.
Consider bandwidths $(\lambda,\mu)$ satisfy Assumptions $\boldsymbol{\mathcal{A}_2(\alpha)}$. Denote $\widehat{\HSIC}_{\lambda,\mu}$ the test statistic defined by \eqref{esHSIC} and $q_{1-\alpha}^{\lambda,\mu}$ its $(1-\alpha)$-quantile under $P_{f_1\otimes f_2}$. 
Then, 
\begin{equation*}
q_{1-\alpha}^{\lambda,\mu} \leq \displaystyle \frac{C\pa{M_f, p, q}}{n \sqrt{\lambda_1 \ldots \lambda_p \mu_1 \ldots \mu_q} } \log\left( \frac{1}{\alpha} \right).
\end{equation*}
\end{prop}

Combining Lemma \ref{powerful-test} with Propositions \ref{LTV} and \ref{LTQ}, Corollary \ref{powerful-testparam} provides a sufficient condition on $\HSIC_{\lambda,\mu}(f)$ depending on the bandwidths $\lambda$, $\mu$ and the sample size $n$ in order to control the second kind error rate by $\beta$. 

\begin{corr} \label{powerful-testparam}
Under the assumptions of Lemma \ref{powerful-test}, Propositions \ref{LTV} and \ref{LTQ},
one has $P_{f}(\Delta_{\alpha}^{\lambda,\mu}=0)\leq \beta$ as soon as 
\begin{equation*}
\HSIC_{\lambda,\mu}(f) > C\pa{M_f, p, q ,\beta} \left\{\displaystyle \frac{1}{\sqrt{n}} + \frac{1}{n \sqrt{\lambda_1 \ldots \lambda_p \mu_1 \ldots \mu_q}} \log \left( \frac{1}{\alpha} \right)\right\}. 
\end{equation*}
\end{corr}

Note that the right-hand side term given in Corollary \ref{powerful-testparam} depends on the unknown density $f$. However, this dependence is weak since it only involves the infinite norm of $f$ and its marginals.

\subsection{Control of the second kind error in terms of $\L_2$-norm}
\label{Second-error-control-L2}

For the sake of interpretation, and in order to upper bound the non-asymptotic uniform separation rates w.r.t. the $\L_2$-norm, we now want to express the condition on $\HSIC_{\lambda,\mu}(f)$ obtained in Corollary \ref{powerful-testparam} in terms of the $\L_2$ norm of the difference $f - f_1 \otimes f_2$.
To do so, we first give in Lemma \ref{Hpsi} a link between $\HSIC_{\lambda,\mu}(f)$ and $\norm{f- f_1\otimes f_2}_{2}^2$. 

\begin{lemm} \label{Hpsi}
Let $\psi= f- f_1\otimes f_2 $. 
The $\HSIC$ measure of $f$ associated to kernels $k_\lambda$ and $l_\mu$ and defined in Equation \eqref{formHSIC} can be written as 
\begin{equation*}
\HSIC_{\lambda,\mu}(f) = \langle \psi , \psi \ast (\varphi_\lambda \otimes \phi_\mu) \rangle_{2},
\end{equation*}
where $\varphi_\lambda$ and $\phi_\mu$ are the functions defined in Equation \eqref{varphilambdaphimu}, 
and $\langle \cdot , \cdot \rangle_{2}$ denotes the usual scalar product in the $\L_2$ space. One can easily deduce that 
\begin{equation}\label{L2HSIC}
\HSIC_{\lambda,\mu}(f) = \frac{1}{2} \bigg(\norm{\psi}_{2}^2 + \norm{\psi \ast (\varphi_\lambda \otimes \phi_\mu)}_{2}^2 - \norm{\psi - \psi \ast (\varphi_\lambda \otimes \phi_\mu)}_{2}^2 \bigg).
\end{equation}
\end{lemm}

Theorem \ref{th:powerful-testparam} gives a sufficient condition on $\norm{f- f_1\otimes f_2}_{2}^2$, for the second kind error of the test $\Delta_{\alpha}^{\lambda,\mu}$ to be upper bounded by $\beta$. 

\begin{theorem} \label{th:powerful-testparam}
Let $\alpha$, $\beta$ in $(0,1)$ and consider the test $\Delta_{\alpha}^{\lambda,\mu}$ defined by \eqref{Deltalambdamu}. 
Assume that the density $f$ satisfies $\boldsymbol{\mathcal{A}_1}$ and that the bandwidths $(\lambda,\mu)$ satisfy $\boldsymbol{\mathcal{A}_2(\alpha)}$. 
Then, $P_{f}(\Delta_{\alpha}^{\lambda,\mu}=0)\leq \beta$ as soon as 
\begin{equation}\label{CondNorm2}
\norm{\psi}_{2}^2 > \norm{\psi - \psi \ast (\varphi_\lambda \otimes \phi_\mu)}_{2}^2 + \frac{C\pa{M_f, p, q , \beta}}{n \sqrt{\lambda_1 \ldots \lambda_p \mu_1 \ldots \mu_q}} \log \pa{\frac{1}{\alpha}}.
\end{equation}
where $\psi=f-f_1\otimes f_2$, 
and $C(\cdot)$ denotes a positive constant depending only on its arguments. 
\end{theorem}

In Condition \eqref{CondNorm2} appears a compromise between a bias term, namely $ \norm{\psi - \psi \ast (\varphi_\lambda \otimes \phi_\mu)}_{2}^2 $, and a term induced by the square root of the variance of the estimator $ \widehat{\HSIC}_{\lambda,\mu}$. Note that, due to Proposition \ref{LTQ}, this variance term also controls the quantile term. Comparing the conditions on the HSIC given in Corollary \ref{powerful-testparam} and on the $\L_2$-norm $\norm{f- f_1\otimes f_2}_{2}^2$ given in Theorem \ref{th:powerful-testparam}, the meticulous reader may notice that the term in $1/\sqrt{n}$ has been removed. This suppression seems to be necessary to obtain optimal separation rates according to the literature in other testing frameworks. This derives from quite tricky computations that we point out here and that directly prove Theorem \ref{th:powerful-testparam}. 
By combining Lemmas \ref{powerful-test} and \ref{Hpsi}, direct computations lead to the condition
\begin{equation}\label{eq:etapinter}
\norm{\psi}_{2}^2 > \norm{\psi - \psi \ast (\varphi_\lambda \otimes \phi_\mu)}_{2}^2 - \norm{\psi \ast (\varphi_\lambda \otimes \phi_\mu)}_{2}^2 + 2\sqrt{\frac{\Var_f(\widehat{\HSIC}_{\lambda,\mu})}{\beta}} + 2q_{1-\alpha}^{\lambda,\mu}.
\end{equation}

If one directly considers the upper bound of the variance $\Var_f(\widehat{\HSIC}_{\lambda,\mu})$ given in Proposition \ref{LTV}, one would get the unwanted $1/\sqrt{n}$ term. 
The idea is to take advantage of the negative term $- \norm{\psi \ast (\varphi_\lambda \otimes \phi_\mu)}_{2}^2$ to compensate such term. To do so, we need a more refined control of the variance given in the technical Proposition \ref{prop:finemajvar}. 

\begin{prop} \label{prop:finemajvar}
Let $f$ be a density satisfying Assumption $\boldsymbol{\mathcal{A}_1}$, and $(X_i,Y_i)_{1\leq i\leq n}$ be an i.i.d. sample with distribution $P_f$. 
Consider the test statistic $\widehat{\HSIC}_{\lambda,\mu}$ defined by \eqref{esHSIC}. 
Then, 
\begin{equation*}
\Var_f(\widehat{\HSIC}_{\lambda,\mu}) \leq C(M_f)\frac{\norm{\psi \ast (\varphi_\lambda \otimes \phi_\mu)}_{2}^2}{n} + \frac{C\pa{M_f, p, q}}{n^2\lambda_1 \ldots \lambda_p \mu_1 \ldots \mu_q}.\end{equation*}
\end{prop}

Finally, using standard inequalities such as $\sqrt{a+b}\leq \sqrt{a} + \sqrt{b}$ or such as $2\sqrt{ab}\leq c a + b/c$ for all positive $a,\ b$ and $c$, one can prove that
$$2\sqrt{\frac{\Var_f(\widehat{\HSIC}_{\lambda,\mu})}{\beta}} \leq \norm{\psi \ast (\varphi_\lambda \otimes \phi_\mu)}_{2}^2 + \frac{C\pa{M_f,\beta}}{n} + \frac{C\pa{M_f, p, q, \beta}}{n\sqrt{\lambda_1 \ldots \lambda_p \mu_1 \ldots \mu_q}},$$ 
which leads to Theorem \ref{th:powerful-testparam} when combined with Equation \eqref{eq:etapinter} and Proposition \ref{LTQ}.
Notice that such trick is already present in \cite{fromont2013two}.

\subsection{Uniform separation rate over Sobolev balls}
\label{Uniform-separation-rate}

The bias term in Theorem \ref{th:powerful-testparam} comes from the fact that we do not estimate $\norm{f-f_1 \otimes f_2}_{2}^2$ but $\HSIC_{\lambda,\mu}(f)$. In order to have a control of the bias term w.r.t $\lambda$ and $\mu$, we assume that $f-f_1 \otimes f_2$ belongs to some class of regular functions. \\

The Sobolev ball $\mathcal{S}^{\delta}_d (R)$ in dimension $d$ in $\mathbb{N}^*$, with regularity parameter $\delta > 0$ and radius $R > 0$, is defined by
\begin{equation}\label{eq:Sobolevball}
\mathcal{S}^{\delta}_d (R) = \ac{ s: \mathbb{R}^d \rightarrow \mathbb{R} \ ;\ s \in \L_1 (\mathbb{R}^d) \cap \L_2 (\mathbb{R}^d), \int_{\mathbb{R}^d} \norm{u}^{2 \delta} \abs{\hat{s} (u)}^2 \mathrm{d}u \leq (2 \pi)^{d} R^2 },
\end{equation}
where $\norm{\cdot}$ denotes the Euclidean norm associated to the usual scalar product $\langle\cdot, \cdot \rangle$ in $\R^d$, and $\hat{s}$ denotes the Fourier transform of $s$, defined on $\R^d$ by $\hat{s} (u) = \int_{\mathbb{R}^d} s(x) e^{i \langle x,u \rangle} \; \mathrm{d}x$. 
Lemma \ref{TBSB} gives an upper bound for the bias term in the case where $f-f_1 \otimes f_2 $ belongs to particular Sobolev balls.

\begin{lemm} \label{TBSB}
Assume that $ \psi = f-f_1 \otimes f_2$ belongs to the Sobolev ball $\mathcal{S}^{\delta}_{p+q} (R)$ with positive parameters $\delta$ and  $R$, defined in \eqref{eq:Sobolevball}. Let $\varphi_\lambda$ and $\phi_\mu$ be the functions defined in \eqref{varphilambdaphimu}. Then, there exists $T_\delta$ in $[0,1]$ such that 
$$\norm{\psi- \psi\ast (\varphi_\lambda \otimes \phi_\mu)}_2^2 \leq \pa{1-e^{-T_\delta^2/2}} \norm{\psi}_2^2 + C(p,q,\delta,R) \left[ \sum_{i=1}^p \lambda_i^{2 \delta} +  \sum_{j=1}^q \mu_j^{2 \delta} \right].$$ 
Moreover, if $\delta$ belongs to $(0,2]$, then $T_\delta=0$ and the term with $\norm{\psi}_2$ vanishes. 
\end{lemm}


In the following, we study optimality over $S_{p+q}^{\delta}(R,R')$ defined by 
\begin{equation}\label{eq:Sobolevinf}
S_{p+q}^{\delta}(R,R') = S_{p+q}^\delta(R) \cap \ac{f ; \max\ac{\norm{f}_{\infty}, \norm{f_1}_{\infty}, \norm{f_2}_{\infty}}\leq R'}.
\end{equation} 
One can deduce from Theorem \ref{th:powerful-testparam} upper bounds for the uniform separation rates, defined in \eqref{seprate}, of the test $\Delta_{\alpha}^{\lambda,\mu}$ over Sobolev balls.

\begin{theorem}
\label{ThresholdSobolev}
Let $\alpha$, $\beta$ in $(0,1)$, and positive parameters $\delta$, $R$ and $R'$. 
Consider bandwidths $(\lambda,\mu)$ satisfying Assumptions $\boldsymbol{\mathcal{A}_2(\alpha)}$ and denote $\Delta_{\alpha}^{\lambda,\mu}$ the test defined by \eqref{Deltalambdamu}. 
Then, the uniform separation rate defined in \eqref{seprate} of the test $\Delta_{\alpha}^{\lambda,\mu}$ over the Sobolev ball $\mathcal{S}^{\delta}_{p+q} (R,R')$ defined in Equation \eqref{eq:Sobolevinf} can be upper bounded as follows
\begin{multline}\label{eq:majsobolev}
\cro{\rho \left( \Delta_{\alpha}^{\lambda,\mu} , \mathcal{S}^{\delta}_{p+q} (R,R') , \beta \right)}^2 
\leq \displaystyle C(p,q,\delta,R) \left[ \sum_{i=1}^p \lambda_i^{2 \delta} + \sum_{j=1}^q \mu_j^{2 \delta} \right] \\
+ \frac{C\pa{R', p, q, \beta}}{n\sqrt{\lambda_1 \ldots \lambda_p \mu_1 \ldots \mu_q}} \log\left( \frac{1}{\alpha} \right),
\end{multline}
where $C(\cdot)$ denote positive constants depending only on their arguments.
\end{theorem}

One can now determine optimal bandwidths $(\lambda^*, \mu^*)$ which minimize the right-hand side of Equation \eqref{eq:majsobolev}. 
To do so, the idea is to find for which $(\lambda, \mu)$ both terms in the right-hand side of \eqref{eq:majsobolev} are of the same order w.r.t. $n$. 
We also provide an upper bound for the uniform separation rate of the optimized test $\Delta_{\alpha}^{\lambda^*,\mu^*} $ over Sobolev balls.

\begin{corr}
\label{cor:ThresholdSobolev} Let $\alpha$ in $(0,1/e)$, $\beta$ in $(0,1)$, and $\delta,R,R'>0$. Define for all $i$ in $\{1,\dots,p\}$ and for all $j$ in $\{1,\dots,q\}$, 
$$ \lambda_i^* = \mu_j^* = n^{- 2 / (4 \delta + p+q)}.$$
If $n  > \pa{\log(1/\alpha)}^{1+(p+q)/(4\delta)}$, then, $(\lambda^*,\mu^*)$ satisfy $\boldsymbol{\mathcal{A}_2(\alpha)}$ and the uniform separation rate of the optimized test $ \Delta_{\alpha}^{\lambda^*,\mu^*}$ over the Sobolev ball $\mathcal{S}^{\delta}_{p+q} (R,R')$ is controlled as follows 
\begin{equation*}
\rho \left( \Delta_{\alpha}^{\lambda^*,\mu^*} , \mathcal{S}^{\delta}_{p+q} (R,R') , \beta \right) 
\leq \displaystyle C\pa{p, q, \alpha, \beta, \delta,R,R'} n^{- 2 \delta/(4 \delta + p+q)}.
\end{equation*}
\end{corr}

Note that, in the definition of the Sobolev ball $\mathcal{S}^{\delta}_{p+q} (R,R')$, we have the same regularity parameter $\delta >0$ for all directions in $ \R^{p+q}$. This corresponds to an isotropic regularity condition. Similar results over anisotropic Nikol'skii regularities are obtained in Appendix \ref{section:supplNikol} in the supplementary material. 

Moreover, the test $\Delta_{\alpha}^{\lambda^*,\mu^*}$ with the optimized bandwidths depends on the regularity parameter $\delta$ and cannot be computed in practice. In the next section, for the purpose of adaptivity, we build an aggregated testing procedure taking into account a collection of bandwidths. In particular, this avoids the delicate choice of arbitrary bandwidths. We then prove an oracle-type inequality and show that the uniform separation rate of this aggregated procedure is of the same order as the smallest uniform separation rate of the tests in the chosen collection, up to a logarithmic term. 

Finally, note that subsequently, \cite{kim2020minimax} generalize Theorem \ref{ThresholdSobolev} to the permuted tests. 
However, they obtain a polynomial dependence in $\alpha$, that is $1/\sqrt{\alpha}$ instead of $\log(1/\alpha)$, which leads to minimax optimal tests for an optimized bandwidth choice, as in Corollary \ref{ThresholdSobolev}. 
Yet, the dependence in $\alpha$ is not sharp enough to provide adaptive tests by aggregating.

\section{Aggregated HSIC-based test of independence}
\label{Multiple}

In Section \ref{Single kernel}, we consider single tests based on Gaussian kernels associated to given bandwidths $(\lambda,\mu)$. However, there is as yet no justified method to choose $\lambda$ and $\mu$ with theoretical guarantees. In many cases, authors choose these parameters w.r.t the available data $(X_i,Y_i)_{1 \leq i \leq n}$ by taking for example $\lambda$ (resp. $\mu$) as the empirical median (see \cite{Zhang2011uai}) or the empirical mean (see, e.g., \cite{de2017sensitivity,marrel2020statistical}) of $(\norm{X_i-X_j})_{1\leq i<j\leq n}$ (resp. $(\norm{Y_i-Y_j})_{1\leq i<j\leq n}$). To avoid this arbitrary choice, we consider in this section an aggregated testing procedure combining a collection of single tests based on different bandwithds. 

\subsection{The aggregated testing procedure}
\label{Agreg-proc}

Consider now a finite or countable collection $\W\subset(0,+\infty)^p\times (0,+\infty)^q$ of bandwidths $(\lambda,\mu) $ and a collection of positive weights $\ac{\omega_{\lambda,\mu}}_{(\lambda,\mu) \in \W}$ such that $\sum_{(\lambda,\mu) \in \W} e^{- \omega_{\lambda,\mu}} \leq 1$.

For a given $\alpha$ in $(0,1)$, the aggregated test rejects $(\mathcal{H}_0)$ if there is at least one $(\lambda,\mu)$ in $\W$ such that the corresponding single test with corrected level $u_\alpha \exp(-\omega_{\lambda,\mu})$ rejects $(\mathcal{H}_0)$, that is 
\begin{equation*}
\exists (\lambda,\mu) \in \W \ ;\ \widehat{\HSIC}_{\lambda,\mu} \; > \; q_{1- u_\alpha e^{- \omega_{\lambda,\mu}}}^{\lambda,\mu},
\end{equation*}
where $u_\alpha$ is the less conservative value such that the  aggregated test is of level $\alpha$. More precisely, this level correction is defined by 
\begin{equation}\label{u_alpha}
u_\alpha = \sup \ac{ u > 0\ ; \ P_{f_1\otimes f_2} \pa{ \sup_{(\lambda,\mu) \in \W} \ac{ \widehat{\HSIC}_{\lambda,\mu} - q_{1- u e^{- \omega_{\lambda,\mu}}}^{\lambda,\mu}} > 0 } \leq \alpha }.
\end{equation}

We should mention here that the supremum in Equation \eqref{u_alpha} exists since the function 
$$u \mapsto P_{f_1\otimes f_2} ( \sup_{(\lambda,\mu) \in \W} \{ \widehat{\HSIC}_{\lambda,\mu} - q_{1- u e^{- \omega_{\lambda,\mu}}}^{\lambda,\mu} \} >0)$$
is well defined for $u$ in the interval $(0,\inf_{(\lambda,\mu) \in \W}\ac{\exp(\omega_{\lambda,\mu})})$, non-decreasing, and converges to 0 and 1 respectively at the boundaries of this interval. 
Moreover, we can choose weights such that $\sum_{(\lambda,\mu) \in \W} e^{- \omega_{\lambda,\mu}} = 1$. Yet, in practice, it  just changes the value of $u_\alpha$ and leads to the same test.   

The (theoretical) aggregated test $\Delta_{\alpha}$ is then defined by
\begin{equation}\label{agregtest}
\Delta_{\alpha} = \mathds{1}\!\!\!\!\!_{\underset{(\lambda,\mu) \in \W}{\sup} \ac{ \widehat{\HSIC}_{\lambda,\mu} - q_{1- u_\alpha e^{- \omega_{\lambda,\mu}}}^{\lambda,\mu}} > 0},
\end{equation}
and rejects $(\mathcal{H}_0)$ if $\Delta_{\alpha} = 1$. By definition of $u_\alpha$, the test $\Delta_{\alpha}$ is of level $\alpha$. \\

For computational limitations, the collection $\W$  is finite in practice.
Moreover, as for the quantile, the correction $u_\alpha$ of the level is not analytically computable since it depends on the unknown marginal densities $f_1$ and $f_2$. 
In practice, it can also be approached by a permutation method with Monte Carlo approximation, as done in \cite{albert2015tests}. 
More precisely, consider the notations of Section \ref{Single-tests}. 
First, generate $B_1$ independent and uniformly distributed random permutations of $\{1,\ldots,n\}$, denoted $\rperm_1, \ldots , \rperm_{B_1}$, independent of $\ZZ_n$ and compute for each $(\lambda,\mu)$ in $\W$ and each $u>0$ the permuted quantile with Monte Carlo approximation $\hat{q}^{\lambda,\mu}_{1-u e^{-\omega_{\lambda,\mu}}}$ as defined in \eqref{eq:permquant}. 
Second, in order to estimate the probabilities under $P_{f_1\otimes f_2}$ in \eqref{u_alpha}, generate $B_2$ independent and uniformly distributed random permutations of $\{1,\ldots,n\}$, denoted $\rpermb_1, \ldots , \rpermb_{B_2}$, independent of $\ZZ_n$ and of $\rperm_1, \ldots , \rperm_{B_1}$. Denote for all permutation $\rpermb_b$, the corresponding permuted statistic
$$\widehat{H}^{\rpermb_b}_{\lambda,\mu}=\widehat{\HSIC}_{\lambda,\mu}\pa{\ZZ_n^{\rpermb_b}}$$
Then, the correction $u_\alpha$ is approached by Monte Carlo as follows:
\begin{equation}\label{Ualphaestimation}
\hat{u}_\alpha= \sup\ac{u>0\ ;\ \frac{1}{B_2} \sum_{b=1}^{B_2} \mathds{1}_{\underset{(\lambda,\mu) \in \W}{\max} \ac{\widehat{H}^{\rpermb_b}_{\lambda,\mu} - \hat{q}^{\lambda,\mu}_{1-u e^{-\omega_{\lambda,\mu}}}}>0}\leq \alpha}. 
\end{equation}
In the end, the permuted aggregated test $\hat{\Delta}_{\alpha}$ with Monte Carlo approximation is defined by 
\begin{equation}\label{agregtestperm}
\hat{\Delta}_{\alpha} = \mathds{1}\!\!\!\!\!_{\underset{(\lambda,\mu) \in \W}{\max} \ac{ \widehat{\HSIC}_{\lambda,\mu} - \hat{q}^{\lambda,\mu}_{1-\hat{u}_\alpha e^{-\omega_{\lambda,\mu}}} } > 0}.
\end{equation}

As for the single tests, a comparison in terms of power is done on simulated data in Section \ref{sect:comparTheoPerm} justifying the restriction of the following theoretical study to the theoretical aggregated test.

\subsection{Oracle-type conditions for the uniform separation rate over Sobolev balls}
\label{Agreg-oracle}

As a reminder, our goal is to construct a testing procedure with a uniform separation rate as small as possible and whose implementation does not require any information about the regularity of the difference $f-f_1\otimes f_2$. 

The main advantage of the aggregated procedure is that its second kind error is upper bounded by the smallest error of the single tests (with corrected levels) in the collection. 
The main argument is highlighted in Lemma~\ref{prop5}.

\begin{lemm}\label{prop5}
Let $\alpha$, $\beta$ in $(0,1)$, and consider the aggregated test $\Delta_{\alpha}$ defined in Equation \eqref{agregtest}. Then, $u_\alpha\geq \alpha$ and
$$P_{f} \pa{\Delta_\alpha = 0}
\leq 
\inf_{(\lambda,\mu) \in \W} \ac{P_{f} \pa{\Delta_{\alpha e^{- \omega_{\lambda,\mu}}}^{\lambda,\mu}=0}}.
$$
\end{lemm}

According to Lemma \ref{prop5}, if there exists at least one single test $\Delta_{\alpha e^{- \omega_{\lambda,\mu}}}^{\lambda,\mu}$ with a probability of second kind error at most equal to $\beta$, then the same control holds for the aggregated test $\Delta_{\alpha}$. 
Theorem \ref{theorem3Sobol} gives an oracle-type inequality for the uniform separation rate of the aggregated testing procedure $\Delta_{\alpha}$, showing the interest of this procedure.

\begin{theorem}\label{theorem3Sobol}
Let $\alpha, \beta$ in $(0,1)$. Consider a finite or countable collection $\W\subset (0,+\infty)^p\times (0,+\infty)^q$ of bandwidths $(\lambda,\mu)$ and a collection of positive weights $\ac{\omega_{\lambda,\mu}}_{(\lambda,\mu) \in \W}$ such that $\sum_{(\lambda,\mu) \in \W} e^{- \omega_{\lambda,\mu}} \leq 1$ and such that 
all $(\lambda,\mu)$ in $\W$ verifies Assumption $\boldsymbol{\mathcal{A}_2(\alpha e^{-\omega_{\lambda,\mu}})}$. 
Then, the uniform separation rate over Sobolev balls $\mathcal{S}^{\delta}_{p+q} (R,R')$ with positive parameters $\delta$, $R$ and $R'$ of the aggregated test $\Delta_{\alpha}$ defined in Equation \eqref{agregtest} can be upper bounded as follows
\begin{multline*}
\cro{\rho \left( \Delta_{\alpha} , \mathcal{S}^{\delta}_{p+q} (R,R') , \beta \right) }^2\leq \displaystyle C\pa{p, q, \beta , \delta, R,R'} \times \\ \inf_{(\lambda,\mu) \in \W} \left\{ 
\left[ \sum_{i=1}^p \lambda_i^{2 \delta} + \sum_{j=1}^q \mu_j^{2 \delta} \right] +  \frac{1}{n\sqrt{\lambda_1 \ldots \lambda_p \mu_1 \ldots \mu_q}} \left( \log\left( \frac{1}{\alpha} \right) + \omega_{\lambda,\mu} \right) \right\},
\end{multline*}
where $C(\cdot)$ is a positive constant depending only on its arguments.
\end{theorem}

Theorem \ref{theorem3Sobol} can be interpreted as an oracle-type condition for the uniform separation rate of the test $\Delta_{\alpha}$. Indeed, without knowing the regularity of $f- f_1\otimes f_2 $, we prove that the uniform separation rate of $\Delta_{\alpha}$ is of the same order as the smallest uniform separation rate of the single tests corresponding to  bandwidths $(\lambda,\mu)$ in $\W$, up to an additional term $\omega_{\lambda,\mu}$ due to the correction of the individual levels.

\subsection{Uniform separation rate over Sobolev balls}
\label{USR}

In this section, we consider the aggregated test for a particular choice of bandwidths collection $\W$ defined by 
\begin{equation}\label{eq:defLambdaU}
\W =  \bigg\{ (2^{-m} \boldsymbol{1}_{p+q}, m \in \ac{1, \ldots, M_n^{p,q}}\bigg\},
\end{equation}
where $\boldsymbol{1}_{p+q} = (1,1,\ldots,1) \in \R^{p+q}$ and, denoting $\lfloor \cdot \rfloor$ the floor function, 
$$ M_n^{p,q} = \left\lfloor {\log_2\pa{\cro{\frac n {\log(n)}}^{\frac 2{p+q}}} }\right\rfloor.$$
In addition, we associate to every $(\lambda, \mu) = 2^{-m} \boldsymbol{1}_{p+q}$ in $\W$ the positive weight
\begin{equation}\label{eq:weigths}
\omega_{\lambda, \mu} = 2  \log\pa{m  \times \frac{\pi}{\sqrt{6}} },
\end{equation} 
so that $\sum_{(\lambda,\mu) \in \W} e^{- \omega_{\lambda,\mu}} \leq 1$. 
Corollary \ref{corr1Sobol} justifies that this particular choice of bandwidths collection and associated weights is well adapted to Sobolev regularities. 

 
\begin{corr}\label{corr1Sobol}
Let $\alpha, \beta$ in $(0,1)$. 
Consider the aggregated test $\Delta_{\alpha}$ defined in \eqref{agregtest}, with the particular choice of the collection $\W$ and  the weights $\left( \omega_{\lambda, \mu} \right)_{(\lambda, \mu) \in \W}$ defined in \eqref{eq:defLambdaU} and \eqref{eq:weigths}. Assume that $\log \log (n) > 1$.  Under the assumptions of Theorem \ref{theorem3Sobol},  for any $\delta,R,R'>0$, there exists a positive constant $C(p,q,\alpha,\delta)$ such that for all $n \geq C(p,q,\alpha,\delta)$, the uniform separation rate over the Sobolev ball $\mathcal{S}^{\delta}_{p+q} (R,R')$ of $\Delta_{\alpha}$ can be upper bounded as follows:
\begin{equation}\label{eq:adaptsobolaggreg}
\rho \left( \Delta_{\alpha} , \mathcal{S}^{\delta}_{p+q} (R,R') , \beta \right) \leq \displaystyle C\pa{p, q, \alpha, \beta, \delta,R,R'} \left( \frac{\log \log (n)}{n} \right)^{2 \delta/( 4 \delta + p+q)}.
\end{equation}
\end{corr}

According to Corollary \ref{corr1Sobol}, the uniform separation rate of the aggregated procedure over Sobolev balls is of the same order as the one of the optimized test $\Delta_{\alpha}^{\lambda^*,\mu^*}$ (given in Corollary \ref{cor:ThresholdSobolev}), up to a $\log \log(n)$ factor. 
Note that this logarithmic loss is usually the price to pay for aggregated tests (see, e.g., \cite{spokoiny1996adaptive,ingster2000adaptive}).
Similar results over Nikol'skii-Besov spaces are also obtained in the supplementary material.

\section{Lower bound for the non-asymptotic minimax rate over Sobolev balls}
\label{LowerBoundSobolev}

In this section, we 
present a general method based on a Bayesian approach to lower bound the non-asymptotic minimax rate of testing as defined in \eqref{minimaxrate}. The general idea of this method is due to \cite{ingster1993asymptotically} and relies on Lemma \ref{Lem:MethodLowerBound}.

\begin{lemm}\label{Lem:MethodLowerBound} 
Let $\alpha, \beta,\eta$ in $(0,1)$ such that $\alpha+\beta+\eta<1$. 
Let $\regSp_\delta$ denote some regularity space, and recall that for all positive $\rho$, the set $\mathcal{F}_{\rho} (\regSp_\delta)$ is defined by 
$$\mathcal{F}_{\rho} (\regSp_\delta) = \{ f\ ;\ f - f_1 \otimes f_2 \in \regSp_\delta,\ \norm{f - f_1 \otimes f_2}_{2} \geq \rho \}.$$ 
Let us denote 
$$\beta \big[ \mathcal{F}_{\rho} (\regSp_\delta) \big] = \underset{\Delta_\alpha}{\inf} \underset{f \in \mathcal{F}_{\rho} (\regSp_\delta)}{\sup} P_f \left( \Delta_\alpha = 0 \right),$$ 
where the infimum is taken over all $\alpha$-level tests of $(\mathcal{H}_0)$ against $(\mathcal{H}_1)$. \\

Let $\rho_* > 0$ and consider a probability measure $\nu_{\rho_*}$ defined on the set of densities in $\L_2(\R^p\times \R^q)$ such that $\nu_{\rho_*}(\mathcal{F}_{\rho_*} (\regSp_\delta)) \geq 1-\eta$. 
Define the associated probability measure $P_{\nu_{\rho_*}}$ for all measurable set $A$ in $\R^{n(p+q)}$ by
$$P_{\nu_{\rho_*}}(A) = \int_{\L_2(\R^p\times \R^q)} P_f(A) \, \mathrm{d} \nu_{\rho_*}(f).$$

Assume there exists a density $f_0$ that satisfies $(\mathcal{H}_0)$ such that the probability measure $P_{\nu_{\rho_*}}$ is absolutely continuous w.r.t. $P_{f_0}$ and verifies
\begin{equation}\label{Ineq:Lnu}
\esps{P_{f_0}}{ L^2_{\nu_{\rho_*}} \left( \mathbb{Z}_n \right) } < 1 + 4 (1 - \alpha - \beta-\eta)^2,
\end{equation}
where the likelihood ratio $L_{\nu_{\rho_*}}$ is defined by 
$L_{\nu_{\rho_*}} = \mathrm{d} P_{\nu_{\rho_*}}/\mathrm{d} P_{f_0}.$
Then, for all $\rho \leq \rho_*$ we have that 
$ \beta \big[ \mathcal{F}_{\rho} (\regSp_\delta) \big] > \beta.$
It follows that 
$$\rho \left(\regSp_\delta, \alpha, \beta \right) = \inf_{\Delta_\alpha}\rho \left( \Delta_\alpha, \regSp_\delta, \beta \right) \geq \rho_*.$$
\end{lemm} 

We aim at proving that 
\begin{equation*}\rho_n^* = C n^{-2\delta/(4\delta+p+q)}\end{equation*}
is a lower bound for the non-asymptotic minimax rate of testing, defined in \eqref{minimaxrate}, over Sobolev balls $\mathcal{S}^{\delta}_{p+q} (R,R')$, for some positive constant $C$, that is, $\rho \left(\mathcal{S}^{\delta}_{p+q} (R,R'), \alpha, \beta \right)\geq \rho_n^*$. 
According to Lemma \ref{Lem:MethodLowerBound}, it is sufficient to find a probability distribution $\nu_{\rho_n^*}$ such that  $\nu_{\rho_n^*}(\mathcal{F}_{\rho_n^*} (\mathcal{S}^{\delta}_{p+q} (R,R'))) \geq 1-\eta$ and such that Condition \eqref{Ineq:Lnu} holds. \\

To do so, we generalize the construction of \cite{butucea2007goodness} to our multidimensional framework. The idea is to construct a finite set of alternatives $(f_\theta)_\theta$ by perturbing the uniform density on $[0,1]^{p}\times [0,1]^{q}$, and define $\nu_{\rho_n^*}$ as a uniform mixture of these alternatives. For this, consider the function $G$ defined for all $t$ in $\R$ by 
\begin{equation}\label{Eq:functionG}
G(t) = e^{-1/[1 - (4t + 3)^2]} \mathds{1}_{(-1,-1/2)} (t) - e^{-1/[1 - (4t + 1)^2]} \mathds{1}_{(-1/2,0)} (t).
\end{equation}
One may notice that $G$ is continuous, with support in $[-1,0]$ and that $\int_{\R}G(t)\mathrm{d}t=0$. The function $G$ together with its Fourier transform has valuable properties for our study. 

Let $h_n$ be in $(0,1]$ to be specified later such that $M_n := 1/h_n$ in an integer. 
Denote $\I_{n,p,q}=\ac{1,\ldots, M_n}^p \times \ac{1,\ldots, M_n}^q$. 
For all $\theta = (\theta_{(j,l)})_{(j,l)\in \I_{n,p,q}}$ in $\ac{-1, 1}^{M_n^{p+q}}$, define for all $(x,y)$ in $\R^p\times \R^q$, 
\begin{multline}\label{ftheta}
f_\theta (x,y) = 
 \mathds{1}_{[0,1]^{p+q}}(x,y)  \\
+C_0 h_n^{\delta + (p+q)} \sum_{(j,l)\in \I_{n,p,q}} \theta_{(j,l)} \prod_{r = 1}^p G_{h_n} (x_r - j_r h_n) \prod_{s = 1}^q G_{h_n}(y_s - l_s h_n),
\end{multline}
where for all $h>0$, $G_h(\cdot) = \left( 1/h \right) G(\cdot/h)$ and $C_0$ is a constant depending on $(p,q,\delta,R,R',\eta)$ that will be specified later. 
One may notice that for all $\theta$, the alternative $f_\theta$ is supported in $[0,1]^{p+q}$. Moreover, since the integral of $G$ over $\R$ equals $0$, the marginals $f_{\theta,1}$ and $f_{\theta,2}$ of $f_\theta$ are respectively the uniform densities on $[0,1]^p$ and $[0,1]^q$. 
Lemmas \ref{prop:alternftheta} and \ref{lemma:alternfthetaSobolev} justify the choice of these alternatives. 

\begin{lemm}\label{prop:alternftheta}
Let $\delta>0$, $R>0$ and $R'\geq 1$. Consider $h_n$ in $(0,1]$ such that $M_n := 1/h_n$ is an integer.
Then, for all $\theta = (\theta_{(j,l)})_{(j,l)\in \I_{n,p,q}}$ in $\ac{-1, 1}^{M_n^{p+q}}$, the function $f_\theta$ defined in Equation \eqref{ftheta} satisfies the following properties. 
\begin{enumerate}
\item\label{borninfproba} If $C_0 \leq \min\{1,R'-1\}e^{p+q}$, then the function $f_\theta$ is a density function and  $$\max\{\norm{f_\theta}_\infty,\norm{f_{\theta,1}}_\infty,\norm{f_{\theta,2}}_\infty\}\leq R'.$$ 
\item\label{borninfdistH0} The function $f_\theta$ is such that $\norm{f_{\theta} - f_{\theta,1} \otimes f_{\theta,2}}_2 = C_0 \norm{G}_2^{p+q} h_n^{\delta}$.
\end{enumerate}
\end{lemm}

Let us now consider a uniform mixture $\nu_{\rho_n^*}$ of the alternatives $(f_\theta)$, for $\theta$ in $\ac{-1, 1}^{M_n^{p+q}}$. Note that this is equivalent to considering a random alternative $f_\Theta$ where $\Theta = (\Theta_{(j,l)})_{(j,l)\in I_{n,p,q}}$ with i.i.d. Rademacher components $\Theta_{(j,l)}$.
The aim of Lemma \ref{lemma:alternfthetaSobolev} is to prove that, for a well chosen constant $C_0$, the random function 
$ f_\Theta - f_{\Theta,1} \otimes f_{\Theta,2} $ belongs to the Sobolev ball $\mathcal{S}^\delta_{p+q} (R,R')$ with high probability.

\begin{lemm}\label{lemma:alternfthetaSobolev}
Let $\delta>0$, $R>0$ and $R'\geq 1$. Let $\Theta $ be the random vector $ \Theta= (\Theta_{(j,l)})_{(j,l)\in I_{n,p,q}}$ with i.i.d. Rademacher components $\Theta_{(j,l)}$. Consider $f_\Theta$ defined by \eqref{ftheta}, where the vector $\theta$ is replaced by the random vector $\Theta$. 
Then, there exists a positive constant $C(p,q,\delta,\eta)$ such that, if $C_0^2 \leq (2\pi)^{p+q} R^2 / [2C(p,q,\delta,\eta)]$, we have that
$$\proba{f_\Theta - f_{\Theta,1} \otimes f_{\Theta,2} \in \mathcal{S}^\delta_{p+q} (R)}\geq 1-\eta. $$
\end{lemm}

Following Lemma \ref{Lem:MethodLowerBound}, let $P_{\nu_{\rho_n^*}}$ be the probability measure defined for all measurable set $A$ in $\R^{n(p+q)}$ by
\begin{equation}\label{eq:Pnu}
P_{\nu_{\rho_n^*}} (A) = \int_{\{-1,1\}^{M_n^{p+q}}} P_{f_\theta}(A) \pi (\mathrm{d} \theta) = \frac{1}{2^{M_n^{p+q}}} \sum_{\theta\in\{-1,1\}^{M_n^{p+q}}}P_{f_\theta}(A),
\end{equation}
where $\pi$ is the distribution of a $(M_n^{p+q})$-sample of i.i.d. Rademacher random variables. Proposition \ref{prop:rapportvrais} justifies the use of these alternatives and this probability measure to prove the lower bound. 

\begin{prop}\label{prop:rapportvrais}
Let $\alpha, \beta, \eta$ in $(0,1)$ such that $\alpha+\beta+\eta<1$, and let $\delta,R>0$ and $R'\geq 1$. 
Denote $f_0$ the uniform density on $[0,1]^{p+q}$. 
Assume that $C_0=C_0(p,q,\delta,R,R',\eta)$ satisfies the assumptions of Lemmas \ref{prop:alternftheta} and \ref{lemma:alternfthetaSobolev}.
There exists some positive constant $C(p,q,\alpha,\beta,\delta,R,R',\eta)$ such that, if we set 
\begin{equation}\label{eq:defhn}
M_n = \left\lceil \frac{1}{C(p,q,\alpha,\beta,\delta,R,R',\eta) n^{-2/(4\delta+p+q)}}\right\rceil \quad \mbox{and}\quad h_n = \frac{1}{M_n},
\end{equation}
we have $\nu_{\rho_n^*}(\mathcal{F}_{\rho_n^*} (\mathcal{S}^{\delta}_{p+q} (R,R'))) \geq 1-\eta$. 
Furthermore, if we define $P_{\nu_{\rho_n^*}}$ by Equations \eqref{ftheta} and \eqref{eq:Pnu}, then we have, for $n$ large enough,
$$\esps{P_{f_0}}{ \pa{ \frac{\mathrm{d} P_{\nu_{\rho_n^*}}}{\mathrm{d} P_{f_0}} (\mathbb{Z}_n) }^2 } < 1 + 4(1 - \alpha-\eta - \beta)^2.$$ 
\end{prop}

Finally, combining Lemmas \ref{Lem:MethodLowerBound}, \ref{prop:alternftheta} and \ref{lemma:alternfthetaSobolev} with Proposition \ref{prop:rapportvrais} leads to a lower bound for the non-asymptotic minimax  rate of testing in Theorem \ref{Th:Lowerbound}. 

\begin{theorem}\label{Th:Lowerbound}
Consider $\alpha,\beta,\eta$ in $(0,1)$ such that $\alpha+\beta+\eta<1$. 
Let $\delta>0$, $R>0$ and $R'\geq 1$. 
Then, there exists a positive constant $C(p,q,\alpha,\beta,\delta,R,R',\eta)$ such that, for $n$ large enough,
$$\rho \left(\mathcal{S}^{\delta}_{p+q} (R,R'), \alpha, \beta \right) \geq C(p,q,\alpha,\beta,\delta,R,R',\eta)\ n^{-2\delta/(4\delta+p+q)}.$$
\end{theorem}

Theorem \ref{Th:Lowerbound} proves that the optimized test $\Delta_{\alpha}^{\lambda^*,\mu^*}$ introduced in Corollary \ref{cor:ThresholdSobolev} is optimal in the minimax sense over Sobolev balls since the upper and lower bounds coincide up to constants. 
Moreover, the aggregated testing procedure defined in Corollary \ref{corr1Sobol} is optimal up to a $\log \log(n)$ term over Sobolev balls. 
Note that this logarithmic term obtained in the upper bound \eqref{eq:adaptsobolaggreg} is sometimes unavoidable for adaptivity (c.f. \cite{ingster2000adaptive} for the test of uniformity on $[0,1]$). It seems reasonable to conjecture that it is also the case for independence testing. 
Hence, since the aggregated testing procedure does not depend on the prior knowledge of the regularity parameter $\delta$, we may conclude that it is adaptive.

\section{Numerical simulations}
\label{NumericalSimulations}

In this section, numerical simulations are performed in order to study the practical validity of our testing procedures. 
More precisely, we first compare the \emph{theoretical} aggregated test $\Delta_\alpha$ defined in \eqref{agregtest} (studied in theory) and the \emph{permuted} aggregated test $\hat{\Delta}_\alpha$ defined in \eqref{agregtestperm} (applied in practice) in terms of power. 
A similar verification for the single tests, together with a comparison of the power for different bandwidth collections and weights choices, are also carried out in Appendix \ref{Single_permuted_tests} in the supplementary material.
Then, we compare the \emph{permuted} aggregated test with existing nonparametric independence tests on simulated data. 

\subsection{Comparison between the theoretical and the permuted aggregated tests}
\label{sect:comparTheoPerm}

In this section, we numerically illustrate that the power of the permuted aggregated HSIC test approximates very well the power of the theoretical aggregated test, as soon as enough permutations are used to estimate the quantile under the null hypothesis. \\

All along this section, we rely on the following data generating mechanism inspired from the Ishigami function \cite{ishigami1990importance}. Let
\begin{equation}\label{eq:modelM}
X=U_1 \quad \mbox{and}\quad Y = \sin (U_1) + 4 \sin^2 (U_2) + 0.5\ U_3^4 \sin (U_1).
\end{equation}
where $U_1$, $U_2$ and $U_3$ are independent uniform random variables on $[0,1]$.\\

The practical implementation of the theoretical and permuted aggregated testing procedures are described in Algorithms \ref{AlgoTheo} and \ref{AlgoPermu}. They both require the estimation of the value of $u_\alpha$ defined in Equation \eqref{u_alpha}. A very straightforward approach to do so is to proceed by dichotomy on the interval $\left[\alpha, M \right]$, where $M = \inf_{(\lambda, \mu) \in \W} \ac{e^{\omega_{\lambda,\mu}}}$ ($u_\alpha$ belonging to this interval as mentioned in Section \ref{Agreg-proc}). 
More precisely, we need to estimate for different values of $u$, the probability
\begin{equation}\label{eq:Pu}
P(u) = P_{f_1\otimes f_2} \pa{ \sup_{(\lambda,\mu) \in \W} \ac{\widehat{\HSIC}_{\lambda,\mu} - q_{1- u e^{- \omega_{\lambda,\mu}}}^{\lambda,\mu}} > 0}. 
\end{equation}
In the theoretical case, this probability is approached by Monte Carlo independently on the observation (provided that we can simulate under the null hypothesis) whereas in the permuted case, it is based on samples obtained by permuting the observation. 
The algorithmic complexity of Algorithm \ref{AlgoPermu} is $O\pa{(B_1+B_2) \abs{\W} n^2}$, corresponding to the estimation of the HSIC for all the permutations in Step 1, and all the windows in the collection $\W$.

\begin{algorithm}[h] 
\caption{\textit{Theoretical aggregated procedure}}
\flushleft\textbf{Input:} The observed $n$-sample, a prescribed level $\alpha$, a collection of bandwidths $\W$ and a family of weights $(\omega_{\lambda,\mu})_{(\lambda,\mu)\in\W}$. 
\begin{enumerate}
\item\label{steptheoun} Simulate a first set, denoted set \texttt{(A)}, of 500.000 $n$-samples under the null hypothesis (to estimate the quantiles) and a second set, denoted \texttt{(B)} of 1000 $n$-samples also under the null hypothesis (to estimate the probabilities $P(u)$ defined in Equation \eqref{eq:Pu} for different values of $u$).
\item Set $u_{min}=\alpha$ and $u_{max}=M$, where $M= \inf_{(\lambda, \mu) \in \W} \ac{e^{\omega_{\lambda,\mu}}}$. 
\item\label{stepDichoth} While $(u_{max}-u_{min}) > 10^{-3} u_{min}$, repeat the following steps. 
\begin{enumerate}
\item Set $u=(u_{min}+u_{max})/2$. 
\item For all $(\lambda,\mu)$ in $\W$, compute the Monte Carlo estimator $\tilde{q}_{1- u e^{- \omega_{\lambda,\mu}}}^{\lambda,\mu}$ of the quantile $q_{1- u e^{- \omega_{\lambda,\mu}}}^{\lambda,\mu}$ using the 500.000 samples of set \texttt{(A)}. 
\item Estimate the probability$P(u)$ by Monte Carlo using the 1000 samples of set \texttt{(B)}. More precisely, consider $\hat{P}_u$ as the ratio of times at least one $\widehat{\HSIC}_{\lambda,\mu}$ is greater than $\tilde q_{1- u e^{- \omega_{\lambda,\mu}}}^{\lambda,\mu}$. 
\item If $\hat{P}_u\leq \alpha$, then set $u_{min}=u$. Else set $u_{max}=u$ and repeat Step \ref{stepDichoth}. 
\end{enumerate}
\item\label{steptheoder} Set $\tilde{u}_\alpha=u$ and the quantiles with corrected levels $\pa{\tilde{q}_{1- \tilde{u}_\alpha e^{- \omega_{\lambda,\mu}}}^{\lambda,\mu}}_{(\lambda,\mu)\in\W}$. 
\item\label{steptheoappli} Finally, compute the observed statistics $(\widehat{\HSIC}_{\lambda,\mu})_{(\lambda,\mu)\in\W}$ (on the given observation) and reject the null hypothesis if there is at least one $(\lambda, \mu)$ such that 
$$\widehat{\HSIC}_{\lambda,\mu} > \tilde{q}_{1- \tilde{u}_\alpha e^{- \omega_{\lambda,\mu}}}^{\lambda,\mu}.$$ 
\end{enumerate}
\label{AlgoTheo}
\end{algorithm}

\begin{algorithm}[h!] 
\caption{\textit{Permuted aggregated procedure}}
\flushleft\textbf{Input:} The observed $n$-sample $\ZZ_n$, a prescribed level $\alpha$, a collection of bandwidths $\W$ and a family of weights $(\omega_{\lambda,\mu})_{(\lambda,\mu)\in\W}$. 
\begin{enumerate}
\item\label{steppermun} Generate a first set, say \texttt{(A')}, of $B_1$ i.i.d. random permutations of $\{1,\ldots,n\}$ (to estimate the quantiles), and independently generate a second set, denoted \texttt{(B')}, of $B_2$ i.i.d. random permutations of $\{1,\ldots,n\}$ (to estimate the probabilities $P(u)$ defined in Equation \eqref{eq:Pu}), all independent of $\ZZ_n$.
\item\label{stepinitDichoperm} Set $u_{min}=\alpha$ and $u_{max}= M$, where $M= \inf_{(\lambda, \mu) \in \W} \ac{e^{\omega_{\lambda,\mu}}}$. 
\item\label{stepDichoperm} While $(u_{max}-u_{min}) > 10^{-3} u_{min}$, repeat the following steps. 
\begin{enumerate}
\item Set $u=(u_{min}+u_{max})/2$. 
\item For all $(\lambda,\mu)$ in $\W$, compute the permuted quantile with Monte Carlo approximation $\hat{q}^{\lambda,\mu}_{1-u e^{-\omega_{\lambda,\mu}}}$ as defined in \eqref{eq:permquant} using the set \texttt{(A')}. 
\item Estimate $P(u)$ by permutation with Monte Carlo approximation using the set \texttt{(B')}. More precisely, consider 
$$\hat{P}^\star_u(\ZZ_n) = \frac{1}{B_2} \sum_{b=1}^{B_2} \mathds{1}_{\max_{(\lambda,\mu) \in \W} \ac{\widehat{H}^{\rpermb_b}_{\lambda,\mu} - \hat{q}^{\lambda,\mu}_{1-u e^{-\omega_{\lambda,\mu}}}}>0},$$ where $(\rpermb_b)_{1\leq b\leq B_2}$ denote the permutations of set \texttt{(B')} and $\widehat{H}^{\rpermb_b}_{\lambda,\mu}$ is the statistic computed on the $b${th} permuted sample $\ZZ_n^{\rpermb_b}$, namely $\widehat{\HSIC}_{\lambda,\mu}\pa{\ZZ_n^{\rpermb_b}}$ . 
\item If $\hat{P}^\star_u(\ZZ_n)\leq \alpha$, then set $u_{min}=u$. Else set $u_{max}=u$ and repeat Step \ref{stepDichoth}. 
\end{enumerate}
\item\label{steppermder} Set $\hat{u}_\alpha=u$ and the quantiles with corrected levels $\pa{\hat{q}_{1- \hat{u}_\alpha e^{- \omega_{\lambda,\mu}}}^{\lambda,\mu}}_{(\lambda,\mu)\in\W}$. 
\item\label{steptheoappli} Finally, compute the observed statistics $(\widehat{\HSIC}_{\lambda,\mu})_{(\lambda,\mu)\in\W}$ (on the given observation) and reject the null hypothesis if there is at least one $(\lambda, \mu)$ such that 
$$\widehat{\HSIC}_{\lambda,\mu} > \hat{q}_{1- \hat{u}_\alpha e^{- \omega_{\lambda,\mu}}}^{\lambda,\mu}.$$ 
\end{enumerate}
\label{AlgoPermu}
\end{algorithm}

\paragraph{Theoretical power} For a given sample size $n$ and a given collection of bandwidths $\W$ with associated weights, we estimate the power of the theoretical aggregated test as follows. 
Since the approximation of the value of $u_\alpha$ and of the quantiles can be done independently of the observation, we run Steps \ref{steptheoun} to \ref{steptheoder} of Algorithm \ref{AlgoTheo} only once.
Then, we generate 1000 i.i.d. samples (observations) and for each one, we apply Step \ref{steptheoappli} of Algorithm \ref{AlgoTheo}. 
Finally, we estimate the theoretical power by $\hat\pi_{th} (n,\alpha)$ which is the proportion of times the aggregated procedure rejects the null hypothesis.

\paragraph{Permuted power} 
Unlike the theoretical case, we do not assume we are able to simulate under the null hypothesis to estimate the quantiles and to compute the correction $u_\alpha$. 
Note that for the permuted test, Step \ref{stepDichoperm} of Algorithm \ref{AlgoPermu} depends on the observation and needs to be done for each new observation. 
Hence, for a given sample size $n$, a given collection of bandwidths $\W$ and associated weights, we generate 1000 i.i.d. samples and for each one, we apply Steps \ref{steppermun} to \ref{steptheoappli} of Algorithm \ref{AlgoPermu}. 
Finally, we estimate the power of the permuted aggregated test by $\hat\pi(n,\alpha,B_1,B_2)$ which is the ratio of times the null hypothesis is rejected.

\paragraph{Numerical results}
In all the following, the prescribed level of the tests is set to $\alpha = 0.05$ and we consider sample sizes $n$ in $\ac{50, 100, 200}$. 
We consider six different collections of bandwidths $\left( \W_r \right)_{2 \leq r \leq 7}$, defined for all $r$ by 
\begin{equation} \label{eq:collectpermvstheo}
\W_r = \left\{1, 1/2, \ldots, 1/2^{r-1} \right\}^2.
\end{equation}
Note that, the case $r = 1$ would correspond to the single test with $\lambda = \mu =1$.
Moreover, for each $r$, we consider uniform weights defined for all $(\lambda ,\mu)$ in the collection $\W_r$ by
\begin{equation}\label{eq:defPoidsUnif}
\omega_{\lambda,\mu} = \log(r^2).
\end{equation} 

For the permuted aggregated procedure, the number $B_1$ of permutations used to estimate the quantiles varies in $\ac{100, 200, 500, 1000, \ldots, 5000}$ and the number of permutations used to estimate the probabilities $P(u)$ is set to $B_2 = 500$. 

For each triplet $(r,n,B_1)$, the empirical power of both the theoretical and the permuted aggregated testing procedures, respectively denoted $\hat\pi_{th} (n,\alpha)$ and $\hat\pi(n,\alpha,B_1,B_2)$, are obtained from 1000 different samples as described above. 
To compare them, we consider the relative absolute error defined by
$$Err(n,\alpha,B_1,B_2) = \frac{\abs{\hat\pi(n,\alpha,B_1,B_2) - \hat\pi_{th} (n,\alpha)}}{\hat\pi_{th} (n,\alpha)}.$$ 
Results are given in Figure~\ref{fig:PowerPermutedAgreg}. Notice that, regardless of the sample size $n$, the required number $B_1$ of permutations to well approximate the theoretical power increases with $r$. In fact, the supremum in Equation \eqref{Ualphaestimation} becomes more difficult to estimate as the number $r^2$ of aggregated tests increases. Unsurprisingly, for a given $B_1$, the accuracy of the power estimation increases with $n$ as in the case of single tests. In particular, we observe that for $n = 50$, the largest error becomes less than $10 \%$ from $B_1 = 3500$, while this threshold seems to be achieved from $B_1 = 3000$ for $r = 4,5,6$ and from $B_1 = 500$ for $r = 2,3$. For larger sample sizes $n = 100$ and $200$, a good approximation of the theoretical test seems to be achieved from small values of $B_1$, even for a relatively large number of aggregated tests. In particular, for $n = 200$, an error smaller than $10 \%$ is reached for all values of~$B_1$.

\begin{figure}
 \centering
 \includegraphics[width=0.32\textwidth]{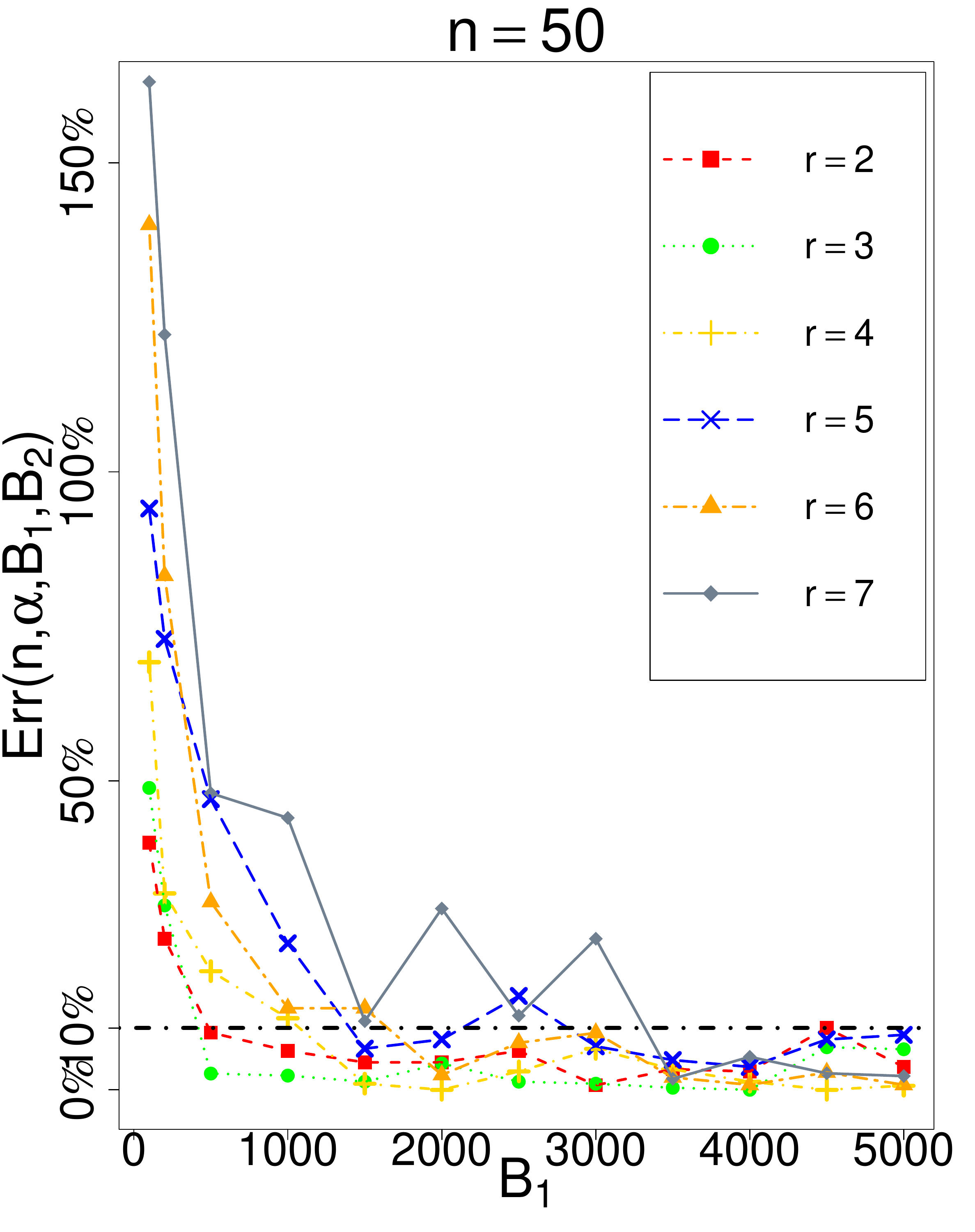}
 \includegraphics[width=0.32\textwidth]{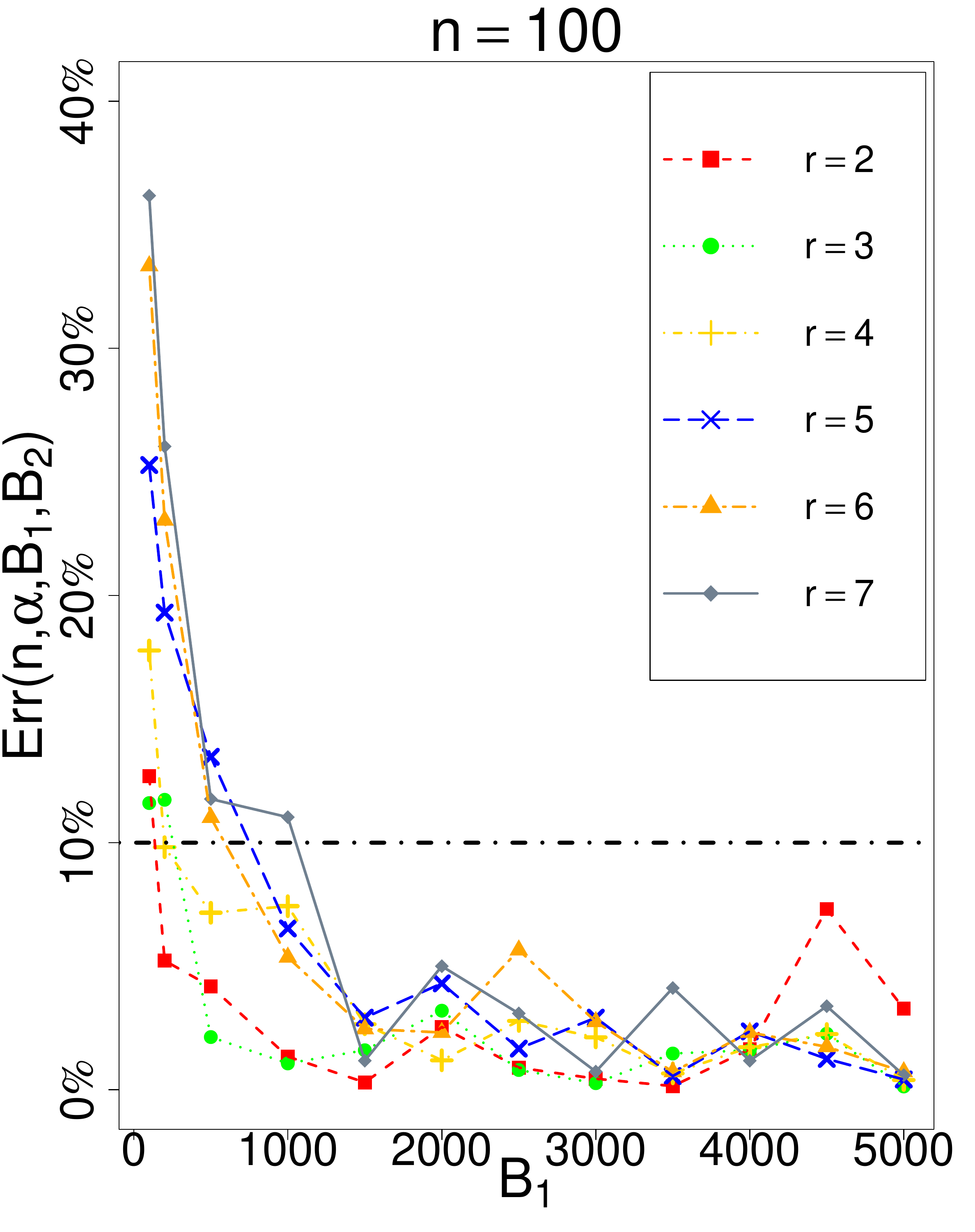}
 \includegraphics[width=0.32\textwidth]{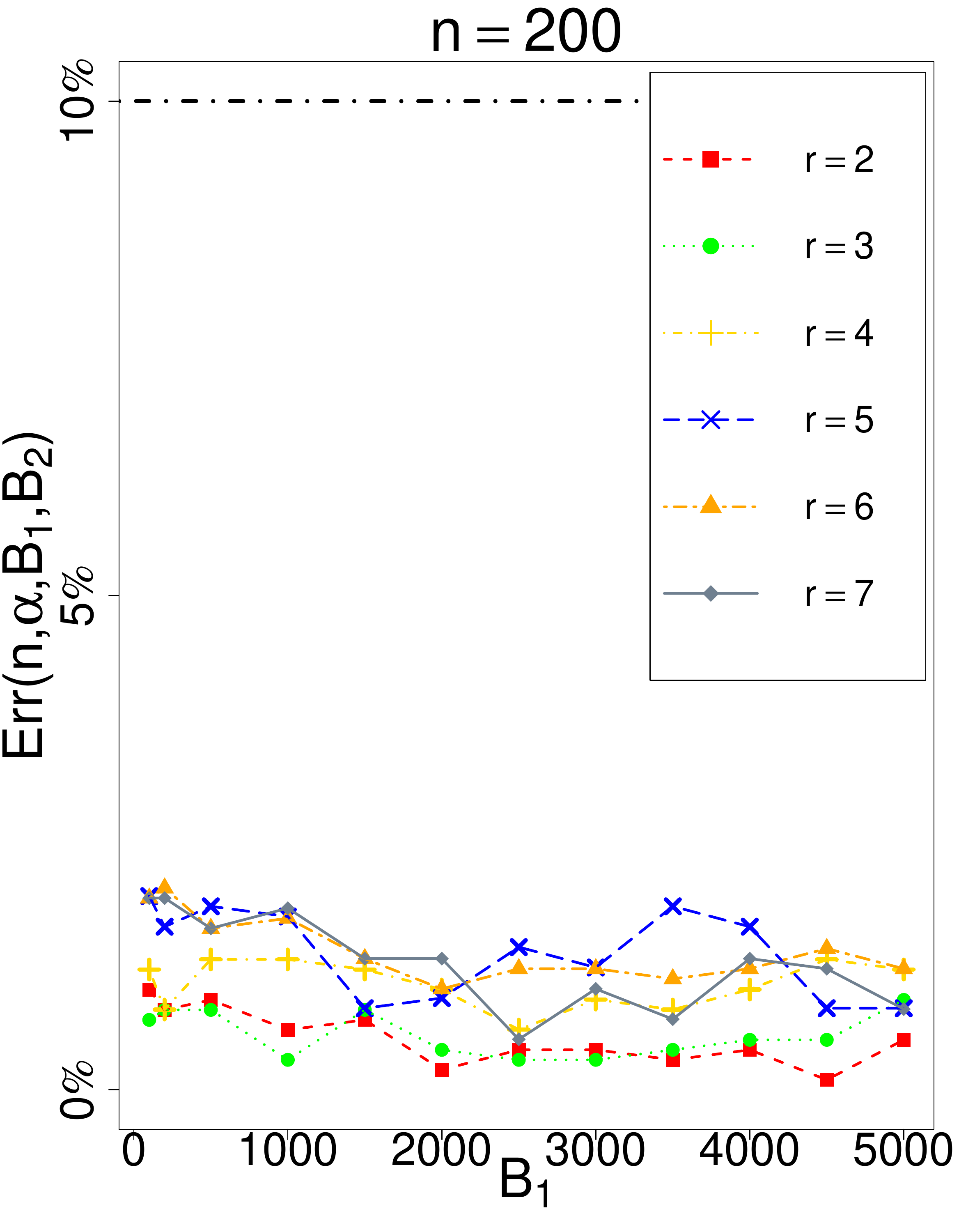}
 \caption{Absolute relative error between the empirical power of the theoretical and the permuted aggregated HSIC procedures w.r.t. the number $B_1$ of permutations, estimated from 1000 samples of sizes $n = 50, 100$ and $200$, with $B_2=500$, the bandwidth collections $\W_r$ and associated weights resp. defined in \eqref{eq:collectpermvstheo} and \eqref{eq:defPoidsUnif} and prescribed level $\alpha = 0.05$. 
}
 \label{fig:PowerPermutedAgreg}
\end{figure}

\paragraph{}
All these results show that both theoretical and permuted tests have comparable powers provided that the sample size and the number of permutations are large enough. 
In the following, we numerically study the power of the permuted tests, which are used in practice. 

\subsection{Comparison with existing tests}
\label{Comparison_MINT}

To complete this simulation study, we compare our aggregated procedure with some existing reference tests of independence. For this, we simulate accordingly to the data generating mechanisms of \cite{berrett2019nonparametric}, and a basic Gaussian model. 
\begin{enumerate}[label=(\roman*)]
\item\label{Berrettex1} For $l$ in $\ac{1,\ldots,10}$, define the joint density $f_{[l]}$ of $(X,Y)$ for all $(x,y)$ in $[-\pi, \pi]$ by $$f_{[l]} (x,y) = \cro{1 + \sin(lx) \sin(ly)}/(4 \pi^2).$$
\item\label{Berrettex2} For $l$ in $\ac{1,\ldots,10}$, let $X = L \cos \Theta + \varepsilon_1/4$ and $Y = L \sin \Theta + \varepsilon_2/4,$ where $L$, $\Theta$, $\varepsilon_1$ and $\varepsilon_2$ are independent, with $L$ is uniformly distributed on $\left\{ 1, \ldots, l \right\}$, $\Theta$ is uniformly distributed on $[0,2 \pi]$ and $\varepsilon_1$, $\varepsilon_2$ are standard normal random variables. 
\item\label{Berrettex3} For $\rho$ in $\ac{0.1, 0.2, \ldots ,1}$, let $X$ be a uniform random variable on $[-1,1]$ and define $Y = \abs{X}^{\rho} \varepsilon,$ where $\varepsilon$ is a standard normal random variable independent with $X$. 
\item\label{Gaussian} For $\rho$ in $\ac{0,0.1,\ldots,0.9}$, let $(X,Y)$ be a centered Gaussian vector such that $\Var(X)=\Var(Y)=1$ and $\cov(X,Y)=\rho$. 
\end{enumerate}
We also consider the bivariate case $X = (X^{(1)}, X^{(2)})$ and $Y = (Y^{(1)},Y^{(2)})$ where $(X^{(1)},Y^{(1)})$ is generated according to mechanisms \ref{Berrettex1}, \ref{Berrettex2} or \ref{Berrettex3}, while $X^{(2)}$, $Y^{(2)}$ are independent uniform random variables on $[0,1]$ and independent from $(X^{(1)},Y^{(1)})$. \\

\begin{sloppypar}
The numerical study of the impact of the bandwidth collection and the associated weights on the power of the aggregated procedure done in Appendix \ref{sect:choicecollweights} of the supplementary material suggests the following methodological choices. 
Inspired by usual heuristic bandwidths (see, e.g., \cite{de2017sensitivity}), define 
\begin{equation}\label{RefWidths}
\widetilde{\lambda}^2 = \frac{1}{2n(n-1)} \sum_{1\leq i\neq j\leq n} \norm{X_i - X_j}^2 \ \text{and} \ \widetilde{\mu}^2 = \frac{1}{2n(n-1)} \sum_{1\leq i\neq j\leq n} \norm{Y_i - Y_j}^2,
\end{equation}
where $\norm{\cdot}$ denotes the Euclidean norm. 
Note that, in the univariate case ($p=q=1$), $\widetilde{\lambda}$ and $\widetilde{\mu}$ are the empirical standard deviation of $X$ and $Y$ respectively. 
In the univariate case, we consider the collections defined by 
\begin{equation}\label{eq:collectioncompartestsunidim6}
\widetilde{\W} = \ac{2^{-m}\pa{\widetilde{\lambda},\widetilde{\mu}} ; 0\leq m\leq 6}.
\end{equation}
Similarly, in the bivariate case, the bandwidth collections are defined by 
\begin{equation}\label{eq:collectioncompartestsbidim}
\widetilde{\W} = \ac{2^{-m}\pa{\widetilde{\lambda},\widetilde{\lambda},\widetilde{\mu},\widetilde{\mu}} ; 0\leq m\leq 6}.
\end{equation}
We also consider exponential weights, that are defined, by analogy with Equation \eqref{eq:weigths}, for all bandwidths 
$2^{-m}(\widetilde{\lambda}, \widetilde{\mu})$ or $2^{-m}(\widetilde{\lambda},\widetilde{\lambda}, \widetilde{\mu}, \widetilde{\mu})$ as 
\begin{equation}\label{omegavarsur2iet2j}
\omega_{[m]} = 2\log\pa{m+1} + \log \left( \sum_{m'=0}^6\frac{1}{(m'+1)^2} \right).
\end{equation}
Note that the last term in \eqref{omegavarsur2iet2j} ensures that $\sum_{(\lambda,\mu) \in \widetilde{\W}} e^{- \omega_{\lambda,\mu}} = 1$. \\
\end{sloppypar}

\begin{figure}
 \centering
 \includegraphics[width=0.42\textwidth]{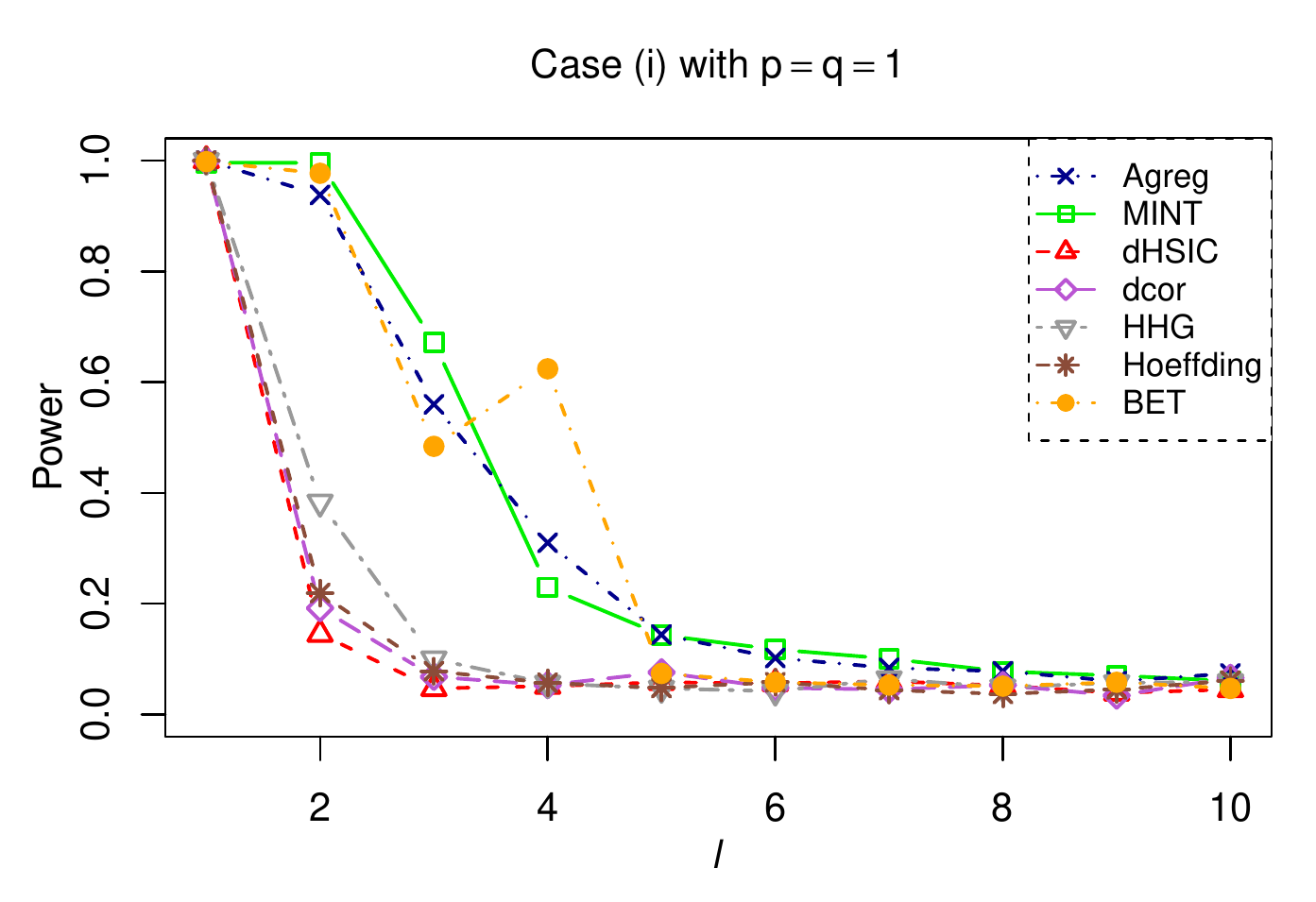}
 \includegraphics[width=0.42\textwidth]{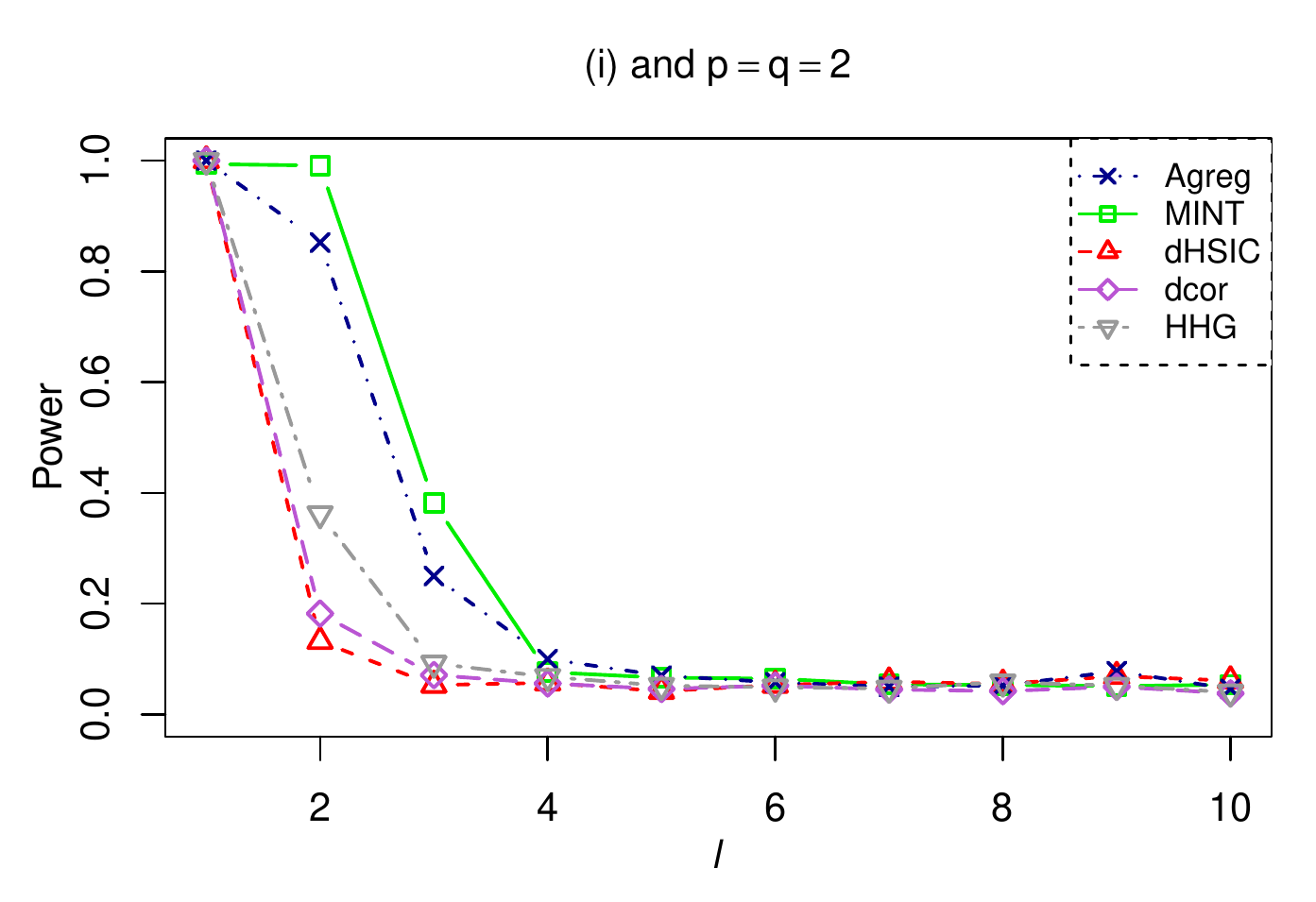}
 \includegraphics[width=0.42\textwidth]{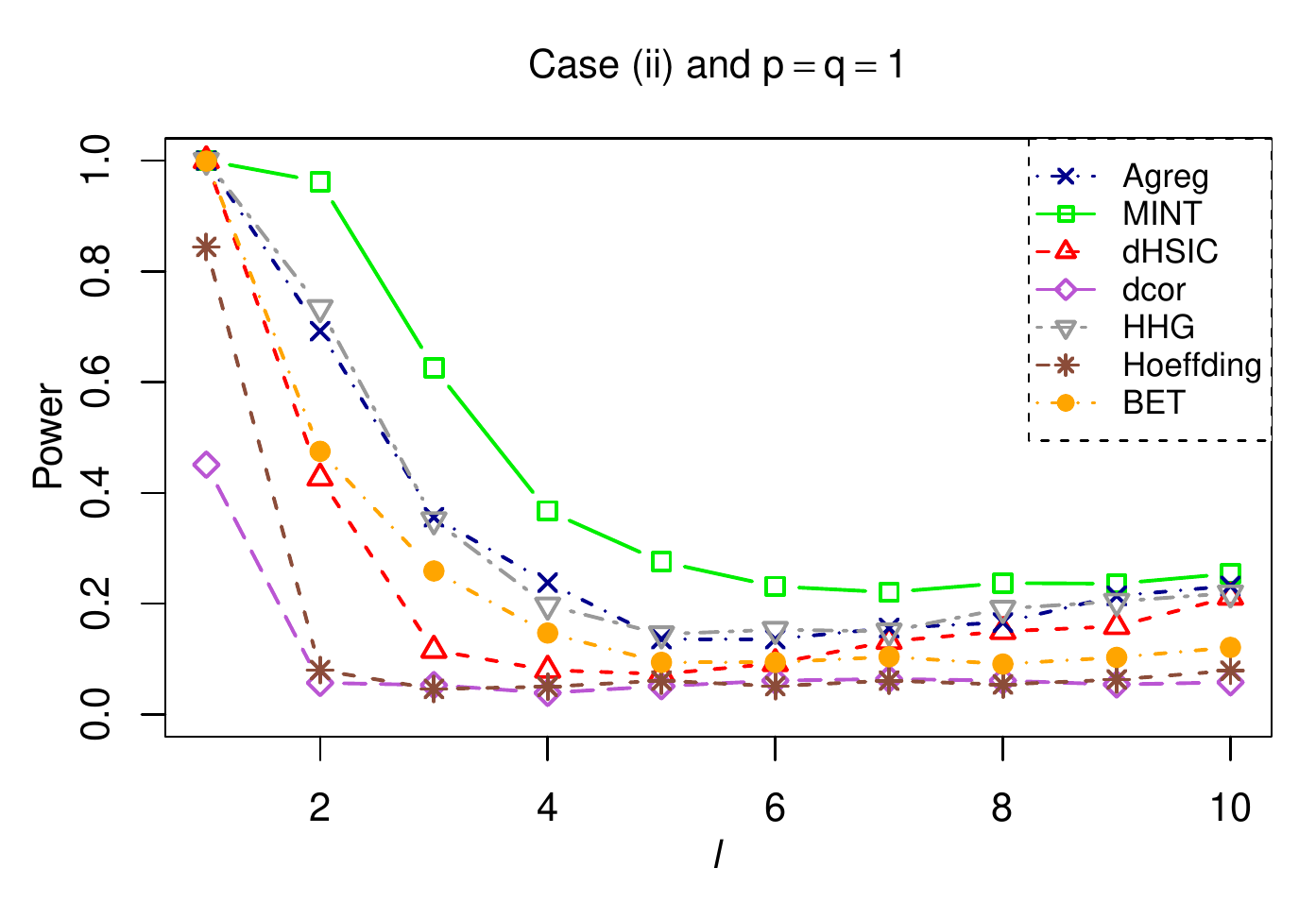}
 \includegraphics[width=0.42\textwidth]{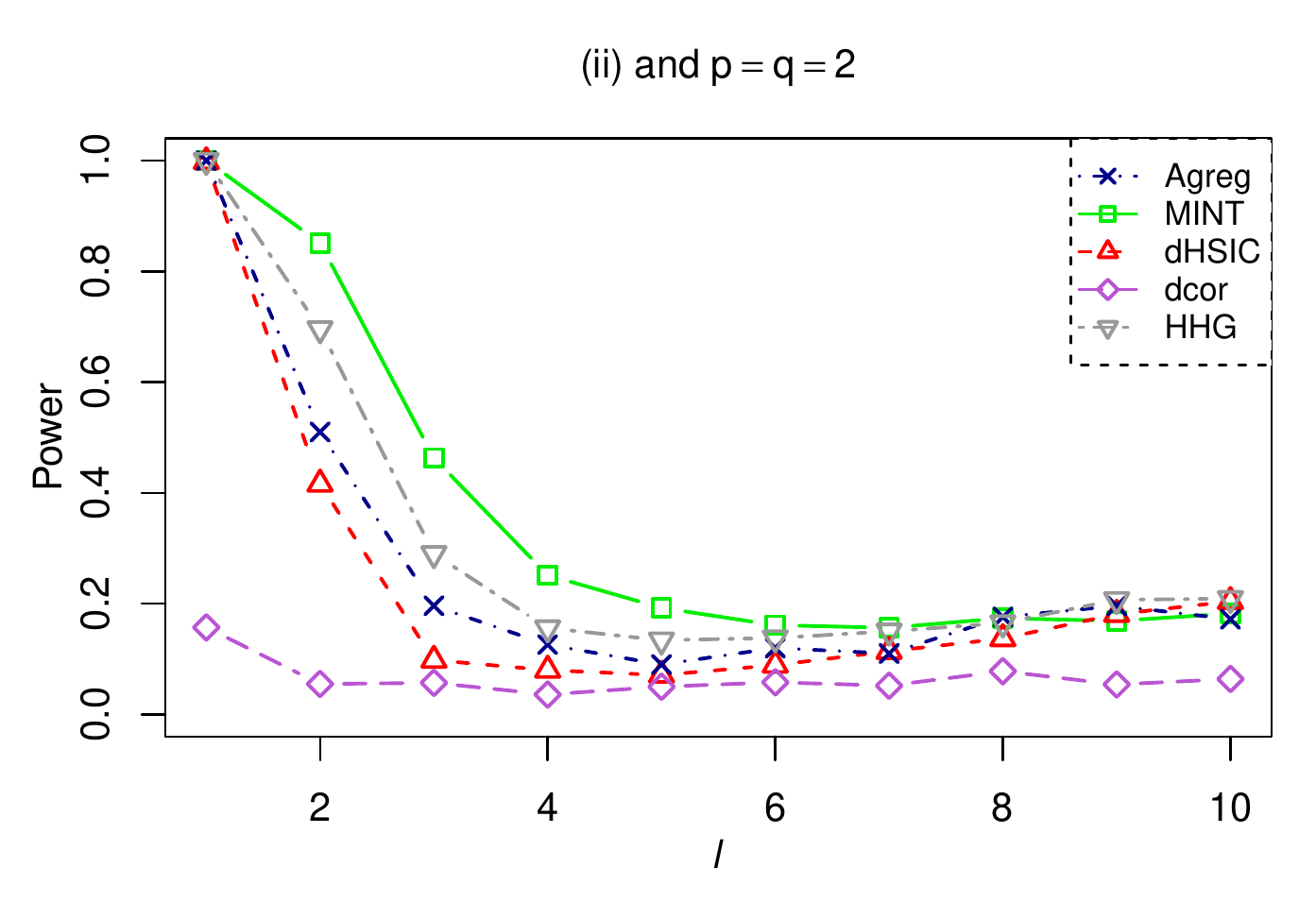}
 \includegraphics[width=0.42\textwidth]{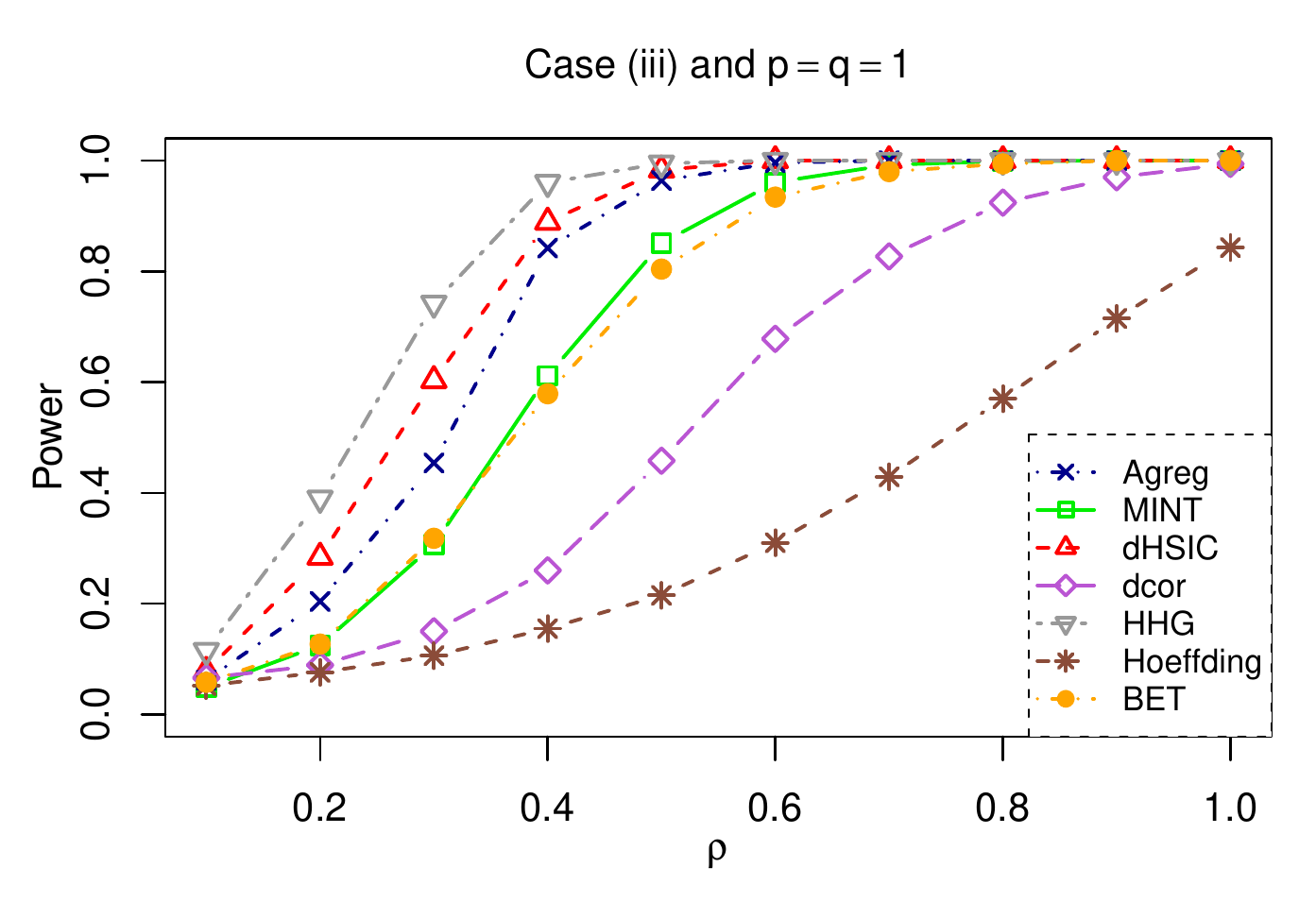}
 \includegraphics[width=0.42\textwidth]{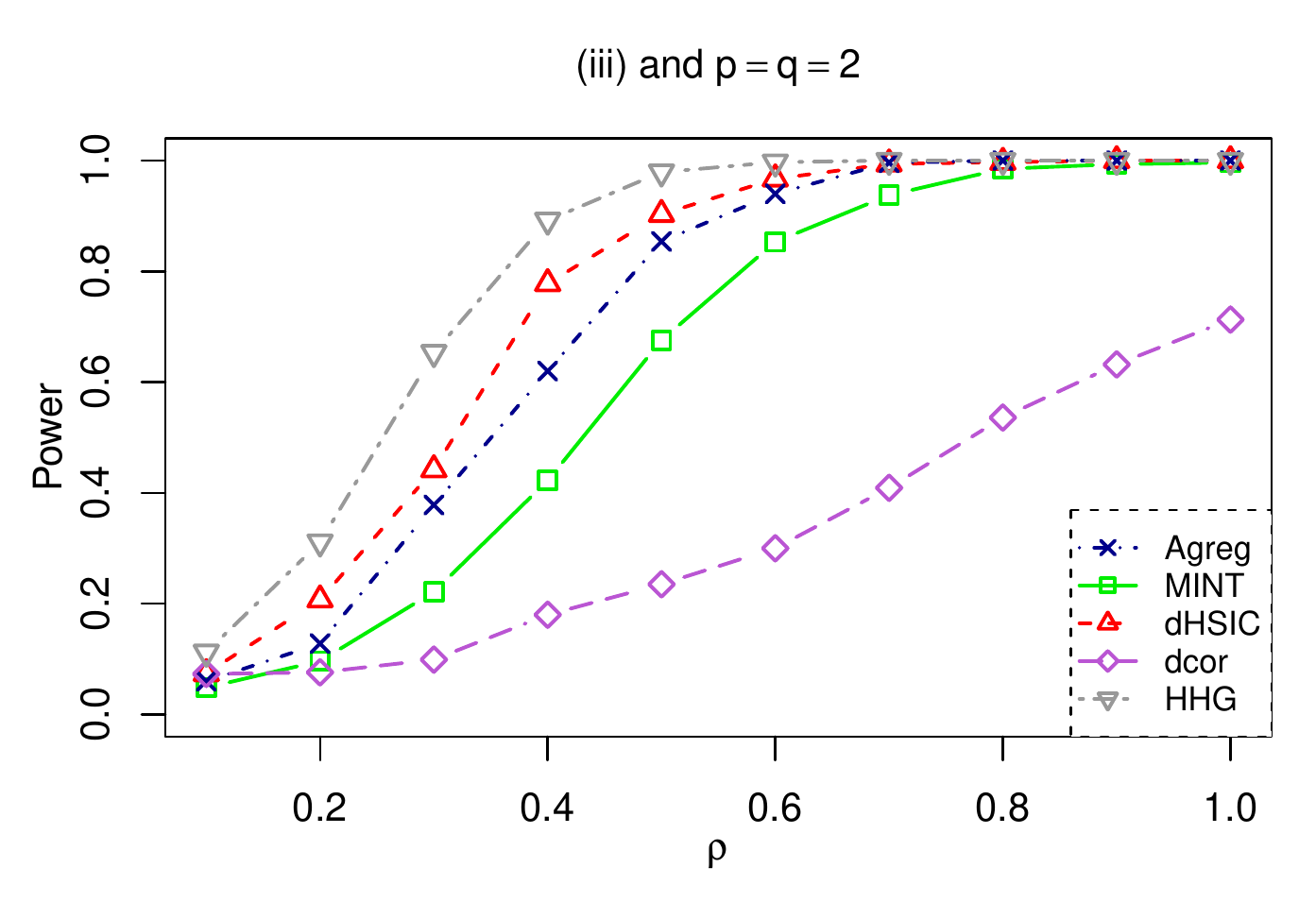}
 \includegraphics[width=0.42\textwidth]{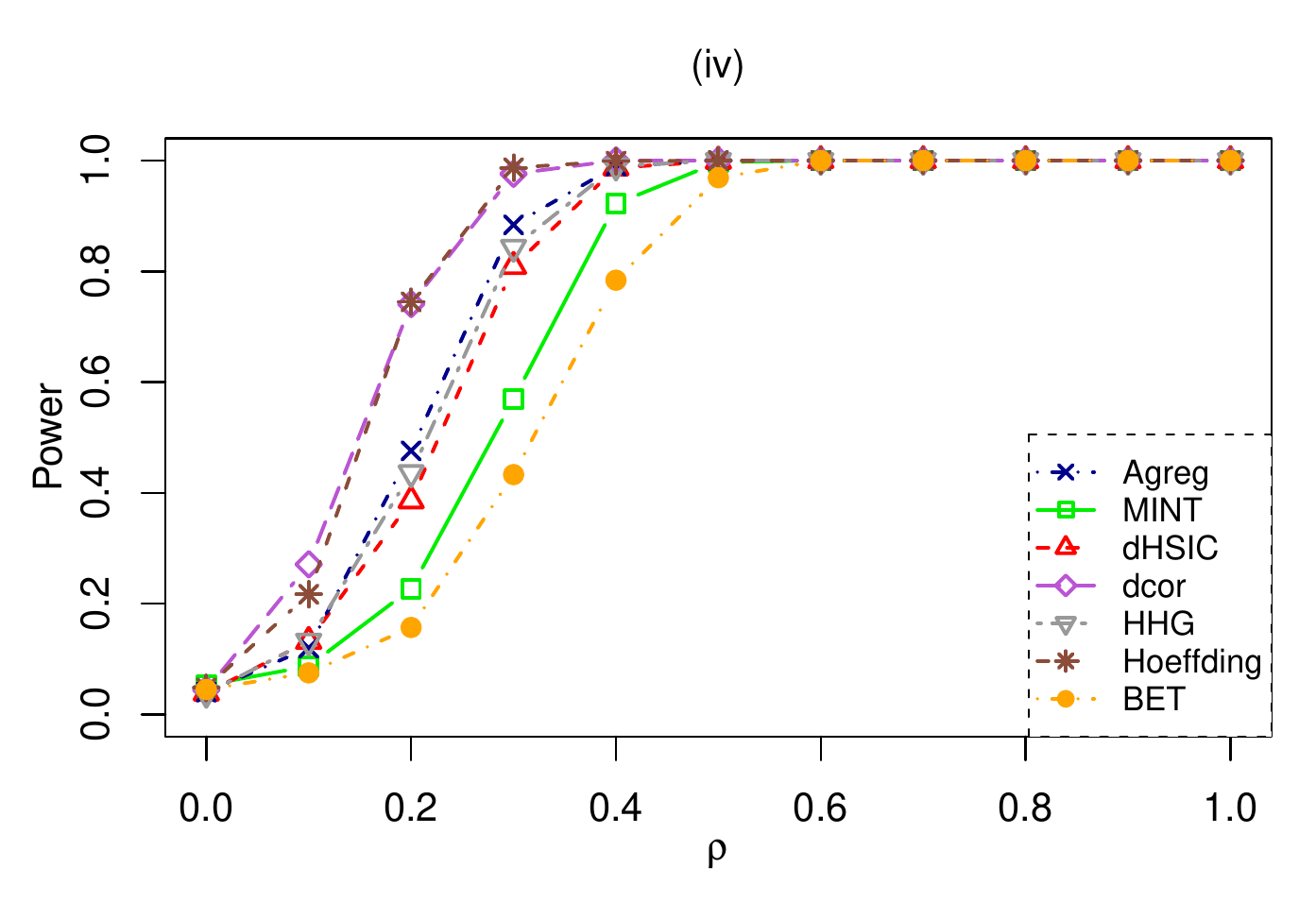} 
 \hspace{0.42\textwidth}
 \caption{Power curves of the permuted aggregated HSIC test with $B_1 = 3000$ and $B_2 = 500$, bandwidth collection $\widetilde{\W}$ defined in \eqref{eq:collectioncompartestsunidim6} or \eqref{eq:collectioncompartestsbidim} and exponential weights \eqref{omegavarsur2iet2j}. It is compared to the MINT, the single HSIC test, the distance covariance, Heller, Heller and Gorfine's test (HHG), Hoeffding's $D$-test and the BET. The empirical power is estimated from $1000$ samples of size $n = 200$. The prescribed level is $\alpha=0.05$.}
\label{fig:PowerCurvesMINTHSICAgreg}
\end{figure}

\begin{sloppypar}
In Figure \ref{fig:PowerCurvesMINTHSICAgreg}, we compare our permution-based aggregated HSIC test with the mutual information test (MINT) of \cite{berrett2019nonparametric} implemented in the {\tt R} package {\tt IndepTest}, the permutation-based HSIC single test (HSIC) implemented in the {\tt R} package {\tt dHSIC} \cite{pfister2018kernel} with $B = 1000$ permutations, 
the distance covariance of \cite{szekely2007measuring} implemented in the {\tt R} package {\tt energy}, the test of \cite{heller2016consistent} (HHG), the $D$-test of \cite{hoeffding1948non} implemented in the {\tt R} package {\tt Hmisc} 
and the binary expansion test (BET) of \cite{zhang2019bet}. 
\end{sloppypar}

For each example, we simulate samples with size $n=200$. In line with the results obtained in Section \ref{sect:comparTheoPerm}, Algorithm \ref{AlgoPermu} is applied with $B_1 = 3000$ and $B_2 = 500$. The power of the different tests is estimated using 1000 different samples of $(X,Y)$ and is represented w.r.t $l$ for simulated data from \ref{Berrettex1} and \ref{Berrettex2} and w.r.t $\rho$ for \ref{Berrettex3} and \ref{Gaussian}. 

As expected, no procedure of testing constantly yields the best performances in all cases. Indeed, it is well known that no uniformly most powerful test exists. 
However, as for the MINT procedure, the HSIC aggregated procedure seems to yield competitive results on all examples, contrarily to other procedures such as the distance covariance or Hoeffding's $D$-test which perform very well in the Gaussian case, but are not very powerfull in the other cases. 
Moreover, in most cases, the HSIC aggregated procedure performs better than the single HISC test, which illustrates the benefits of aggregation. 

\medskip

\begin{center}
{\Large\textsc{Supplementary material}}
\end{center}
\textbf{Supplement to \enquote{Adaptive test of independence based on HSIC measures}}\\
This Supplement contains sharp upper bounds for the uniform separation rates over Nikol'skii-Besov balls, a further numerical study and all the proofs.

\newpage

\begin{center}
{\Large \textsc{Supplement to \\
\medskip
\enquote{Adaptive test of independence based on HSIC measures}}}
\end{center}

\bigskip

Appendix \ref{section:supplNikol} contains sharp upper bounds for the uniform separation rates over anisotropic Nikol'skii-Besov balls of the single and the aggregated HSIC tests developed in the main article. 
In Appendix \ref{sect:NumericalSimulationsbis}, further simulations illustrate a comparison in terms of power between the theoretical and permuted single tests on the one hand, and the impact of the bandwidth collection and the weights choices on the power on the other hand. 
Finally, Appendix \ref{sect:proofs} is devoted to all the proofs. 
\bigskip

The references of Equations, Theorems, Propositions, etc, that use only numbers such as (3.1) for instance, refer to the main article {\it Adaptive test of independence based on HSIC measures}.

\appendix

\section{Control of the uniform separation rate over anisotropic Nikol'skii-Besov balls}
\label{section:supplNikol}

In this section, we consider anisotropic Nikol'skii-Besov balls which allow to take into account different regularity parameters in each direction in $ \R^{p+q}$. 
The anisotropic Nikol'skii-Besov ball $\mathcal{N}^{\delta}_{2,d} (R)$ in dimension $d$ in $\mathbb{N}^*$, with regularity parameter $\delta = (\delta_1, \ldots, \delta_d)$ in $(0,+\infty)^d$ and positive radius $R$, is defined by
\begin{align}
\mathcal{N}^{\delta}_{2,d} (R) = \biggl\lbrace & s: \mathbb{R}^d \rightarrow \mathbb{R} \ ;\ \text{for all $1\leq i\leq d$ and all $u_1, \ldots , u_{d} , v \in \mathbb{R}$},\biggr. \nonumber \\
& \text{$s$ has continuous partial derivatives $D_i^{\lfloor \delta_i \rfloor}$ of order $\lfloor \delta_i \rfloor$ w.r.t $u_i$, and}  \nonumber \\
& \biggl. \norm{ D_i^{\lfloor \delta_i \rfloor}s (u_1, \ldots ,u_i +v, \ldots ,u_d) - D_i^{\lfloor \delta_i \rfloor}s (u_1, \ldots ,u_d) }_{2} \leq R \abs{v}^{\delta_i - \lfloor \delta_i \rfloor} \biggr\rbrace, \nonumber 
\end{align}
where  $\lfloor  \delta_i\rfloor$ denotes the floor function of $\delta_i $ if $\delta_i$ is not integer and $ \lfloor  \delta_i\rfloor = \delta_i-1$ if $\delta_i$ is an integer. \\

As in the Sobolev case, we study optimality over $\mathcal{N}^{\delta}_{2,d} (R,R')$ defined by 
\begin{equation}\label{eq:Nikolinf}
\mathcal{N}^{\delta}_{2,d}(R,R') = \mathcal{N}^{\delta}_{2,d}(R) \cap \ac{f ; \max\ac{\norm{f}_{\infty}, \norm{f_1}_{\infty}, \norm{f_2}_{\infty}}\leq R'}.
\end{equation}

As in the Sobolev case, we prove upper bounds for the uniform separation rate of the tests defined in the main article over these new regularity spaces. 
Section \ref{unifseprateNikolsingle} is devoted to the single test $\Delta_\alpha^{\lambda,\mu}$ with fixed bandwidths defined in Equation \eqref{Deltalambdamu}, and the study of the aggregated test $\Delta_\alpha$ defined in Equation \eqref{agregtest} in done in Sections \ref{unifseprateNikoloracleagg} and \ref{unifseprateNikolagg}. 

\subsection{Uniform separation rate of the single tests over Nikol'skii-Besov balls}
\label{unifseprateNikolsingle}

In this section, we consider a fixed bandwidth $(\lambda,\mu)$.
Lemma \ref{TBNB} provides an upper bound of the bias term, similar to that of Lemma \ref{TBSB}, in the case when $ f-f_1 \otimes f_2$ belongs to an anisotropic Nikol'skii-Besov ball.

\begin{lemm} 
Let $ \psi = f-f_1 \otimes f_2 $ and assume that $\psi$ belongs to $\mathcal{N}^{\delta}_{2,p+q} (R)$, where the regularity parameter $\delta = (\nu_1, \ldots , \nu_p, \gamma_1, \ldots ,\gamma_q)$ belongs to $(0,2]^{p+q}$. 
Let $\varphi_\lambda$ and $\phi_\mu$ be the functions defined in \eqref{varphilambdaphimu}. 
Then, the bias term can be controlled as follows
\begin{equation*}
\norm{ \psi - \psi \ast (\varphi_\lambda \otimes \phi_\mu) }_{2}^2 \leq C(\delta,R) \left[ \sum_{i=1}^p \lambda_i^{2 \nu_i} +  \sum_{j=1}^q \mu_j^{2 \gamma_j} \right].
\end{equation*}
\label{TBNB}
\end{lemm}

In the Nikol'skii-Besov case, the control of the bias term requires a restriction on the regularity parameter to $(0,2]^{p+q}$, which comes from the fact that the Gaussian kernel is of order 2. In order to extend the range of the upper bound, kernels of higher order should be considered. This generalization lies beyond the scope of this article and requires further developments. 
As in Section \ref{Uniform-separation-rate}, one can deduce from Theorem \ref{th:powerful-testparam}  upper bounds for the uniform separation rates of the single test $\Delta_{\alpha}^{\lambda,\mu}$ over Nikol'skii-Besov balls.

\begin{theorem} \label{ThresholdNikol}
Let $\alpha$, $\beta$ in $(0,1)$, $\delta = (\nu_1, \ldots , \nu_p, \gamma_1, \ldots ,\gamma_q)$ in $(0,2]^{p+q}$ and $R,R'>0$.
Consider bandwidths $(\lambda,\mu)$ satisfying Assumptions $\boldsymbol{\mathcal{A}_2(\alpha)}$ and denote $\Delta_{\alpha}^{\lambda,\mu}$ the test defined by \eqref{Deltalambdamu}. 
Then, the uniform separation rate defined in \eqref{seprate} of the test $\Delta_{\alpha}^{\lambda,\mu}$ over the Nikol'skii-Besov ball $\mathcal{N}^{\delta}_{2,p+q} (R,R')$ defined in Equation \eqref{eq:Nikolinf} can be upper bounded as follows
\begin{multline}\label{eq:majnik}
\cro{\rho \left( \Delta_{\alpha}^{\lambda,\mu} , \mathcal{N}^{\delta}_{2,p+q} (R,R') , \beta \right) }^2\leq \displaystyle C(\delta,R) \left[ \sum_{i=1}^p \lambda_i^{2 \nu_i} +  \sum_{j=1}^q \mu_j^{2 \gamma_j} \right]  \\
+ \frac{C\pa{R', p, q, \beta}}{n\sqrt{\lambda_1 \ldots \lambda_p \mu_1 \ldots \mu_q}} \log\left( \frac{1}{\alpha} \right).
\end{multline}
where $C(\cdot)$ are positive constants depending only on their arguments.
\end{theorem}
As in Section \ref{Uniform-separation-rate}, we can deduce optimal bandwidths $(\lambda^*, \mu^*)$ which minimize the right-hand side of Equation \eqref{eq:majnik} and compute an upper bound for the uniform separation rate of the optimized test $\Delta_{\alpha}^{\lambda^*,\mu^*}$ over Nikol'skii-Besov balls.

\begin{corr}\label{cor:ThresholdNikol}
Let $\alpha$ in $(0,1/e)$, $\beta$ in $(0,1)$, $\delta = (\nu_1, \ldots , \nu_p, \gamma_1, \ldots ,\gamma_q)$ in $(0,2]^{p+q}$ and $R,R'>0$. 
Define for all $i$ in $\{1,\dots,p\}$ and for all $j$ in $\{1,\dots,q\}$, 
\begin{equation*}
\lambda_i^* = 
n^{- 2 \eta/[\nu_i (1 + 4 \eta)]} \quad \text{and} \quad \mu_j^* =  n^{- 2 \eta / [\gamma_j (1 + 4 \eta)]} \quad\quad \quad \mbox{where} \quad \frac{1}{\eta} = \sum_{i = 1}^p \frac{1}{\nu_i} + \sum_{j = 1}^q \frac{1}{\gamma_j}.
\end{equation*} 
If $n \geq \pa{\log(1/\alpha)}^{1+1/(4 \eta)}$, then, $(\lambda^*,\mu^*)$ satisfy $\boldsymbol{\mathcal{A}_2(\alpha)}$ and the uniform separation rate of the optimized test $\Delta_{\alpha}^{\lambda^*,\mu^*}$ over the Nikol'skii-Besov ball $\mathcal{N}^{\delta}_{2,p+q} (R,R')$ is controlled as follows 
\begin{equation}
\rho \left( \Delta_{\alpha}^{\lambda^*,\mu^*} , \mathcal{N}^{\delta}_{2,p+q} (R,R'), \beta \right) \leq \displaystyle C\pa{p, q,  \alpha, \beta, \delta,R,R'} n^{-2 \eta/(1 + 4 \eta)}.
\end{equation}
\end{corr}
Notice that the upper bound obtained for Nikol'skii-Besov balls in Corollary \ref{cor:ThresholdNikol} is analogue to that obtained for Sobolev balls in Corollary~\ref{cor:ThresholdSobolev}. Indeed, if we consider the same regularities in all directions in the case of Nikol'skii-Besov balls: $\nu_1 = \ldots = \nu_p = \gamma_1 = \ldots = \gamma_q$, we obtain a similar upper bound. 
These upper bounds obtained in Corollaries \ref{cor:ThresholdSobolev} and \ref{cor:ThresholdNikol} coincide with the asymptotic minimax separation rate of testing mutual independence w.r.t. the $\mathbb{L}_2$-norm over isotropic Nikol'skii-Besov spaces \cite{ingster1989asymptotically}. This suggests that the test $\Delta_{\alpha}^{\lambda^*,\mu^*}$ with optimal bandwidths is optimal in the minimax sense over Nikol'skii-Besov balls with regularity parameter $\delta$ in $(0,2]^{p+q}$. Yet, as in the Sobolev case, it cannot be adaptive since the optimal bandwidths $(\lambda^*,\mu^*)$ depend on the regularity~$\delta$.

Finally, note that subsequently, \cite{kim2020minimax} also generalized Theorem \ref{ThresholdNikol} to the permuted tests, which suggests that the permuted test with bandwidths $(\lambda^*,\mu^*)$ defined in Corollary \ref{cor:ThresholdNikol} is optimal in the minimax sense. However, as in the Sobolev case, they only obtain a polynomial dependence in $\alpha$ which is not sharp enough to provide adaptive tests by aggregating as done in Section \ref{unifseprateNikolagg}.

\subsection{Oracle-type conditions for the uniform separation rate over Nikol'skii-Besov balls}
\label{unifseprateNikoloracleagg}

Theorem \ref{theorem3Nikol} is equivalent to Theorem \ref{theorem3Sobol} over Nikol'skii-Besov balls and provides an oracle-type inequality for the uniform separation rate of the aggregated testing procedure $\Delta_{\alpha}$.

\begin{theorem}\label{theorem3Nikol} 
\begin{sloppypar}
Let $\alpha, \beta$ in $(0,1)$. Consider a finite or countable collection $\W\subset (0,+\infty)^p\times (0,+\infty)^q$ of bandwidths $(\lambda,\mu)$ and a collection of positive weights $\ac{\omega_{\lambda,\mu}}_{(\lambda,\mu) \in \W}$ such that $\sum_{(\lambda,\mu) \in \W} e^{- \omega_{\lambda,\mu}} \leq 1$ and such that 
all $(\lambda,\mu)$ in $\W$ verifies Assumption $\boldsymbol{\mathcal{A}_2(\alpha e^{-\omega_{\lambda,\mu}})}$. 
Then, the uniform separation rate over Nikol'skii-Besov balls $\mathcal{N}^{\delta}_{2,p+q} (R,R')$ with $\delta = (\nu_1, \ldots , \nu_p, \gamma_1, \ldots ,\gamma_q)$ in $(0,2]^{p+q}$ and $R,R'>0$ of the aggregated test $\Delta_{\alpha}$ defined in Equation \eqref{agregtest} can be upper bounded as follows
\end{sloppypar}
\begin{multline*}
\cro{\rho \left( \Delta_{\alpha}, \mathcal{N}^{\delta}_{2,p+q} (R,R') , \beta \right) }^2\leq \displaystyle C\pa{p, q, \beta, \delta,R,R'} \inf_{(\lambda,\mu) \in \W} \Bigg\lbrace 
\left[ \sum_{i=1}^p \lambda_i^{2 \nu_i} +  \sum_{j=1}^q \mu_j^{2 \gamma_j} \right] \\ 
+ \frac{1}{n\sqrt{\lambda_1 \ldots \lambda_p \mu_1 \ldots \mu_q}} \left( \log\left( \frac{1}{\alpha} \right) + \omega_{\lambda,\mu} \right) \Bigg\rbrace,
\end{multline*}
where $C(\cdot)$ is a positive constant depending only on its arguments.
\end{theorem}

As in the Sobolev case, Theorem \ref{theorem3Nikol} can be interpreted as an oracle-type condition for the uniform separation rate of the aggregated test $\Delta_{\alpha}$ over Nikol'skii-Besov balls. Indeed, without knowing the regularity $\delta$ of $f- f_1\otimes f_2 $, the uniform separation rate of $\Delta_{\alpha}$ is of the same order as the smallest uniform separation rate of the single tests corresponding to bandwidths $(\lambda,\mu)$ in $\W$, up to an additional term $\omega_{\lambda,\mu}$ due to the level corrections.

\subsection{Control of the uniform separation rate of the aggregated procedure}
\label{unifseprateNikolagg}

In this section, we provide an upper bound for the uniform separation rate of the aggregated testing procedure $\Delta_{\alpha}$ over Nikol'skii-Besov balls for the following specific choice of bandwidth collection and weights. Let
\begin{multline}\label{eq:defLambdaUNikol} 
\W = \Bigg\{ \pa{2^{-m_{1,1}}, \ldots ,2^{-m_{1,p}}, 2^{-m_{2,1}}, \ldots ,2^{-m_{2,q}}}, \\ 
\pa{m_{1,1}, \ldots ,m_{1,p}, m_{2,1}, \ldots ,m_{2,q}}\in (\mathbb{N}^*)^{p+q}\ ;\  
\sum_{i=1}^p m_{1,i} + \sum_{j=1}^q m_{2,j} \leq  2 \log_2\pa{\frac n {\log(n)} } \Bigg\} ,
\end{multline}
In addition, we associate to every bandwidths $(\lambda, \mu) = (2^{-m_{1,1}}, \ldots ,2^{-m_{1,p}}, 2^{-m_{2,1}}, \ldots ,2^{-m_{2,q}})$ in $\W$ the positive weight 
\begin{equation}\label{eq:weigthsNikol}
\omega_{\lambda, \mu} = 2 \sum_{i=1}^p \log\pa{m_{1,i} \times \frac{\pi}{\sqrt{6}}} + 2 \sum_{j=1}^q \log\pa{m_{2,j} \times \frac{\pi}{\sqrt{6}}},
\end{equation} 
so that $\sum_{(\lambda,\mu) \in \W} e^{- \omega_{\lambda,\mu}} \leq 1$. 
 
\begin{corr}\label{corr1Nikol}
Let $\alpha, \beta$ in $(0,1)$. Consider the aggregated test $\Delta_{\alpha}$ defined in \eqref{agregtest}, with the particular choice of the collection $\W$ and the weights $\left( \omega_{\lambda, \mu} \right)_{(\lambda, \mu) \in \W}$ defined in \eqref{eq:defLambdaUNikol} and \eqref{eq:weigthsNikol}. Assume that $\log \log (n) > 1$. Then, under the assumptions of Theorem \ref{theorem3Nikol}, for any  $\delta = (\nu_1, \ldots , \nu_p, \gamma_1, \ldots ,\gamma_q)$ in $(0,2]^{p+q}$ and positive radii $R,R'$, there exists a positive constant $C(p,q,\alpha,\delta)$ such that for all $n \geq C(p,q,\alpha,\delta)$, the uniform separation rate over the Nikol'skii-Besov ball $\mathcal{N}^{\delta}_{2,p+q} (R,R')$ of $\Delta_{\alpha}$ can be upper bounded as follows:
\begin{equation*}
\rho \left( \Delta_{\alpha} , \mathcal{N}^{\delta}_{2,p+q} (R,R'), \beta \right) \leq \displaystyle C\pa{p, q, \alpha, \beta, \delta,R,R'} \left( \frac{\log \log (n)}{n} \right)^{2 \eta / (1 + 4 \eta)}, 
\end{equation*}
where $\displaystyle  \frac{1}{\eta} = \displaystyle \sum_{i = 1}^p \frac{1}{\nu_i} + \sum_{j = 1}^q \frac{1}{\gamma_j}$. 
\end{corr}

As in the case of Sobolev regularity, according to Corollary \ref{corr1Nikol}, the uniform separation rate of the aggregated procedure over Nikol'skii-Besov balls is of the same order as the one of the optimized test $\Delta_{\alpha}^{\lambda^*,\mu^*}$ (given in Corollary \ref{cor:ThresholdNikol}), up to a $\log \log(n)$ factor which is, once again a usual price to pay for aggregated tests (see, e.g., \cite{spokoiny1996adaptive,ingster2000adaptive}.)

\section{Further numerical simulations}
\label{sect:NumericalSimulationsbis}

\subsection{Single tests comparison}
\label{Single_permuted_tests}

Similarly to Section \ref{sect:comparTheoPerm}, the objective here is to check that the permutation approach does not impact the power of the single HSIC test. To do so, we numerically illustrate that the power of the permuted single HSIC tests approximates very well the power of the theoretical tests, as soon as enough permutations are used for the estimation of the quantile under the null hypothesis. \\

In order to evaluate the accuracy of permuted single HSIC tests, we choose the kernel bandwidth associated to $X$ (resp. $Y$) to be the empirical standard deviation $s$ (resp. $s'$) of $X$ (resp. $Y$), which is a usual choice in the literature on single HSIC-test (see, e.g. \cite{de2017sensitivity}). 

As in Section \ref{sect:comparTheoPerm}, we rely on the data generating mechanism inspired from the Ishigami function \cite{ishigami1990importance} defined in \eqref{eq:modelM}. 
In the following, we illustrate the power for the three sample sizes $n$ in $\ac{50,100,200}$ and the two levels $\alpha$ in $\ac{0.05,0.001}$. 

For each sample size $n$ and level $\alpha$, we first estimate the power of the \emph{theoretical} test. 
To achieve this, we simulate 500.000 $n$-samples under the null hypothesis\footnote{To generate an independent $n$-sample of $(X,Y)$ under the null hypothesis, we first generate an independent $2n$-sample of $(X,Y)$ according to \eqref{eq:modelM}. Only the first $n$ elements are used to compute the marginal sample of $Y$ and the remaining $n$ elements are considered to be the marginal sample of $X$.} and compute the Monte Carlo estimator, denoted $\tilde{q}^{MC}_{1-\alpha}$, of the theoretical $(1-\alpha)$-quantile of $\widehat{\HSIC}_{s,s'}$ under the null hypothesis. 
Then, we generate 1000 different $n$-samples of $(X,Y)$ under the alternative according to \eqref{eq:modelM} and we estimate the power of the \emph{theoretical} test by $\hat\pi_{th} (n,\alpha)$ which is the ratio of times that the observed test statistic $\widehat{\HSIC}_{s,s'}$ exceeds the quantile $\tilde{q}^{MC}_{1 - \alpha}$. 

The second step consists in estimating the power of the \emph{permuted} tests for several values of the number of permutations $B$. The chosen values of $B$ are $\ac{10, 20, \ldots, 100, 200, \ldots, 2500}$. For each value of $n$, $\alpha$ and $B$, we generate 1000 $n$-sample of $(X,Y)$ according to \eqref{eq:modelM}. 
For each $n$-sample, we compute the permuted quantile $\hat{q}_{1-\alpha}$ defined in Equation \eqref{eq:permquant} using $B$ random permutations of this sample. Thereafter, we estimate the power of the \emph{permuted} test, by $\hat\pi(n,\alpha,B)$ which is the ratio of times the value of $\widehat{\HSIC}_{s,s'}$ exceeds the permuted quantile $\hat{q}_{1-\alpha}^{s,s'}$ (computed on the corresponding sample). 

As in Section \ref{sect:comparTheoPerm}, to compare the empirical powers of \emph{theoretical} and \emph{permuted} tests (resp. $\hat\pi_{th} (n,\alpha)$ and $\hat\pi(n,\alpha,B)$), we consider the relative absolute error $Err(n,\alpha,B)$ defined as $$Err(n,\alpha,B) = \frac{\abs{\hat\pi(n,\alpha,B) - \hat\pi_{th} (n,\alpha)}}{\hat\pi_{th} (n,\alpha)}.$$ 

\begin{figure}
 \centering
 \includegraphics[width=0.3\textwidth]{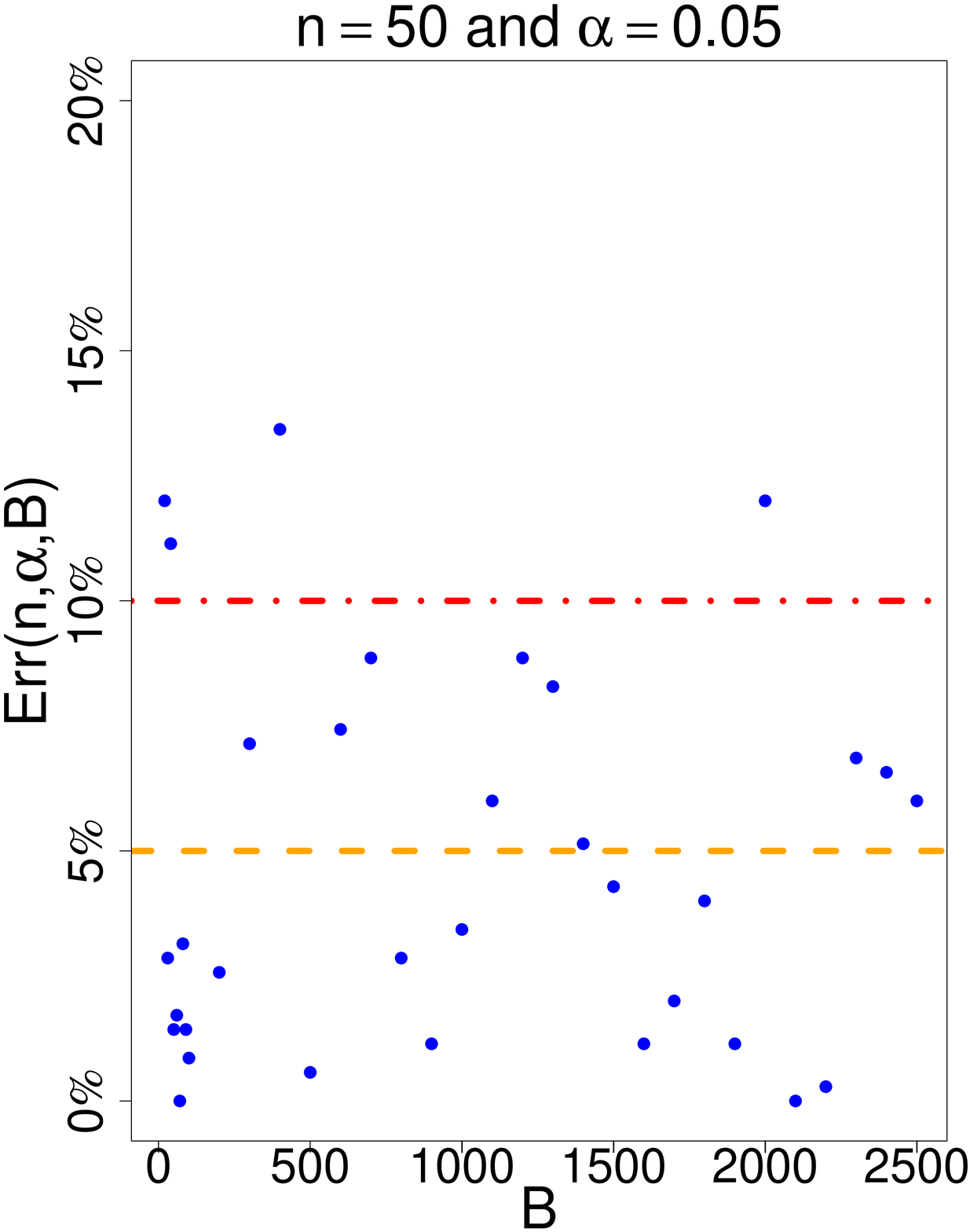}
 \includegraphics[width=0.3\textwidth]{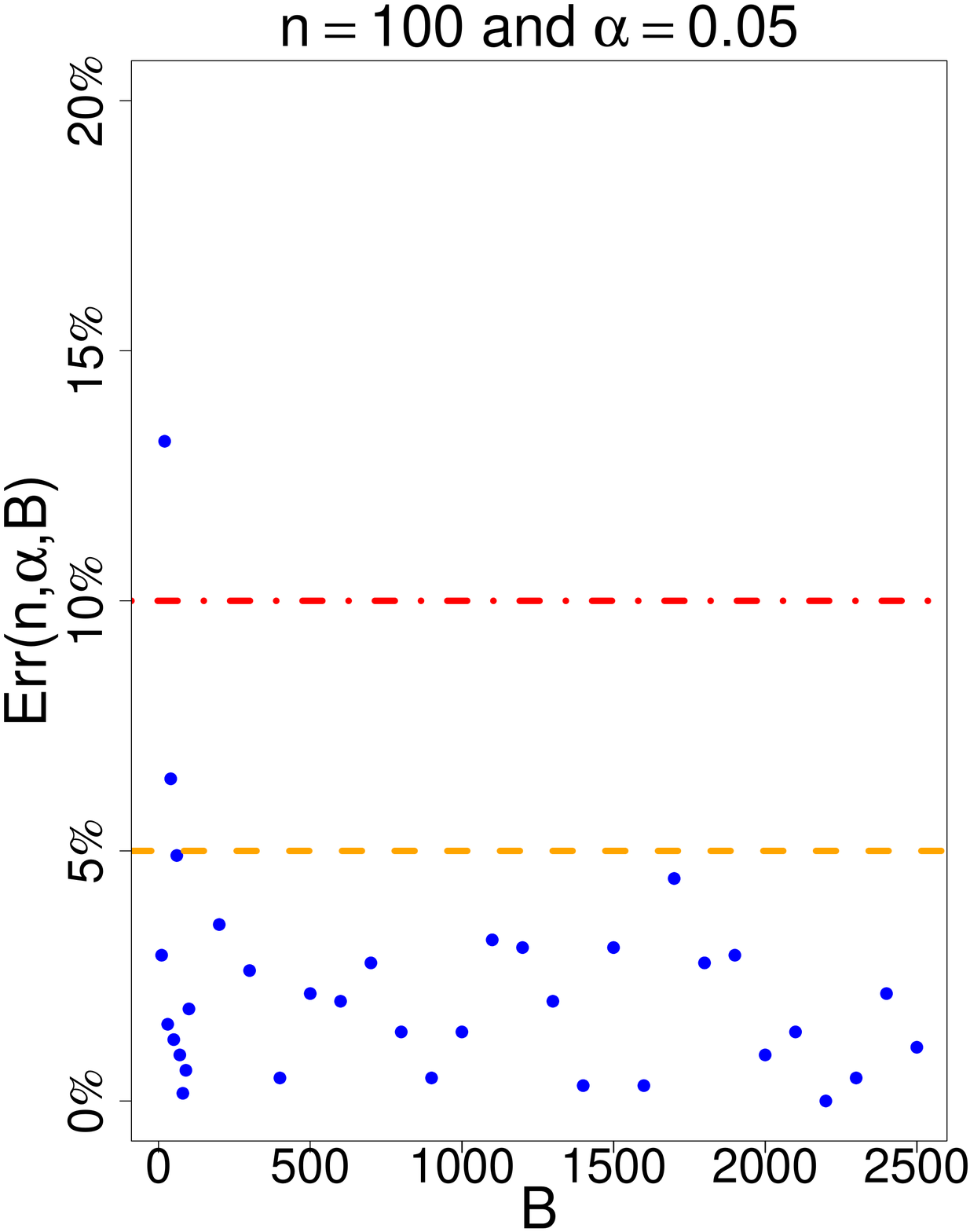}
 \includegraphics[width=0.3\textwidth]{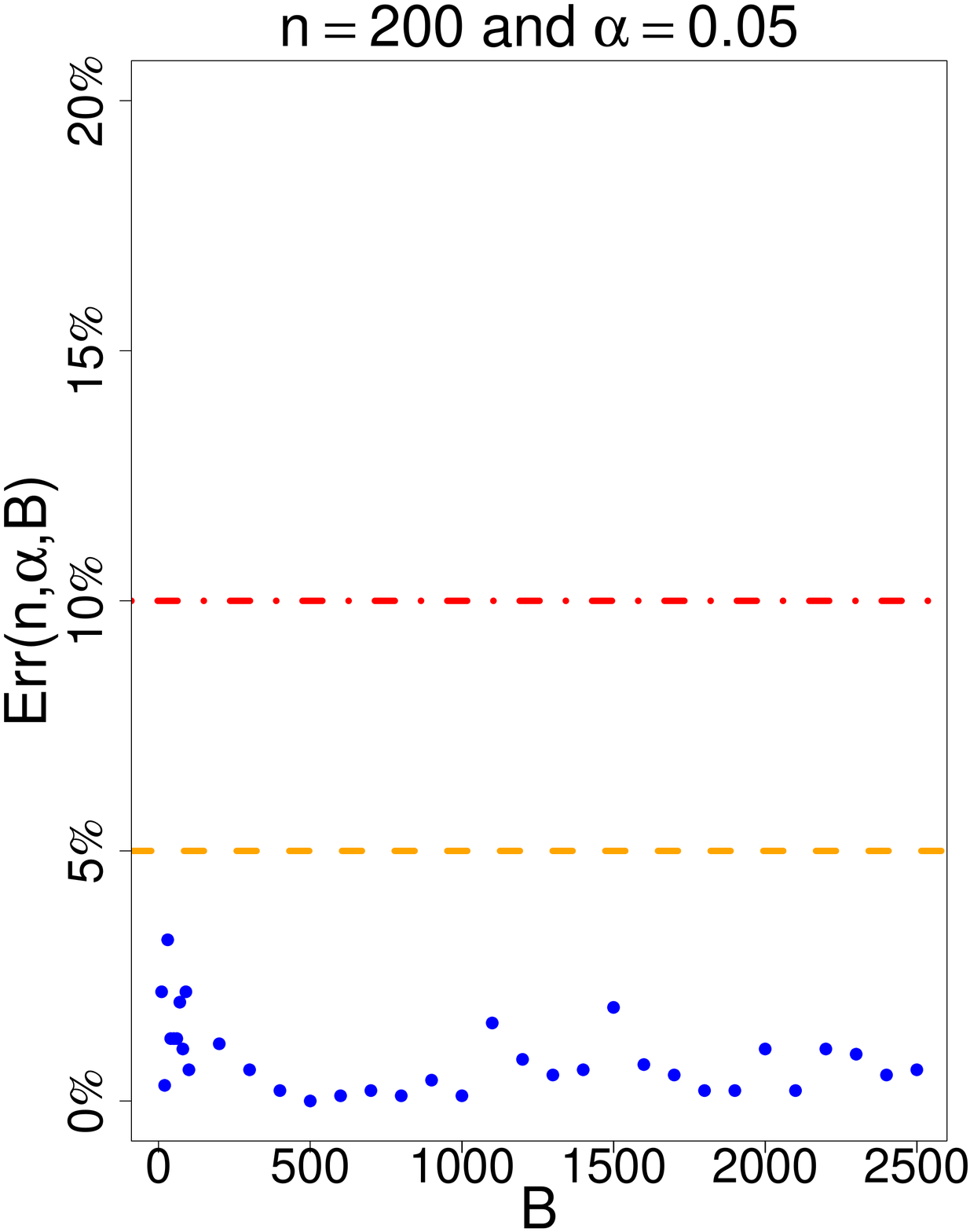}
 \caption{Absolute relative error between the empirical powers of the theoretical and permuted HSIC-tests, w.r.t the number $B$ of permutations, for samples generated according to Ishigami's data generating mechanism defined in Equation \eqref{eq:modelM} with sizes $n= 50$, $100$ and $200$. The presumed level is $\alpha = 0.05$. The red (resp. orange) dashed line represents the error threshold of $10 \%$ (resp. $5 \%$).}
 \label{fig:PowerPermutedalpha0.05}
\end{figure}

The results obtained for $\alpha = 0.05$ and different $n$ values are given by Figure \ref{fig:PowerPermutedalpha0.05}. We can see that the accuracy of the permuted approach tends to increase as $n$ increases. This is probably due to the fact that the power of the theoretical test increases as the sample size increases. Another explanation may be that, on the one hand, the power of the theoretical test is more difficult to estimate for small sample sizes, which explains the fluctuations observed for $n=50$. On the other hand, as $n$ increases, the approximation of the distribution of $\widehat{\HSIC}_{s,s'}$ under the null hypothesis based on $B$ permutations becomes more accurate, and this for any value of $B$ larger than $500$. 
Hence, the approximation of the quantile by permutation becomes more accurate, and thus, there are less fluctuations for larger sample sizes. 

Generally, the permutation approach allows to obtain the power of the theoretical test with an acceptable precision, even for small values of $B$. In particular, we observe for $n=50$ that aside from very small values of $B$ and two outliers, the absolute relative error is always less than $10 \%$. Moreover, from $n=100$ this error is mostly less than $10 \%$ and no observed error is greater than $5 \%$ for $n = 200$.

\begin{figure}
 \centering
 \includegraphics[width=0.3\textwidth]{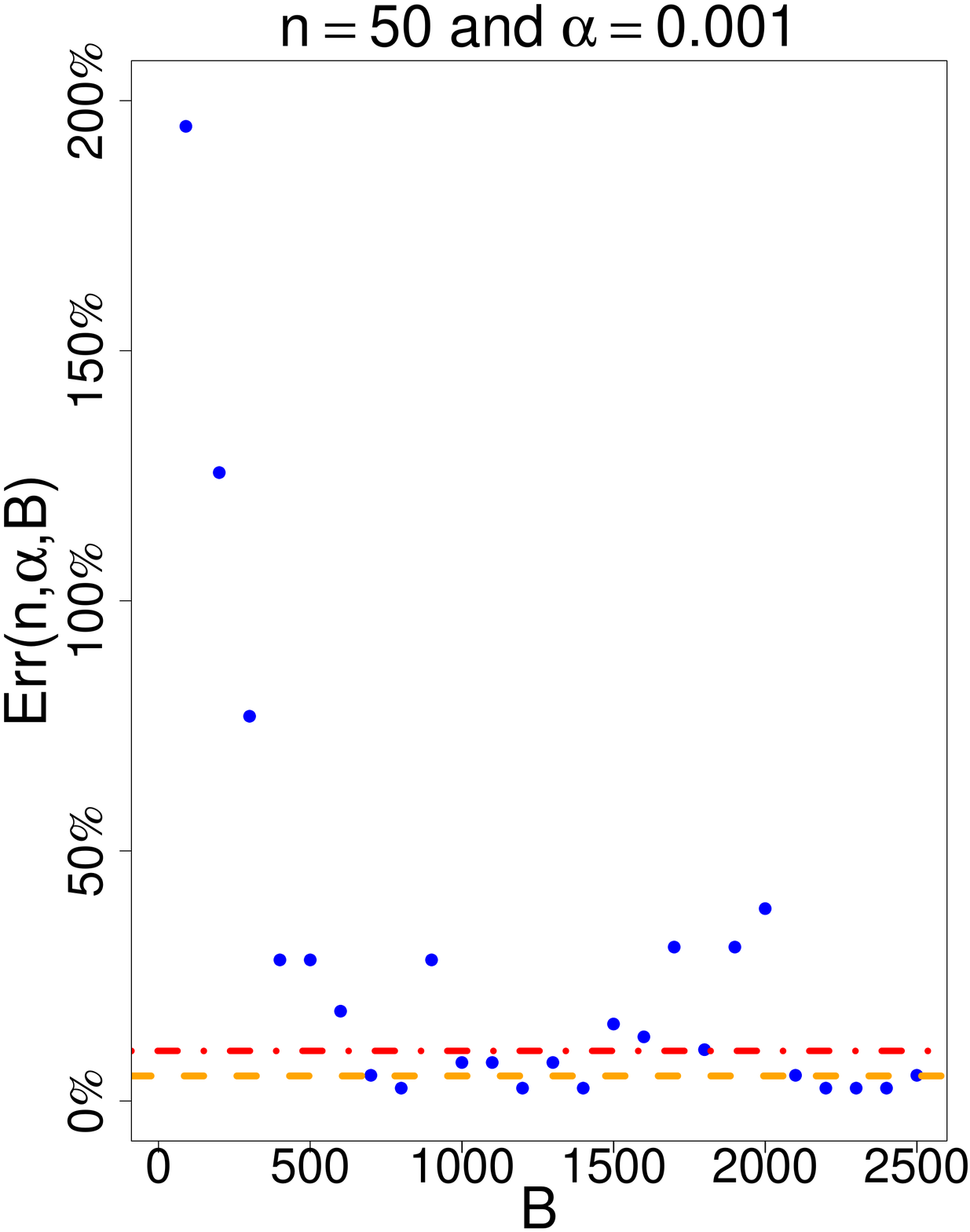}
 \includegraphics[width=0.3\textwidth]{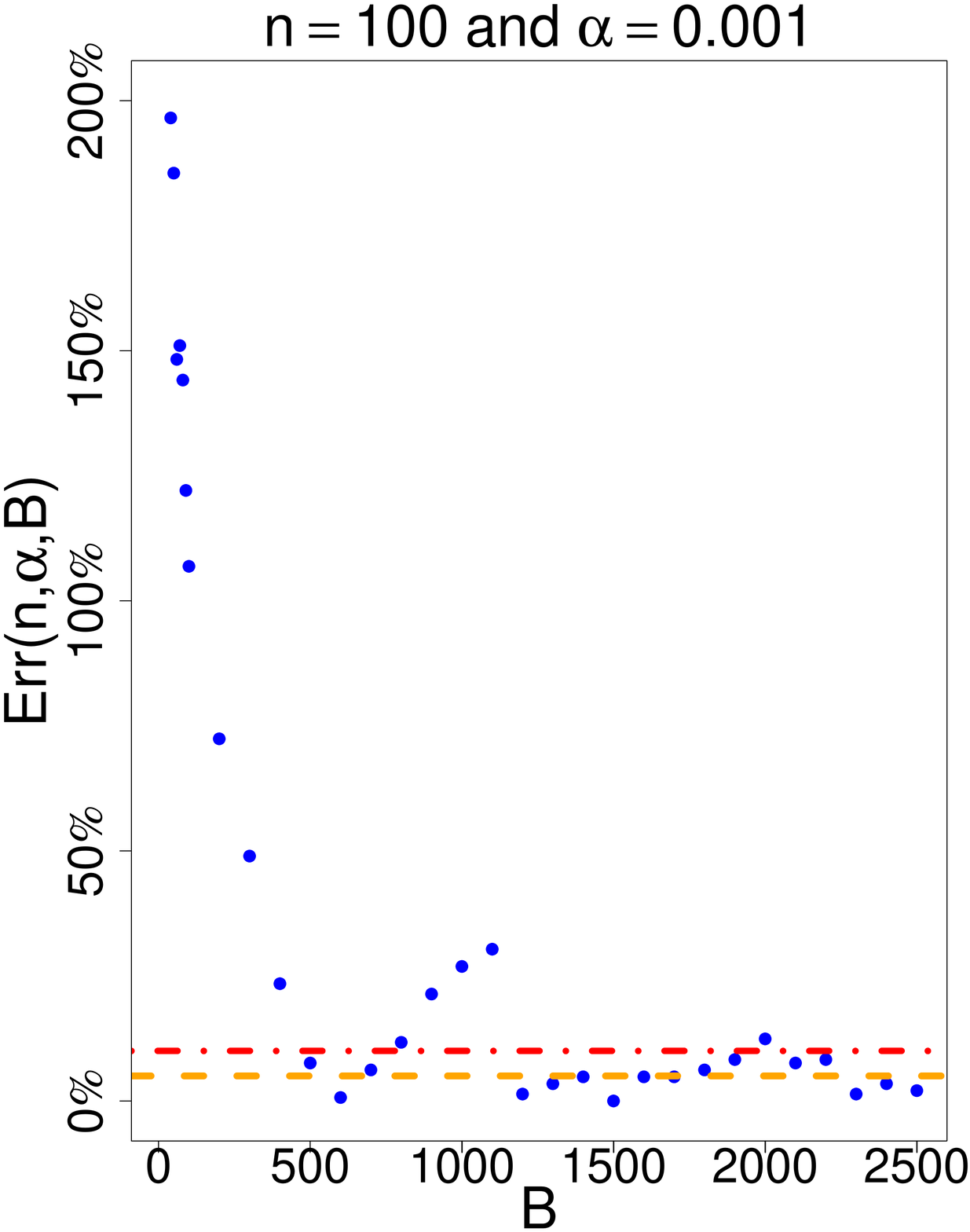}
 \includegraphics[width=0.3\textwidth]{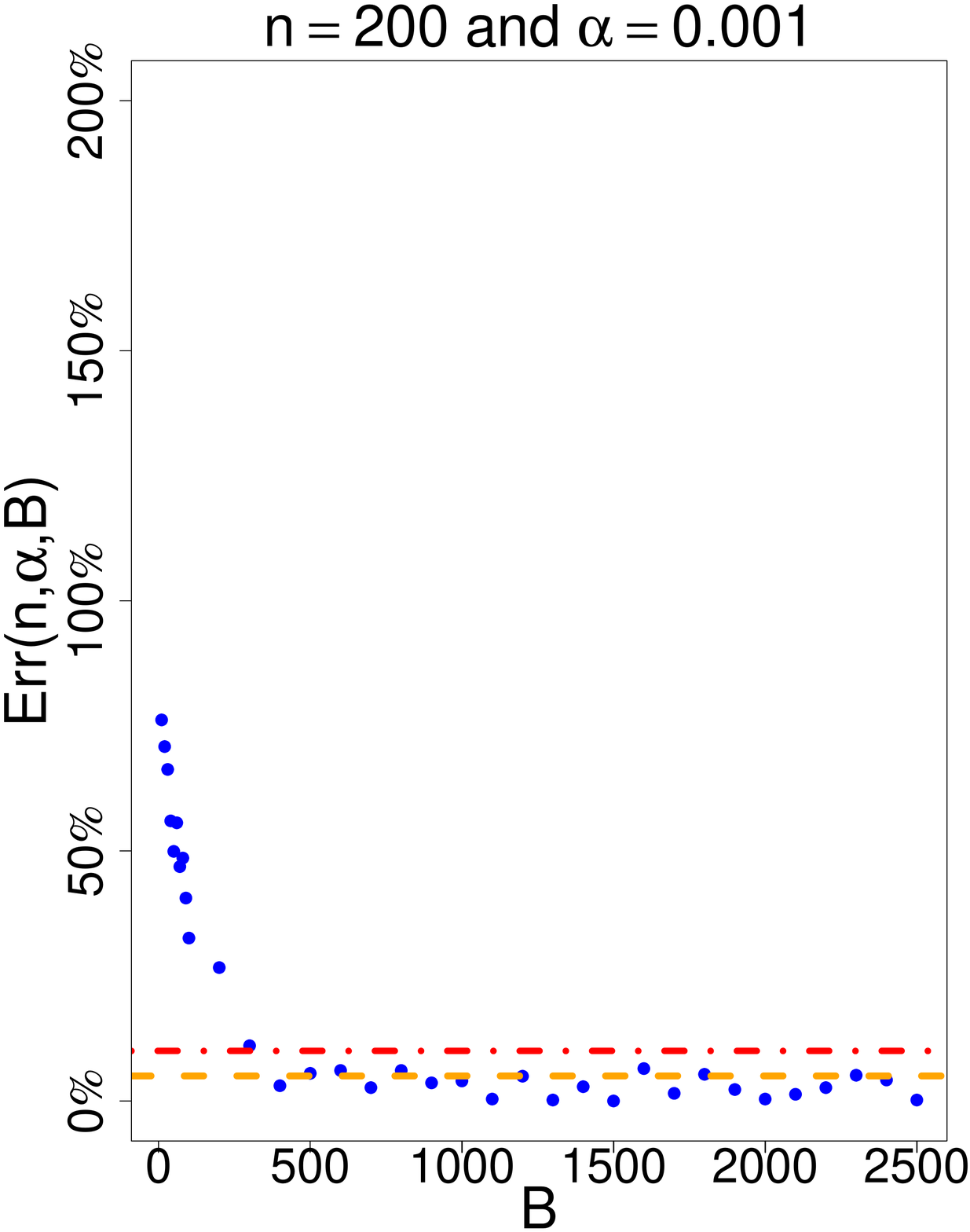}
 \caption{Absolute relative error between the empirical powers of the theoretical and the permuted HSIC-tests, w.r.t the number $B$ of permutations, for samples generated according to Ishigami's data generating mechanism defined in Equation \ref{eq:modelM} of Section \ref{sect:comparTheoPerm} with sizes $n= 50$, $100$ and $200$. The presumed level of tests is $\alpha = 0.001$. The red (resp. orange) dashed line represents the error threshold of $10 \%$ (resp. $5 \%$).}
 \label{fig:PowerPermutedalpha0.001}
\end{figure}

Since the aggregated procedure requires an individual level correction, we also study the impact of the level on the accuracy of the permutation approximation. We show in Figure \ref{fig:PowerPermutedalpha0.001} the relative absolute error of the power w.r.t. $n$ and $B$ for the extreme level value $\alpha = 0.001$. Contrary to the case $\alpha = 0.05$, we observe here much less precision of the power approximation. In particular, for $n=50$, $B = 2000$ permutations are required to obtain satisfactory accuracy (against $B = 30$ for $\alpha = 0.05$). Similar observations are done for $n=100$ and $200$ with respectively $B = 1200$ and $B = 500$ permutations required (against $B = 30$ and $B = 10$ for $\alpha = 0.05$). This slow convergence results from the difficulty of estimating extreme quantiles. Moreover, this phenomenon seems more significant for small sample sizes. Indeed, as in the previous case, the lowest the power of the test, the biggest its sensitivity to the quantile estimation error.

Similar results for the aggregated procedures are illustrated in Section \ref{sect:comparTheoPerm} of the main article. 

\subsection{Choice of the collections of bandwidths and the weights}
\label{sect:choicecollweights}

In our aggregated procedure, the collection of bandwidths $\W$, together with the weights have to be chosen. 
There is no universal best collection that would ensure optimal test power. To determine the collection, we first study the impact of the bandwidth choice on single HSIC-based tests. This leads us to particular forms of collections. 
Then, we investigate different choices of the collections $\W$ and together with different weights (including the single test case).

\subsubsection{Impact of the bandwidths choice on the power of the single tests} 
\label{Investigation_bandwidth}

The optimal bandwidth depends on the intrinsic characteristics of $X$ and $Y$ and their dependence structure. Consequently, it seems relevant to consider the possible bandwidths relatively to the standard deviations of $X$ and $Y$. Moreover, as already mentioned, the standard deviation is a usual choice for the bandwidth in the literature on single HSIC-test. We assume here that the exact values of standard deviations of $X$ and $Y$, respectively denoted $s$ and $s'$, are known. In such a way, we are able to construct collections which do not depend on the observation. In practice, when only a $n$-sample of $(X,Y)$ is available, we estimate these standard deviations by the usual empirical estimators. Practice shows that the effect of this estimation does not significantly impact the single tests performance. Indeed, standard deviation estimators converge in most cases rapidly w.r.t. $n$. More particularly, this estimation error is small compared to the estimation error of the quantiles. \\ 

For this, we consider the univariate mechanism of dependence \ref{Berrettex2} with $l = 2$ defined in Section \ref{Comparison_MINT}. Moreover, we consider, as possible bandwidths $\lambda$ and $\mu$, multiple or dyadic fractions of $s$ and $s'$ respectively. For each couple $(\lambda ,\mu)$, the power of the permuted single HSIC tests (with $B=1000$) is estimated as explained above. Figure \ref{fig:HeatMap} shows the obtained power maps w.r.t. $(\lambda ,\mu)$, for different sample sizes. First, we can observe that the bandwidths significantly impact the power: in this case, there is an optimal area around $(\lambda,\mu) = (s/4,s'/4)$ with a power close to one for $n = 200$. The power decreases progressively as we move away from this area, until being null for very high and very low values of bandwidths. We can also see that the regularity of the maps increases with the sample size (just like the power  for each point). Similar conclusions can be observed for other values of $l$ and the other data generating mechanisms \ref{Berrettex1} and \ref{Berrettex3} with one or several areas with higher power, but are not presented here. 
These results illustrate that an arbitrary choice of bandwidths is not relevant and justify the interest of considering several bandwidths through an aggregation strategy. Note that, according to our experience, it might be appropriate to consider bandwidths higher than standard deviations. However, in Section \ref{sect:impactweights}, we consider aggregating procedures based on collections $\W_r^{s,s'}$ of types 
\begin{equation}\label{CollectionsChoice}
\W_r^{s,s'} = \left\{ s, s/2, \ldots, s/2^{r-1} \right\} \times \left\{ s', s'/2, \ldots, s'/2^{r-1} \right\}, 
\end{equation}
where $r$ belongs to $\mathbb{N}^*$. Note that in this univariate case, these collections generalize to other sizes $r$, in an anisotropic way, the ones considered in Section \ref{Comparison_MINT} introduced in Equation \eqref{eq:collectioncompartestsunidim6}.
\begin{figure}
 \centering
    \includegraphics[width=0.32\textwidth]{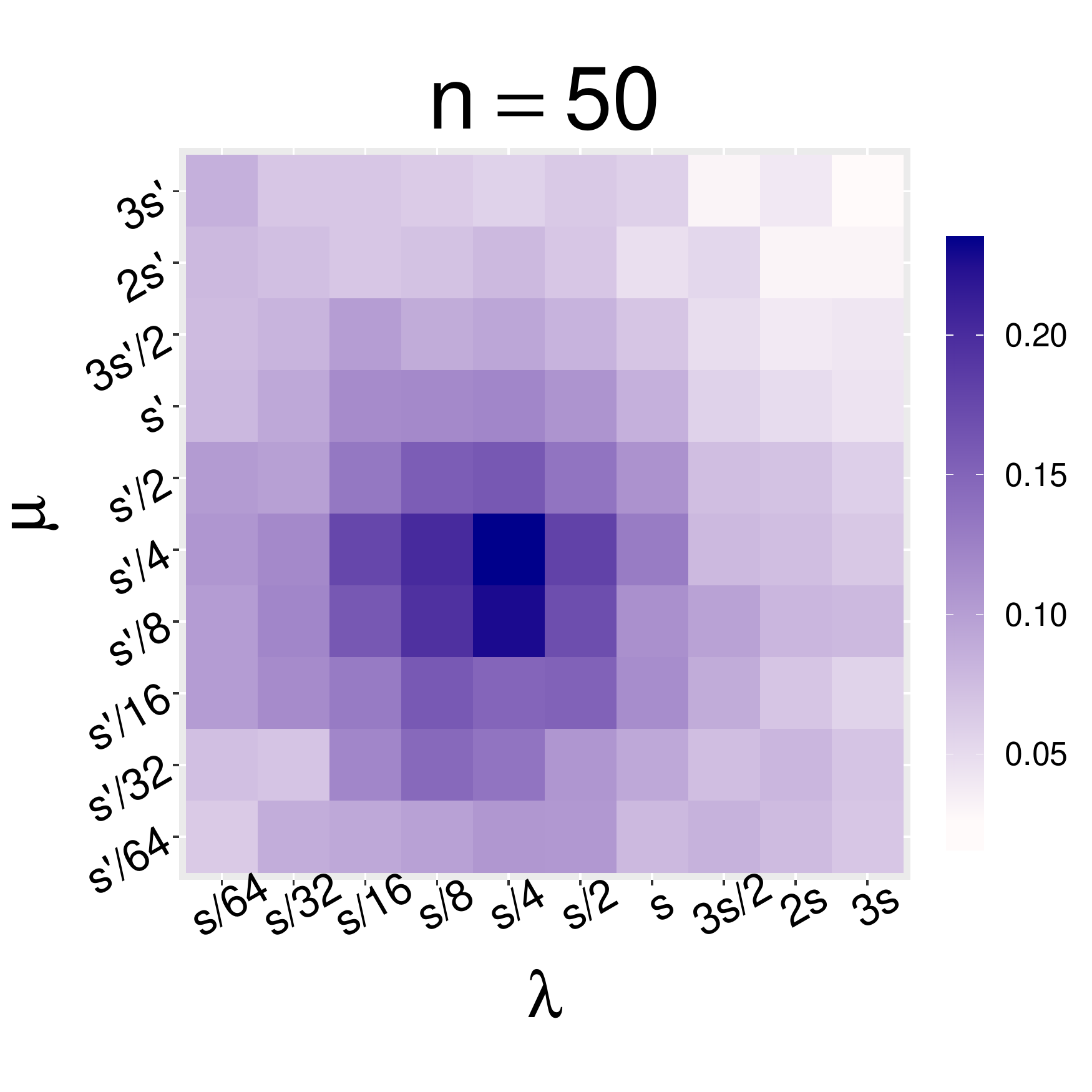} %
    \hspace{-0.25cm}
    \includegraphics[width=0.32\textwidth]{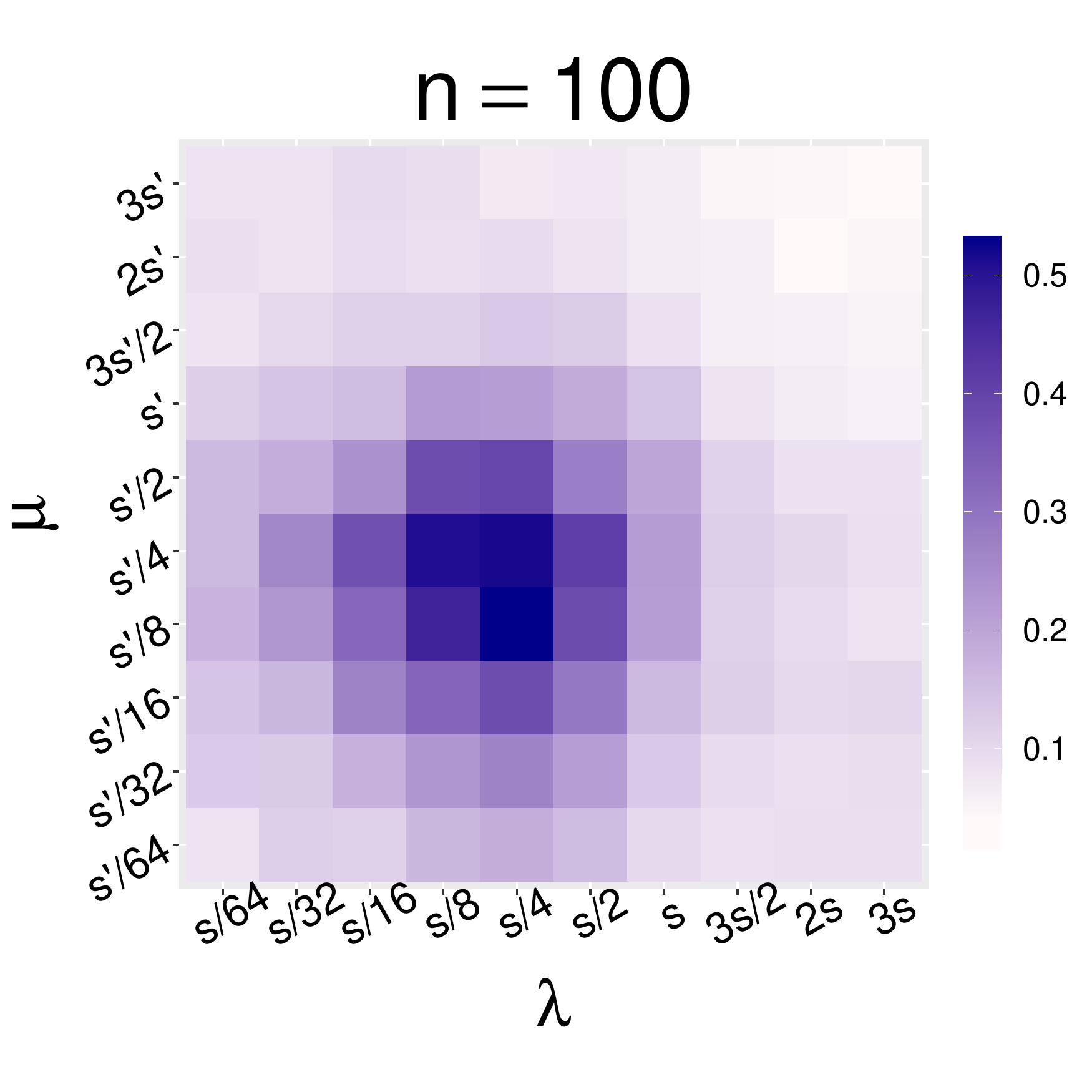}
    \hspace{-0.25cm}
    \includegraphics[width=0.32\textwidth]{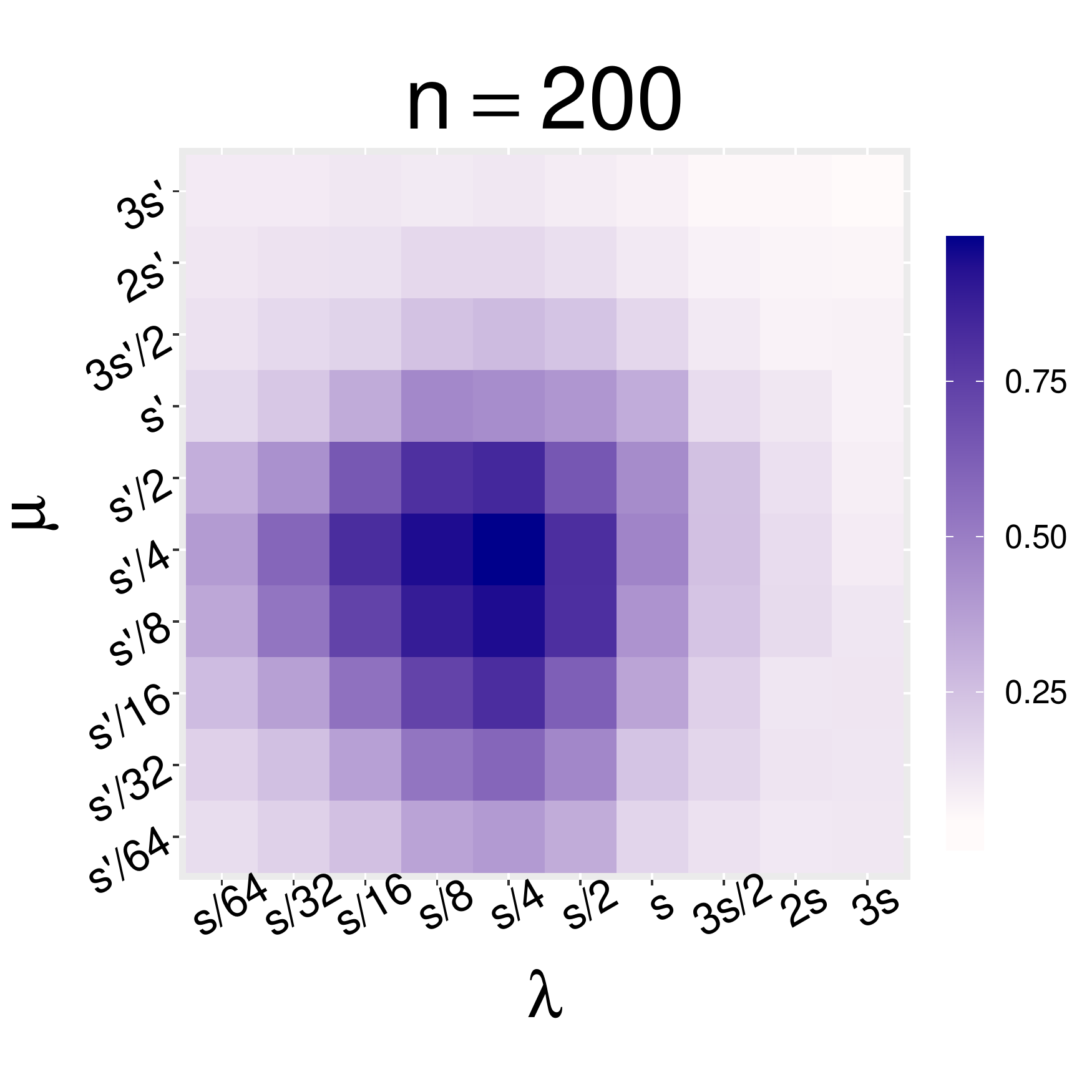}
    \caption{Power map of the permuted single HSIC test w.r.t. to kernel bandwidths $\lambda$ and $\mu$ respectively associated to $X$ and $Y$, for sample generated according to the univariate mechanism of dependence \ref{Berrettex2} with $l = 2$ defined in Section \ref{Comparison_MINT} with sizes $n = 50$, $100$ and $200$, $B=1000$ and $\alpha=0.05$.}
    \label{fig:HeatMap}
\end{figure}

\subsubsection{Impact of the weights choice on the power of the aggregated procedure}
\label{sect:impactweights}

Following the results of Section \ref{Investigation_bandwidth}, we consider bandwidth collections $\W_r^{s,s'}$ as defined in Equation \eqref{CollectionsChoice}, where $s$ and $s'$ are respectively the empirical standard deviations of the $X_i$'s and the $Y_i$'s. 
By now, let us compare two possible choices of weights: uniform and exponential weights. On the one hand, we recall that uniform weights depend only on the cardinalitly of the collection, and are defined in Equation \eqref{eq:defPoidsUnif} for all $(\lambda, \mu)$ in $\W_r^{s,s'}$ by 
$$\omega_{\lambda,\mu} =  \log(r^2).$$
On the other hand, in analogy with Equation \eqref{eq:weigthsNikol}, we consider the exponential weights defined for all bandwidths $(s/2^{m_1}, s'/2^{m_2})$ in $\W_r^{s,s'}$ by 
$$ \omega_{s/2^{m_1}, s'/2^{m_2}} = 2\log\pa{m_1+1} + 2 \log\pa{m_2+1} + \log \left( \sum_{ 1 \leq u,v \leq r} \frac{1}{u^2 v^2} \right). $$

\paragraph{}
The results obtained with the two types of weights are given in Figure \ref{fig:PowerGraph}, for different values of $r$ and sample sizes $n$. In this case, the uniform weights seem to give a better power than the exponential ones. However, we can observe a different behavior w.r.t. $r$. For the uniform weights, the power increases until a specific $r$ ($r=3$ or $4$ w.r.t $n$), before decreasing with $r$, to being lower than the power with exponential weights. On the contrary, the power with exponential weights has a more robust behavior, since it increases with $r$ until it stabilizes. This is a crucial advantage in favor of exponential weights, as the optimal $r$ is unknown in practice. It prevents deterioration of the quality of the test, when too large collection sizes have been chosen. We can also observe that the two aggregated strategies yield a greater power than the single test (which corresponds to the case $r=1$), as soon as the collection $\W$ is large enough. \\

Similar conclusions have been drawn from the other analytical examples, which are not presented here for the sake of brevity. Thus, from our experience, we recommend in practice the use of the aggregated procedure with exponential weights with $r = 5$ or $6$.
\begin{figure}
 \centering
    \includegraphics[width=0.75\textwidth]{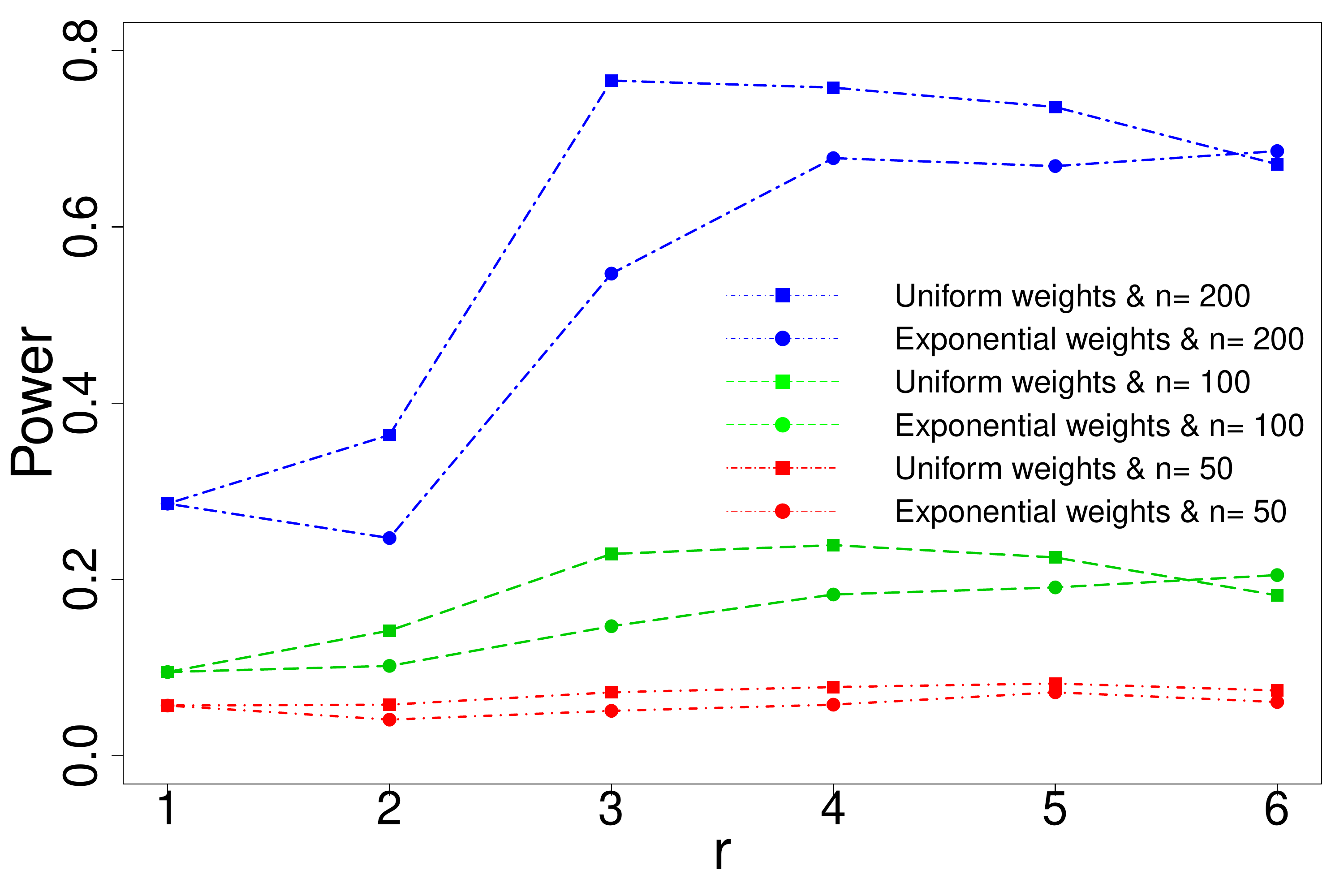}
    \caption{Empirical power of the permuted aggregated procedures with uniform and exponential weights, w.r.t. the number $r$ of aggregated bandwidths in each direction, for samples generated according to the univariate mechanism of dependence \ref{Berrettex2} with $l = 2$ defined in Section \ref{Comparison_MINT} of sizes $n=50$, $100$ and $200$, $B_1=3000$, $B_2=500$ and $\alpha=0.05$.}
    \label{fig:PowerGraph}
\end{figure}

\section{Proofs}
\label{sect:proofs} 

All along the proofs, we set $Z = (X,Y)$ and $Z_i = (X_i,Y_i)$ for all $i$ in $\{1, \ldots ,n\}$. We also denote by $A, B$ and $C$ positive universal constants whose values may change from line to line. Moreover, the generic notation $C(a,b, \ldots)$ denotes a positive constant depending only on its arguments $(a,b, \ldots)$ and that may vary from line to line. 


\subsection{Proof of Proposition \ref{levelapprxtest}}

Let $\alpha$ be in $(0,1)$. In order to prove that the permuted test with Monte Carlo approximation $\widehat{\Delta}_{\alpha}^{\lambda,\mu}$ defined in Equation \eqref{approxtest} is of prescribed level $\alpha$, we use Lemma 1 of \cite{romano2005exact} recalled here.
\begin{lemm}[{\cite[Lemma 1]{romano2005exact}}]\label{Lemm:RomanoWolf}
Let $R_1$, \ldots, $R_{B+1}$ be $(B+1)$ exchangeable random variables. Then, for all $u$ in $(0,1)$
\begin{equation*}
\proba{\frac{1}{B+1}\cro{1 + \sum_{b = 1}^B \mathds{1}_{R_b \geq R_{B+1}}} \leq u} \leq u. 
\end{equation*}
\label{Romano} 
\end{lemm}

\medskip

Recall that for all $1\leq b\leq B$, 
$$\widehat{H}^{\star b}_{\lambda,\mu} = \widehat{\HSIC}_{\lambda,\mu}\pa{\ZZ_n^{\rperm_b}}\quad \mbox{and} \quad \widehat{H}^{\star B+1}_{\lambda,\mu} = \widehat{\HSIC}_{\lambda,\mu}\pa{\ZZ_n} = \widehat{\HSIC}_{\lambda,\mu}\pa{\ZZ_n^{\rperm_{B+1}}},$$
where $\rperm_{B+1}=\id$ is the identity permutation of $\{1,\ldots,n\}$ (deterministic). 

Assume that $f = f_1 \otimes f_2$. Then the random variables $\widehat{H}^{\star 1}_{\lambda,\mu}, \ldots , \widehat{H}^{\star B}_{\lambda,\mu}$ and $\widehat{H}^{\star B+1}_{\lambda,\mu}$ are exchangeable. 
Indeed, let $\echperm$ be a (deterministic) permutation\ of $\{1,\ldots,B+1\}$ and let us prove that
\begin{equation}
\label{eq:echangeabilite}
\pa{\widehat{H}^{\star 1}_{\lambda,\mu}, \ldots , \widehat{H}^{\star B}_{\lambda,\mu}, \widehat{H}^{\star B+1}_{\lambda,\mu}}\quad \mbox{and}\quad \pa{\widehat{H}^{\star \echperm(1)}_{\lambda,\mu}, \ldots , \widehat{H}^{\star \echperm(B+1)}_{\lambda,\mu}} \quad \mbox{have the same distribution.}
\end{equation}

\medskip
\textbf{Case 1.} If $\echperm(B+1)=B+1$, then, since the permutations $(\rperm_b)_{1\leq b\leq B}$ are i.i.d., they are exchangeable. Hence, $(\rperm_{\echperm(1)},\ldots,\rperm_{\echperm(B)})$ is an i.i.d. sample of uniform permutations of $\{1,\dots,n\}$, independent of $\ZZ_n$ and \eqref{eq:echangeabilite} holds by construction.

\medskip
\textbf{Case 2.} If $\echperm(B+1)\neq B+1$, then 
$$\widehat{H}^{\star \echperm(B+1)}_{\lambda,\mu} = \widehat{\HSIC}_{\lambda,\mu}\pa{\ZZ_n^{\rperm_{\echperm(B+1)}}} = \widehat{\HSIC}_{\lambda,\mu}\pa{\tilde{\ZZ}_n}, \quad \mbox{where}\quad \tilde{\ZZ}_n = \ZZ_n^{\rperm_{\echperm(B+1)}}.$$

In particular, for all $b$ in $\{1,\ldots,B\}$, 
$$ \left\{
\begin{array}{ll}
\displaystyle\widehat{H}^{\star \echperm(b)}_{\lambda,\mu} = \widehat{\HSIC}_{\lambda,\mu}\pa{\ZZ_n^{\rperm_{\echperm(b)}}} = \widehat{\HSIC}_{\lambda,\mu}\pa{\tilde{\ZZ}_n^{\rperm_{\echperm(b)}\circ \rperm_{\echperm(B+1)}^{-1}}} & \mbox{if } \echperm(b)\neq B+1, \\
\displaystyle\widehat{H}^{\star \echperm(b)}_{\lambda,\mu} = \widehat{\HSIC}_{\lambda,\mu}\pa{\ZZ_n} = \widehat{\HSIC}_{\lambda,\mu}\pa{\tilde{\ZZ}_n^{\id\circ\rperm_{\echperm(B+1)}^{-1}}} & \mbox{if } \echperm(b) = B+1. 
\end{array}
\right.$$
Therefore, in order to prove \eqref{eq:echangeabilite}, it is sufficient to prove that $\{\rperm_{\echperm(1)}\circ \rperm_{\echperm(B+1)}^{-1},\ldots,\rperm_{\echperm(B)}\circ \rperm_{\echperm(B+1)}^{-1}\}$ is an i.i.d. sample of uniform permutations of $\{1,\dots,n\}$ independent of $\tilde{\ZZ}_n$. 
Let $A$ be a mesurable set, and $\perm_1, \ldots,\perm_B$ be (fixed) permutations of $\{1,\dots,n\}$. Then 
\begin{align*}
\mathbb{P} \Big(\tilde{\ZZ}_n \in A,\ & \rperm_{\echperm(1)}\circ \rperm_{\echperm(B+1)}^{-1}=\perm_1,\ \ldots,\ \rperm_{\echperm(B)}\circ \rperm_{\echperm(B+1)}^{-1}=\perm_B\Big) \nonumber\\
& = \proba{\ZZ_n^{\rperm_{\echperm(B+1)}} \in A, \rperm_{\echperm(1)}=\perm_1\circ \rperm_{\echperm(B+1)},\ldots,\rperm_{\echperm(B)}=\perm_B\circ \rperm_{\echperm(B+1)}} \nonumber\\
& = \esp{\proba{\ZZ_n^{\rperm_{\echperm(B+1)}} \in A, \rperm_{\echperm(1)}=\perm_1\circ \rperm_{\echperm(B+1)},\ldots,\rperm_{\echperm(B)}=\perm_B\circ \rperm_{\echperm(B+1)}\middle|\rperm_{\echperm(B+1)}}}.\nonumber 
\end{align*}
This leads to 
\begin{multline} \label{indepetinvariance} 
\proba{\tilde{\ZZ}_n \in A,\ \rperm_{\echperm(1)}\circ \rperm_{\echperm(B+1)}^{-1}=\perm_1,\ \ldots,\ \rperm_{\echperm(B)}\circ \rperm_{\echperm(B+1)}^{-1}=\perm_B} \\
= \mathbb{E} \Bigg[ \proba{\ZZ_n \in A} \times \pa{\prod_{\underset{b\neq \echperm^{-1}(B+1)}{b=1}}^{B}\proba{\rperm_{\echperm(b)}=\perm_b\circ \rperm_{\echperm(B+1)}\middle| \rperm_{\echperm(B+1)}}} \times  \\
\proba{\id=\perm_{\echperm^{-1}(B+1)}\circ \rperm_{\echperm(B+1)}\middle| \rperm_{\echperm(B+1)}} \Bigg], 
\end{multline}
where \eqref{indepetinvariance} holds by independence of all permutations $\rperm_b$ and of $\ZZ_n$ and since, if $f = f_1 \otimes f_2$, $\ZZ_n^{\rperm_{\echperm(B+1)}}$ and $\ZZ_n$ have the same distribution. Hence, 
\begin{align*}
\mathbb{P} \big(\tilde{\ZZ}_n \in A,\ &\rperm_{\echperm(1)}\circ \rperm_{\echperm(B+1)}^{-1}=\perm_1,\ \ldots,\ \rperm_{\echperm(B)}\circ \rperm_{\echperm(B+1)}^{-1}=\perm_B\big)  \\
& =\ \esp{\proba{\ZZ_n \in A} \pa{\frac{1}{n!}}^{B-1} \proba{\id=\perm_{\pi^{-1}(B+1)}\circ \rperm_{\echperm(B+1)}\middle|\rperm_{\echperm(B+1)}}},\nonumber\\
& =\ \proba{\ZZ_n \in A} \pa{\frac{1}{n!}}^{B-1} \proba{\rperm_{\echperm(B+1)}=\perm^{-1}_{\pi^{-1}(B+1)}}, \nonumber\\
& =\ \proba{\ZZ_n \in A} \pa{\frac{1}{n!}}^{B},\nonumber
\end{align*}
This ends the proof of the exchangeability of the $(\widehat{H}^{\star b}_{\lambda,\mu})_{1\leq b\leq B+1}.$

\paragraph{}
Then, by applying Lemma \ref{Lemm:RomanoWolf} to the $(\widehat{H}_{\lambda,\mu}^{\star b})_{1\leq b\leq B+1}$, we obtain
\begin{align}
P_{f_1 \otimes f_2} \left(\widehat{\Delta}_{\alpha}^{\lambda,\mu} = 1 \right) &=  P_{f_1 \otimes f_2} \left( \widehat{\HSIC}_{\lambda,\mu} > \widehat{q}_{1-\alpha}^{\lambda,\mu} \right) \nonumber\\
&= P_{f_1 \otimes f_2} \left( \widehat{H}_{\lambda,\mu}^{\star B+1} > \widehat{H}_{\lambda,\mu}^{\star (\lceil (B+1)(1-\alpha)\rceil)} \right) \nonumber\\
&= P_{f_1 \otimes f_2} \left(\sum_{b=1}^{B+1} \mathds{1}_{\widehat{H}_{\lambda,\mu}^{\star b} < \widehat{H}_{\lambda,\mu}^{\star B+1}} \geq \lceil (B+1)(1-\alpha)\rceil\right) \nonumber\\
&= P_{f_1 \otimes f_2} \left( \sum_{b=1}^{B+1} \mathds{1}_{\widehat{H}_{\lambda,\mu}^{\star b} \geq \widehat{H}_{\lambda,\mu}^{\star B+1}} \leq \lfloor \alpha(B+1)\rfloor \right), \label{floorceil}
\end{align}
where \eqref{floorceil} comes from the fact that $B+1 - \lceil (B+1)(1-\alpha)\rceil = \lfloor \alpha(B+1)\rfloor.$ 
Then, 
\begin{align}
P_{f_1 \otimes f_2} \left(\widehat{\Delta}_{\alpha}^{\lambda,\mu} = 1 \right) &= P_{f_1 \otimes f_2} \left( \sum_{b=1}^{B+1} \mathds{1}_{\widehat{H}_{\lambda,\mu}^{\star b} \geq \widehat{H}_{\lambda,\mu}^{\star B+1}} \leq \alpha(B+1) \right) \nonumber\\ 
&= P_{f_1 \otimes f_2} \left( \frac{1}{B+1} \left( 1 + \sum_{b=1}^B \mathds{1}_{\widehat{H}_{\lambda,\mu}^{\star b} \geq \widehat{H}_{\lambda,\mu}^{\star B+1}} \right)  \leq \alpha \right) \nonumber\\
&\leq \alpha, \label{byRomano}
\end{align}
where \eqref{byRomano} is obtained from Lemma \ref{Lemm:RomanoWolf}.

\subsection{Proof of Lemma \ref{powerful-test}}

Let $\alpha$ and $\beta$ be in $(0,1)$. We aim here to give a condition on $\HH_{\lambda,\mu}(f)$ w.r.t. the variance $\Var_f (\widehat{\HH}_{\lambda,\mu})$ and the quantile $q_{1-\alpha}^{\lambda,\mu}$, so that the statistical test $\Delta_{\alpha}^{\lambda,\mu}$ defined in Equation \eqref{Deltalambdamu} has a second kind error controlled by $\beta$. For this, we use Chebyshev's inequality. Since $\widehat{\HH}_{\lambda,\mu}$ is an unbiased estimator of $\HH_{\lambda,\mu}(f)$, 
\begin{equation*}
P_{f} \left( \abs{\widehat{\HH}_{\lambda,\mu} - \HH_{\lambda,\mu}(f)} \geq \sqrt{\frac{\Var_f (\widehat{\HH}_{\lambda,\mu})}{\beta}} \right) \leq \beta.
\end{equation*}
We then have the following inequality:
\begin{equation*}
P_{f} \left( \widehat{\HH}_{\lambda,\mu} \leq \HH_{\lambda,\mu}(f) - \sqrt{\frac{\Var_f(\widehat{\HH}_{\lambda,\mu})}{\beta}} \right) \leq \beta.
\end{equation*}
Consequently, one has $P_{f} \left( \widehat{\HH}_{\lambda,\mu} \leq q_{1-\alpha}^{\lambda,\mu} \right) \leq \beta$, as soon as $$\HH_{\lambda,\mu}(f) \geq \displaystyle \sqrt{\frac{\Var_f(\widehat{\HH}_{\lambda,\mu})}{\beta}} + q_{1-\alpha}^{\lambda,\mu}.$$ 

\subsection{Proof of Proposition \ref{LTV}}
\label{Variance}

In order to control the variance $\Var_f (\widehat{\HH}_{\lambda,\mu})$ w.r.t. the bandwidths $\lambda$, $\mu$ and the sample size $n$, let us first give the following lemma for a general $U$-statistic of any order $r$ in $\{1, \ldots ,n \}$.
\begin{lemm}
Let $h$ be a symmetric function with $r \leq n$ inputs, $V_1, \ldots ,V_n$ be independent and identically distributed random variables and $U_n$ be the $U$-statistic defined by 
\begin{equation*}
\displaystyle U_n = \frac{(n-r)!}{n!} \sum_{(i_1, \ldots ,i_r) \in \mathbf{i}_r^n} h(V_{i_1}, \ldots ,V_{i_r}),
\end{equation*}
where $\mathbf{i}^n_r$ is the set of all r-tuples drawn without replacement from $\{ 1,\ldots, n\}$. 
The following inequality gives an upper bound of the variance of $U_n$, 
\begin{equation}
\Var(U_n) \leq C(r) \left( \displaystyle \frac{\sigma^2}{n} + \displaystyle \frac{s^2}{n^2} \right),
\end{equation}
where $\sigma^2 = \Var \left( \esp{h (V_1, \ldots ,V_r) \mid V_1} \right)$ and $s^2 = \Var \left( h (V_1, \ldots ,V_r) \right)$.  
\label{Var_Ustat}
\end{lemm}

\medskip
\begin{proof}[Proof of Lemma \ref{Var_Ustat}]
First, using Hoeffding's decomposition (see e.g. \cite[Lemma A, p. 183]{serfling2009approximation}), the variance of $U_n$ can be decomposed as
\begin{equation*}
\Var (U_n) = \displaystyle {n\choose r}^{-1} \sum_{c = 1}^r {r\choose c} {n-r\choose r-c} \zeta_c,
\end{equation*} 
where $\zeta_c = \displaystyle \Var (\esp{h (V_1, \ldots ,V_r) \mid V_1, \ldots ,V_c})$.
\newline \newline 
Let us now prove that, for all $n \in \mathbb{N}^*$, $r \in \{1, \ldots ,n\}$ and $c \in \{1, \ldots , r\}$,
\begin{equation}
{n\choose r}^{-1} {r\choose c} {n-r\choose r-c} \leq \frac{C(r,c)}{n^c}.
\label{ineq-combin-cnr}
\end{equation}
We first write
\begin{eqnarray}
{n\choose r}^{-1} {r\choose c} {n-r\choose r-c} 
&=& {r\choose c} \times \frac{r!}{(r-c)!} \times \frac{(n-r)!}{(n + c -2r)!} \times \frac{(n-r)!}{n!}. 
\label{combin-cnr}
\end{eqnarray}
Moreover, 
\begin{eqnarray*}
n! &=& (n - r)! \times (n-r + 1) \times \ldots \times (n-r+r) \\
&\geq& (n - r)! \times (n-r + 1)^{r},
\end{eqnarray*}
and 
\begin{eqnarray*}
(n-r)! &=& (n - 2r + c)! \times (n- 2r + c + 1) \times \ldots \times (n- 2r + c + r-c) \\
&\leq& (n - 2r + c)! \times (n-r+1)^{r-c}.
\end{eqnarray*}
\newline
Then, we have $$\frac{(n-r)!}{(n + c -2r)!} \times \frac{(n-r)!}{n!} \leq \frac{1}{(n-r+1)^c}.$$  
\newline
Furthemore, using that $n \geq r$, one can write
\begin{eqnarray*}
\frac{n-r+1}{n} &=& 1 - \frac{r-1}{n}\\
&\geq& 1 - \frac{r-1}{r} = \frac{1}{r}. 
\end{eqnarray*}
This leads to, $\displaystyle\frac{1}{n-r+1} \leq  \frac{r}{n}$. Finally, Equation \eqref{combin-cnr} leads to Equation \eqref{ineq-combin-cnr}.  
\newline \newline
By upper bounding each term in Hoeffding's decomposition of the variance of $U_n$ according to Equation \eqref{ineq-combin-cnr}, we obtain
\begin{equation}
\Var (U_n) \leq C(r) \sum_{c = 1}^r \frac{\zeta_c}{n^c}. 
\label{ineq:varun}
\end{equation}
On the one hand, $\zeta_1 = \sigma^2$. 
On the other hand, using the law of total variance (see e.g. \cite{weiss2006course}), $\zeta_c \leq s^2$ for all $c$  in $\lbrace 2, .. ,r \rbrace$. By injecting this last inequality in Equation \eqref{ineq:varun}, we obtain for all $n$ in $\mathbb{N}^*$, 
\begin{equation*}
\Var(U_n) \leq C(r) \left( \displaystyle \frac{\sigma^2}{n} + \displaystyle \frac{s^2}{n^2} \right),
\end{equation*} 
which achieves the proof of Lemma \ref{Var_Ustat}. 
\end{proof} 

\bigskip
Let us now apply Lemma \ref{Var_Ustat} in order to control the variance of $\widehat{\HH}_{\lambda,\mu}$ w.r.t $\lambda$, $\mu$ and $n$. For this, we first recall that $\widehat{\HH}_{\lambda,\mu}$ can be written as a single $U$-statistic of order 4 as 
\begin{equation*}
\widehat{\HH}_{\lambda,\mu} = \frac{1}{n(n-1)(n-2)(n-3)} \sum_{(i,j,q,r) \in \mathbf{i}^n_4} h_{i,j,q,r},
\end{equation*}
where the general term $h_{i,j,q,r}$ of $\widehat{\HH}_{\lambda,\mu}$ is defined as in \cite{gretton2008kernel}  by 
\begin{equation*}
h_{i,j,q,r} =  \frac{1}{4!} \sum_{(t,u,v,w)}^{(i,j,q,r)} \Big[k_{\lambda}(X_t, X_u)  l_{\mu}(Y_t, Y_u) + k_{\lambda}(X_t,X_u) l_{\mu}(Y_v,Y_w) - 2 k_\lambda(X_t,X_u) l_\mu(Y_t,Y_v)\Big].
\end{equation*}
where the sum represents all ordered quadruples $(t,u,v,w)$ drawn without replacement from $(i,j,q,r)$. 
\newline \newline 
Thus, using Lemma \ref{Var_Ustat}, the variance of $\widehat{\HH}_{\lambda,\mu}$ can be upper bounded as follows:
\begin{equation}
\Var_f \left( \widehat{\HH}_{\lambda,\mu} \right) \leq C \pa{\frac{\sigma^2 (\lambda,\mu)}{n} + \frac{s^2(\lambda,\mu)}{n^2}},
\label{ineq:var-s2-sigma2}
\end{equation} 
where, recalling that $Z_i = (X_i,Y_i)$ for all $i$ in $\{1, \ldots ,n\}$, $\sigma^2 (\lambda,\mu) = \Var_f \left( \mathbb{E}[h_{1,2,3,4} \mid Z_1] \right)$ and $s^2 (\lambda,\mu) = \Var_f \left( h_{1,2,3,4} \right)$.

\subsubsection{Upper bound of $\sigma^2 (\lambda,\mu)$}

By now, we upper bound $\sigma^2 (\lambda,\mu)=\Var_f \left( \mathbb{E}[h_{1,2,3,4} \mid Z_1]\right)$ w.r.t. $\lambda$ and $\mu$. For this, we first notice that in the cases where $k_{\lambda} (X_a,X_b) l_{\mu} (Y_c,Y_d)$ is independent from $Z_1$, the variance of its expectation conditionally on $Z_1$  equals 0. This corresponds to the cases where $a,b,c$ and $d$ are all different from 1. We then have the following inequality:
\begin{equation*}
\sigma^2 (\lambda,\mu) \leq \displaystyle C \sum_{i = 1}^6  \sigma^2_i (\lambda,\mu),
\end{equation*}
where 
\begin{align*}
   \sigma^2_1 (\lambda,\mu) &= \Var_f \left( \mathbb{E}[ k_{\lambda} (X_1,X_2) l_{\mu} (Y_1,Y_2)\mid Z_1] \right) , &
   \sigma^2_2 (\lambda,\mu) &=  \Var_f \left( \mathbb{E}[ k_{\lambda} (X_1,X_2) l_{\mu} (Y_3,Y_4)\mid X_1] \right) , \\
   \sigma^2_3 (\lambda,\mu) &= \Var_f \left( \mathbb{E}[ k_{\lambda} (X_3,X_4) l_{\mu} (Y_1,Y_2)\mid Y_1] \right), &
   \sigma^2_4 (\lambda,\mu) &= \Var_f \left( \mathbb{E}[ k_{\lambda} (X_1,X_2) l_{\mu} (Y_1,Y_3)\mid Z_1] \right), \\
   \sigma^2_5 (\lambda,\mu) &= \Var_f \left( \mathbb{E}[ k_{\lambda} (X_2,X_1) l_{\mu} (Y_2,Y_3)\mid X_1] \right), & 
   \sigma^2_6 (\lambda,\mu) &= , \Var_f \left( \mathbb{E}[ k_{\lambda} (X_2,X_3) l_{\mu} (Y_2,Y_1)\mid Y_1] \right).
\end{align*}

\medskip
\textbf{Case 1.} Upper bound of $\sigma^2_1 (\lambda,\mu)$: 
\begin{eqnarray*}
\sigma_1^2 (\lambda,\mu) 
&\leq& \E \left[\big( \E \left[ k_{\lambda} (X_1,X_2) l_{\mu} (Y_1,Y_2) \mid Z_1 \right] \big)^2 \right] \\ 
&\leq& \E \left[ k_{\lambda} (X_1,X_2) l_{\mu} (Y_1,Y_2) k_{\lambda} (X_1,X_3) l_{\mu} (Y_1,Y_3) \right].
\end{eqnarray*}
\newline 
Moreover, we have
\begin{multline*}
\E \left[ k_{\lambda} (X_1,X_2) k_{\lambda} (X_1,X_3) l_{\mu} (Y_1,Y_2) l_{\mu} (Y_1,Y_3) \right] \\  
= \displaystyle \int_{(\R^p\times \R^q)^3} k_{\lambda} (x_1,x_2) k_{\lambda} (x_1,x_3) l_{\mu} (y_1,y_2) l_{\mu} (y_1,y_3) \; \prod_{k = 1}^3 f (x_k,y_k) \mathrm{d} x_k \mathrm{d} y_k.
\end{multline*}
\newline 
Since $k_\lambda$ and $l_{\mu}$ are nonnegative, one can upper bound $f (x_2,y_2)$ and $f (x_3,y_3)$ by $\norm{f}_{\infty}$, and obtain
\begin{eqnarray*}
\sigma_1^2 (\lambda,\mu) 
& \leq &\norm{f}_{\infty}^2 \int_{(\R^p\times \R^q)^3} k_{\lambda} (x_1,x_2) k_{\lambda} (x_1,x_3) l_{\mu} (y_1,y_2) l_{\mu} (y_1,y_3) \;f (x_1,y_1) \prod_{k = 1}^3 \mathrm{d} x_k \mathrm{d} y_k \\ 
& = &\norm{f}_{\infty}^2 \int_{\R^p\times \R^q} \left[ \int_{\R^p} k_{\lambda} (x_1,x) \mathrm{d} x \right]^2 \left[\int_{\R^q} l_{\mu} (y_1,y) \mathrm{d} y \right]^2 f (x_1,y_1) \mathrm{d} x_1 \mathrm{d} y_1.  
\end{eqnarray*}

Finally, using that $\displaystyle \int_{\R^p} k_{\lambda} (\cdot,x) \mathrm{d} x = \int_{\R^q} l_{\mu} (\cdot,y) \mathrm{d} y = 1$, we write
\begin{equation}
\sigma_1^2 (\lambda,\mu) \leq \norm{f}_{\infty}^2.
\label{sigma1e}
\end{equation}

\medskip
\textbf{Case 2.} Upper bound of $\sigma^2_2 (\lambda,\mu)$:
\begin{eqnarray*}
\sigma_2^2 (\lambda,\mu) &\leq& \E \left[ \big( \E \left[ k_{\lambda} (X_1,X_2) l_{\mu} (Y_3,Y_4) \mid X_1 \right] \big)^2 \right] \\ 
&\leq&  \E \left[ \big( \E \left[ k_{\lambda} (X_1,X_2) \mid X_1 \right] \big)^2 \right] \big( \E \left[ l_{\mu} (Y_3,Y_4) \right] \big)^2 \\ 
&\leq&  \E \left[ k_{\lambda} (X_1,X_2) k_{\lambda} (X_1,X_3) \right] \big( \E \left[ l_{\mu} (Y_3,Y_4) \right] \big)^2.
\end{eqnarray*}
Moreover, it is easy to see that by upper bounding $f_1(x_2)$ and $f_1(x_3)$ by $\norm{f_1}_{\infty}$, and recalling that $\displaystyle \int_{\R^p} k_{\lambda} (x_1,x) \mathrm{d} x = 1$, we have,
\begin{align*}
\E \left[ k_{\lambda} (X_1,X_2) k_{\lambda} (X_1,X_3) \right] 
&= \int_{\R^p} \left[ \int_{\R^p} k_{\lambda} (x_1,x_2) f_1(x_2) \mathrm{d} x_2 \right] \left[ \int_{\R^p} k_{\lambda} (x_1,x_3) f_1(x_3) \mathrm{d} x_3 \right] f_1(x_1) \mathrm{d} x_1 \\ 
&\leq \norm{f_1}_{\infty}^2.
\end{align*} 
Besides, upper bounding $f_2(y_3)$ by $\norm{f_2}_{\infty}$ in the integral form of $\E \left[ l_{\mu} (Y_3,Y_4) \right]$ gives
\begin{equation*}
\E \left[ l_{\mu} (Y_3,Y_4) \right] \leq \norm{f_2}_{\infty}.
\end{equation*} 
By combining these inequalities, we obtain
\begin{equation}
\sigma_2^2 (\lambda,\mu) \leq \norm{f_1}_{\infty}^2 \norm{f_2}_{\infty}^2.
\label{sigma2e}
\end{equation}

\medskip
\textbf{Case 3.} Upper bound of $\sigma^2_3 (\lambda,\mu)$: this case is similar to case 2 by exchanging $X$ by $Y$ and $k_{\lambda}$ by $l_{\mu}$. Thus, we have the inequality
\begin{equation}
\sigma_3^2 (\lambda,\mu) \leq \norm{f_1}_{\infty}^2 \norm{f_2}_{\infty}^2.
\label{sigma3e}
\end{equation}

\medskip
\textbf{Case 4.} Upper bound of $\sigma^2_4 (\lambda,\mu)$:
\begin{eqnarray*}
\sigma_4^2 (\lambda,\mu) &\leq& \E \left[ \big( \E \left[ k_{\lambda} (X_1,X_2) l_{\mu} (Y_1,Y_3) \mid Z_1 \right] \big)^2 \right]  \\ 
&\leq& \E \left[ k_{\lambda} (X_1,X_2) k_{\lambda} (X_1,X_4) l_{\mu} (Y_1,Y_3) l_{\mu} (Y_1,Y_5) \right].
\end{eqnarray*}
By upper bounding $f_1 (x_2)$, $f_1 (x_4)$ by $\norm{f_1}_{\infty}$ and $f_2(y_3)$, $f_2(y_5)$ by $\norm{f_2}_{\infty}$ in the integral form of $\E \left[ k_{\lambda} (X_1,X_2) k_{\lambda} (X_1,X_4) l_{\mu} (Y_1,Y_3) l_{\mu} (Y_1,Y_5) \right]$, we obtain
\begin{equation}
\sigma_4^2 (\lambda,\mu) \leq \norm{f_1}_{\infty}^2 \norm{f_2}_{\infty}^2.
\label{sigma4e}
\end{equation}

\medskip
\textbf{Case 5.} Upper bound of $\sigma^2_5 (\lambda,\mu)$:
 \begin{eqnarray*}
\sigma_5^2 (\lambda,\mu) &\leq & \E \left[ \big( \E \left[ k_{\lambda} (X_2,X_1) l_{\mu} (Y_2,Y_3) \mid X_1 \right] \big)^2 \right] \\ 
&\leq & \E \left[ k_{\lambda} (X_2,X_1)  k_{\lambda} (X_4,X_1) l_{\mu} (Y_2,Y_3) l_{\mu} (Y_4,Y_5) \right].
\end{eqnarray*}
By upper bounding $f (x_2,y_2)$ and $f (x_4,y_4)$ by $\norm{f}_{\infty}$ in the integral form of the last expectation, we have
\begin{equation}
\sigma_5^2 (\lambda,\mu) \leq \norm{f}_{\infty}^2.
\label{sigma5e}
\end{equation}

\medskip
\textbf{Case 6.} Upper bound of $\sigma^2_6 (\lambda,\mu)$: this case is similar to case 5 by exchanging $X$ by $Y$ and $k_{\lambda}$ by $l_{\mu}$. We have then the inequality 
\begin{equation}
\sigma_6^2 (\lambda,\mu) \leq \norm{f}_{\infty}^2.
\label{sigma6e}
\end{equation}
\newline
Finally, by combining inequalities \eqref{sigma1e}, \eqref{sigma2e}, \eqref{sigma3e}, \eqref{sigma4e}, \eqref{sigma5e} and \eqref{sigma6e}, we have the following inequality
\begin{equation}
\sigma^2 (\lambda,\mu) \leq \displaystyle C(M_{f}).
\label{sigmae}
\end{equation}

\subsubsection{Upper bound of $s^2 (\lambda,\mu)$}
\label{Annexe/Variance/snew}

Let us first recall that the general term of the $U$-statistic $\widehat{\HH}_{\lambda,\mu}$ is written as   
\begin{equation*}
h_{1,2,3,4} =  \frac{1}{4!} \sum_{(u,v,w,t)}^{(1,2,3,4)} k_{\lambda} (X_u,X_v) \left[ l_{\mu} (Y_u,Y_v) + l_{\mu} (Y_w,Y_t) - 2 l_{\mu} (Y_u,Y_w) \right].
\end{equation*}
Moreover, all the terms of the last sum have the same distribution. We then have
\begin{eqnarray*}
s^2 (\lambda,\mu) &=& \Var_f \left( h_{1,2,3,4} \right) \\
& \leq & C \Var_f \left( k_{\lambda} (X_1,X_2) \left[ l_{\mu} (Y_1,Y_2) + l_{\mu} (Y_3,Y_4) - 2 l_{\mu} (Y_1,Y_3) \right] \right).
\end{eqnarray*}
It follows that, 
\begin{align*}
\Var_f &\left( h_{1,2,3,4} \right) \\ 
&\leq C \big[ \Var_f \left( k_{\lambda} (X_1,X_2) l_{\mu} (Y_1,Y_2) \right) + \Var_f \left( k_{\lambda} (X_1,X_2) l_{\mu} (Y_3,Y_4) \right) + \Var_f \left( k_{\lambda} (X_1,X_2) l_{\mu} (Y_1,Y_3) \right) \big] \\ 
&\leq C \Big( \mathbb{E} \left[ k_{\lambda}^2 (X_1,X_2) l_{\mu}^2 (Y_1,Y_2) \right] + \mathbb{E} \left[ k_{\lambda}^2 (X_1,X_2) l_{\mu}^2 (Y_3,Y_4) \right] + \mathbb{E} \left[ k_{\lambda}^2 (X_1,X_2) l_{\mu}^2 (Y_1,Y_3) \right] \Big). 
\end{align*} 
In order to bring back to multivariate normal densities, we express $k_{\lambda}^2$ and $l_{\mu}^2$ as
\begin{equation*}
k_{\lambda}^2 = \displaystyle \frac{\displaystyle k_{\lambda'}}{(4 \pi)^{\frac{p}{2}} \lambda_1 \ldots \lambda_p } \quad \text{and} \quad l_{\mu}^2 = \displaystyle \frac{\displaystyle l_{\mu'}}{(4 \pi)^{\frac{q}{2}} \mu_1 \ldots \mu_q},
\end{equation*}
where $\displaystyle  \lambda' = \frac{\lambda}{\sqrt{2}}$ and $\mu' = \displaystyle \frac{\mu}{\sqrt{2}}$.

Consequently, the expectation $\E \left[ k_{\lambda}^2 (X_1,X_2) l_{\mu}^2 (Y_1,Y_2) \right]$ can be expressed as
\begin{align*}
 \E \Big[ k_{\lambda}^2 &(X_1,X_2) l_{\mu}^2 (Y_1,Y_2) \Big] \\ 
 &= \displaystyle \frac{1}{(4 \pi)^{\frac{p+q}{2}} \lambda_1 \ldots \lambda_p \mu_1 \ldots \mu_q} \E \left[ k_{\lambda'} (X_1,X_2) l_{\mu'} (Y_1,Y_2)  \right] \\ 
 &= \displaystyle \frac{1}{(4 \pi)^{\frac{p+q}{2}} \lambda_1 \ldots \lambda_p \mu_1 \ldots \mu_q} \int_{(\R^p\times \R^q)^2} k_{\lambda'} (x_1,x_2) l_{\mu'} (y_1,y_2) f(x_1,y_1) f(x_2,y_2) \mathrm{d} x_1 \mathrm{d} x_2 \mathrm{d} y_1 \mathrm{d} y_2.  
\end{align*}
By upper bounding $f(x_2,y_2)$ by $\norm{f}_{\infty}$ in the last integral, we have
\begin{align*}
\displaystyle \int_{(\R^p\times \R^q)^2} & k_{\lambda'} (x_1,x_2) l_{\mu'} (y_1,y_2) f(x_1,y_1) f(x_2,y_2) \mathrm{d} x_1 \mathrm{d} x_2 \mathrm{d} y_1 \mathrm{d} y_2 \\ 
& \leq \displaystyle \norm{f}_{\infty} \int_{\R^p\times \R^q} \left[ \int_{\R^p} k_{\lambda'} (x_1,x_2) \mathrm{d} x_2 \right] \left[ \int_{\R^q} l_{\mu'} (y_1,y_2) \mathrm{d} y_2 \right] f(x_1,y_1) \mathrm{d} x_1 \mathrm{d} y_1 \\ 
&= \norm{f}_{\infty}. 
\end{align*}
This leads to, 
\begin{equation}
\E \left[ k_{\lambda}^2 (X_1,X_2) l_{\mu}^2 (Y_1,Y_2) \right] \leq \displaystyle \frac{\norm{f}_{\infty}}{(4 \pi)^{\frac{p+q}{2}} \lambda_1 \ldots \lambda_p \mu_1 \ldots \mu_q}.
\label{sineq1}
\end{equation}
We can easily show by similar argument that
\begin{equation}
\E \left[ k_{\lambda}^2 (X_1,X_2) l_{\mu}^2 (Y_3,Y_4) \right] \leq \displaystyle \frac{\norm{f_1}_{\infty} \norm{f_2}_{\infty}}{(4 \pi)^{\frac{p+q}{2}} \lambda_1 \ldots \lambda_p \mu_1 \ldots \mu_q}.
\label{sineq2}
\end{equation} 
and
\begin{equation}
\E \left[ k_{\lambda}^2 (X_1,X_2) l_{\mu}^2 (Y_1,Y_3) \right] \leq \displaystyle \frac{\norm{f}_{\infty}}{(4 \pi)^{\frac{p+q}{2}} \lambda_1 \ldots \lambda_p \mu_1 \ldots \mu_q}. 
\label{sineq3}
\end{equation}
From Equations \eqref{sineq1}, \eqref{sineq2} and \eqref{sineq3}, we have 
\begin{equation}
s^2 (\lambda,\mu) \leq \displaystyle \frac{C(M_{f})}{(4 \pi)^{\frac{p+q}{2}} \lambda_1 \ldots \lambda_p \mu_1 \ldots \mu_q}.
\label{sineq}
\end{equation}
From Equations \eqref{sigmae} and \eqref{sineq}, we deduce the following inequality for $\Var_f(\widehat{\HH}_{\lambda,\mu})$
\begin{equation*}
\Var_f(\widehat{\HH}_{\lambda,\mu}) \leq C(M_{f}, p, q) \left\{  \frac{1}{n} + \frac{1}{n^2\lambda_1 \ldots \lambda_p \mu_1 \ldots \mu_q}\right\}.
\end{equation*}

\subsection{Proof of Proposition \ref{LTQ}}
\label{Quantile}

To give an upper bound for the quantile $q_{1-\alpha}^{\lambda,\mu}$ w.r.t $\lambda$ and $\mu$, we use concentration inequalities for general $U$-statistics. 
Recall that $\widehat{\HH}_{\lambda,\mu}$ can be written as a $U$-statistic of order 4, 
\begin{equation*}
\widehat{\HH}_{\lambda,\mu} = \frac{1}{n(n-1)(n-2)(n-3)} \sum_{(i,j,q,r) \in \mathbf{i}^n_4} h_{i,j,q,r},
\end{equation*}
with general term $h_{i,j,q,r}$ defined by 
\begin{equation*}
h_{i,j,q,r} =  \frac{1}{4!} \sum_{(t,u,v,w)}^{(i,j,q,r)} \Big[k_{\lambda}(X_t, X_u)  l_{\mu}(Y_t, Y_u) + k_{\lambda}(X_t,X_u) l_{\mu}(Y_v,Y_w) - 2 k_\lambda(X_t,X_u) l_\mu(Y_t,Y_v)\Big].
\end{equation*}
where the sum represents all ordered quadruples $(t,u,v,w)$ drawn without replacement from $(i,j,q,r)$. \\

However, sharp upper bounds are obtained only for degenerate $U$-statistics (see e.g. \cite{houdre2003exponential}). We recall that a $U$-statistic of order $r$, denoted $U_n = U_n(V_1, \ldots ,V_r)$, is degenerate if $\E [ U_n \mid V_1,\ldots, V_{r-1}] = 0$. 
Note that this implies that $\E [ U_n \mid V_1,\ldots, V_{i}] = 0$ for all $i$ in $\lbrace 1, \ldots ,r-1 \rbrace$. 
Hence, the first step to upper bound $q_{1-\alpha}^{\lambda,\mu}$ is to write $\widehat{\HH}_{\lambda,\mu}$ as a sum of degenerate $U$-statistics. For this, we rely on the ANOVA-decomposition (ANOVA for ANalyse Of VAriance, see e.g. \cite{sobol2001global}) of the symmetric function $h_{i,j,q,r}$. We then write
\begin{equation}
h_{i,j,q,r} = \displaystyle \frac{1}{2!} \sum_{(t,u)}^{(i,j,q,r)} h_{t,u} + \frac{1}{3!} \sum_{(t,u,v)}^{(i,j,q,r)} h_{t,u,v} + \widetilde{h}_{i,j,q,r}, 
\label{ANOVA} 
\end{equation}
where the first (resp. the second) sum represents all ordered pairs $(t,u)$ (resp. triplets $(t,u,v)$) drawn without replacement from $(i,j,q,r)$ and the terms $h_{t,u}$, $h_{t,u,v}$ and $\widetilde{h}_{i,j,q,r}$ are defined as 
\begin{eqnarray*}
h_{t,u} & = & \esp{ h_{i,j,q,r} \mid Z_t , Z_u }, \\ 
h_{t,u,v} & = & \esp{ h_{i,j,q,r} \mid Z_t , Z_u , Z_v } - \frac{1}{2!} \sum_{(t',u')}^{(t,u,v)} h_{t',u'} , \\ 
\widetilde{h}_{i,j,q,r} & = & h_{i,j,q,r} - \displaystyle \frac{1}{3!} \sum_{(t,u,v)}^{(i,j,q,r)} h_{t,u,v} - \frac{1}{2!} \sum_{(t,u)}^{(i,j,q,r)} h_{t,u}. 
\end{eqnarray*}
Hence, by summing all terms $h_{i,j,q,r}$ for $(i,j,q,r)$ in $\mathbf{i}^n_4$ and then dividing by $n(n-1)(n-2)(n-3)$, we have
\begin{equation}
\widehat{\HH}_{\lambda,\mu} = 6\ \widehat{\HH}_{\lambda,\mu}^{(2,D)} + 4\ \widehat{\HH}_{\lambda,\mu}^{(3,D)} + \widehat{\HH}_{\lambda,\mu}^{(4,D)},
\end{equation} 
where
\begin{equation*}
\widehat{\HH}_{\lambda,\mu}^{(2,D)} = \frac{1}{n(n-1)} \sum_{(i,j) \in \mathbf{i}^n_2} h_{i,j} , \quad \quad \widehat{\HH}_{\lambda,\mu}^{(3,D)} = \frac{1}{n(n-1)(n-2)} \sum_{(i,j,q) \in \mathbf{i}^n_3} h_{i,j,q}\\
\end{equation*}
\begin{equation*}
\widehat{\HH}_{\lambda,\mu}^{(4,D)} = \displaystyle \frac{1}{n(n-1)(n-2)(n-3)} \sum_{(i,j,q,r) \in \mathbf{i}^n_4} \widetilde{h}_{i,j,q,r}.
\end{equation*}
\begin{lemm} 
Let us assume that $f = f_1 \otimes f_2$. Then, the $U$-statistics $\widehat{\HH}_{\lambda,\mu}^{(2,D)}$, $\widehat{\HH}_{\lambda,\mu}^{(3,D)}$ and $\widehat{\HH}_{\lambda,\mu}^{(4,D)}$ are degenerated.
\label{degHSIC}
\end{lemm}

\medskip
\begin{proof}[Proof of Lemma \ref{degHSIC}]
According to Theorem 2 of \cite{gretton2008kernel}, if $f = f_1 \otimes f_2$, we have
\begin{equation*}
\E [ h_{i,j,q,r} \mid Z_i ] = 0. 
\end{equation*}
We then easily show that $\widehat{\HH}_{\lambda,\mu}^{(2,D)}$ is degenerated by writing
\begin{equation}
\E [h_{i,j} \mid Z_i] = \E [h_{i,j,q,r} \mid Z_i] = 0.
\label{Deg2}
\end{equation}
Moreover, to prove that $\widehat{\HH}_{\lambda,\mu}^{(3,D)}$ is degenerated, we have
\begin{eqnarray}
\E [h_{i,j,q} \mid Z_i , Z_j] &=& \E [h_{i,j,q,r} \mid Z_i , Z_j] - \E [h_{i,j} \mid Z_i , Z_j] - \E [h_{i,q} \mid Z_i] - \E [h_{j,q} \mid Z_j] \nonumber \\ 
&=& h_{i,j} - h_{i,j} \label{Deg3} \\ 
&=&0 \nonumber,
\end{eqnarray}
where \eqref{Deg3} holds by definition of $h_{i,j}$ and Equation \eqref{Deg2}. 
Finally, from previous cases, and by definition of $h_{i,j,q}$, we obtain
\begin{eqnarray*}
\E [\widetilde{h}_{i,j,q,r} \mid Z_i , Z_j, Z_q] &=& \E [h_{i,j,q,r} \mid Z_i , Z_j, Z_q] - h_{i,j,q} - h_{i,j} -  h_{i,q} - h_{j,q} \\ 
&=& 0, 
\end{eqnarray*}
which proves that $\widehat{\HH}_{\lambda,\mu}^{(4,D)}$ is degenerated. 
\end{proof}

Once we have upper bounds of the $(1-\alpha)$-quantiles of $\widehat{\HH}_{\lambda,\mu}^{(r,D)}$ with $r$ in $\{2,3,4\}$ under the assumption $P_{f_1 \otimes f_2}$, an upper bound of the quantile $q_{1-\alpha}^{\lambda,\mu}$ is naturally obtained. In fact, we can easily show that,
\begin{equation}\label{eq:majq234}
q_{1-\alpha}^{\lambda,\mu} \leq 6\ q_{1-\alpha/3,2}^{\lambda,\mu} + 4\ q_{1-\alpha/3,3}^{\lambda,\mu} + q_{1-\alpha/3,4}^{\lambda,\mu} 
\end{equation}
where $q_{1-\alpha,r}^{\lambda,\mu}$ is the $(1-\alpha)$-quantiles of $\widehat{\HH}_{\lambda,\mu}^{(r,D)}$ under $P_{f_1 \otimes f_2}$.

\subsubsection{Upper bound of $q_{1-\alpha,2}^{\lambda,\mu}$}

First, \cite{gretton2008kernel} page 10 prove that,  under the hypothesis $f= f_1 \otimes f_2$, $h_{i,j} = \esp{h_{i,j,q,r}\mid Z_i,Z_j}$ can be written as follows
$$h_{i,j} = h_{(2)}(Z_i,Z_j),$$
where, for all $z_1=(x,y)$ and $z'=(x',y')$ in $\R^p\times \R^q$, 
\begin{align*}
h_{(2)} (z,z')  = \frac{1}{6} &\biggl\{ k_{\lambda}(x,x') - \esp{k_{\lambda}(x,X')} - \esp{k_{\lambda}(X,x')} + \esp{k_{\lambda}(X,X')} \biggr\} \\ 
\times  &\biggl\{ l_{\mu}(y,y') - \esp{l_{\mu}(y,Y')} - \esp{l_{\mu}(Y,y')} + \esp{l_{\mu}(Y,Y')} \biggr\},
\end{align*}
for $(X,Y)$ and $(X',Y')$ independent random variables with common density $f_1\otimes f_2$. \\

To upper bound the quantile $q_{1-\alpha,2}^{\lambda,\mu}$, we use the concentration inequality for degenerated $U$-statistics of order 2 given in \cite[p.15, Equation (3.5)]{gine2000exponential}. We write for all $t > 0$,
\begin{equation}
\mathbb{P} \left(\abs{ \displaystyle \sum_{(i,j) \in i^n_2} h_{(2)}(Z_i,Z_j) } > t \right) \leq A \exp \left( - \frac{1}{A} \min\left\{ \frac{t}{M}, \left( \frac{t}{L} \right)^{2/3}, \left( \frac{t}{K} \right)^{1/2} \right\} \right),
\label{Gine2000}
\end{equation}  
where $A>1$ is an absolute constant, 
\begin{equation*}
K = \norm{ h_{(2)} }_{\infty} , \quad\mbox{and}\quad M^2 = \sum_{(i,j) \in i_2^n} \esp{h_{(2)}^2(Z_i,Z_j)} = n(n-1) \esp{h_{(2)}^2(Z_1,Z_2)},
\end{equation*}
\begin{equation*}
L^2 = \max\left\{ \displaystyle \norm{ \sum_{i = 1}^n \esp{h_{(2)}^2(Z_i,\cdot)}}_{\infty} , \norm{ \sum_{j = 1}^n \esp{ h_{(2)}^2 \left(\cdot,Z_j \right) } }_{\infty}  \right\} = n \norm{\esp{h_{(2)}^2(Z_1,\cdot)}}_{\infty}. 
\end{equation*}
By setting $\varepsilon = \displaystyle \frac{t}{n^2}$, and using Equation \eqref{Gine2000}, we obtain
\begin{equation*}
\mathbb{P} \left( \displaystyle \frac{1}{n^2} \abs{ \displaystyle \sum_{(i,j) \in i_2^n} h_{i,j} } > \varepsilon \right) \leq A \exp \left( - \frac{1}{A} \min\left\{ \frac{n^2 \varepsilon}{M}, \left( \frac{n^2 \varepsilon}{L} \right)^{2/3}, \left( \frac{n^2 \varepsilon}{K} \right)^{1/2} \right\} \right). 
\end{equation*} 
Therefore, we have for all $\varepsilon > 0$,
\begin{align*}
\mathbb{P} \left( \displaystyle \frac{1}{n^2} \abs{ \displaystyle \sum_{(i,j) \in i_2^n} h_{i,j} } > \varepsilon \right)  \leq& \displaystyle A \max\left\{  \exp \left(-\frac{n^2 \varepsilon}{AM}\right), \exp \left(-\frac{n^{4/3} \varepsilon^{2/3}}{AL^{2/3}}\right), \exp \left(-\frac{n \varepsilon^{1/2}}{AK^{1/2}}\right) \right\}. 
\end{align*}
By adjusting the constant $A$, we can replace in the last inequality $\displaystyle \frac{1}{n^2} \displaystyle \sum_{(i,j) \in i_2^n} h_{i,j}$ by $\widehat{\HH}_{\lambda,\mu}^{(2,D)}$,
\begin{equation*}
\mathbb{P} \left( \abs{ \widehat{\HH}_{\lambda,\mu}^{(2,D)} } > \varepsilon \right) \leq \displaystyle A \max\left\{  \exp \left(-\frac{n^2 \varepsilon}{AM}\right), \exp \left(-\frac{n^{4/3} \varepsilon^{2/3}}{AL^{2/3}}\right), \exp \left(-\frac{n \varepsilon^{1/2}}{AK^{1/2}}\right) \right\}.
\end{equation*}

Hence, if $\varepsilon_{\alpha}$ is a positive number verifying 
\begin{equation*}
\alpha = \displaystyle A \max\left\{  \exp \left(-\frac{n^2 \varepsilon_{\alpha}}{AM}\right), \exp \left(-\frac{n^{4/3} \varepsilon_{\alpha}^{2/3}}{AL^{2/3}}\right), \exp \left(-\frac{n \varepsilon_{\alpha}^{1/2}}{AK^{1/2}}\right) \right\},
\end{equation*}
then, by definition of the quantile, 
\begin{equation}\label{eq:majorq2varepsalpha}
q_{1-\alpha,2}^{\lambda,\mu} \leq \varepsilon_{\alpha}.
\end{equation}
By now, we upper bound $\varepsilon_{\alpha}$ (and consequently $q_{1-\alpha,2}^{\lambda,\mu}$), in the 3 cases considered bellow.

\medskip
\textbf{Case 1.} If $\displaystyle \alpha =  A \exp \left(-n^2 \varepsilon_{\alpha} / \cro{AM}\right),$ then $\varepsilon_{\alpha}$ is expressed as
\begin{equation*}
\varepsilon_{\alpha} = \displaystyle \frac{AM}{n^2} \left( \log \left(\frac{1}{\alpha}\right) + \log \left( A \right) \right). 
\end{equation*}
Since in $\boldsymbol{\mathcal{A}_2(\alpha)}$, we assume that $\log(1/\alpha) > 1$, and since $A>1$, we can then bound $\varepsilon_{\alpha}$ as 
\begin{equation}
0<\varepsilon_{\alpha} \leq \displaystyle \frac{CM}{n^2} \log \left(\frac{1}{\alpha}\right). 
\label{2M}
\end{equation}
for some absolute positive constant $C$. \\

Let us upper bound $M$ w.r.t $\lambda$, $\mu$ and $n$. First notice that 
$$M^2 = n(n-1) \esp{h_{(2)}^2(Z_1,Z_2)} \leq n^2 \esp{h_{1,2}^2}.$$ 
Moreover, by the law of total variance, 
\begin{eqnarray*}
\esp{h_{1,2}^2} &=& \Var \pa{ \esp{h_{1,2,3,4} \mid Z_1, Z_2} } \\ 
&\leq& \Var \left( h_{1,2,3,4} \right). 
\end{eqnarray*}
Furthermore, we have shown in Equation \eqref{sineq} (see Section \ref{Annexe/Variance/snew}) that,
\begin{equation*}
\Var \left( h_{1,2,3,4} \right) \leq \frac{C(M_{f_1\otimes f_2}, p, q)}{\lambda_1 \ldots \lambda_p \mu_1 \ldots \mu_q} \leq \frac{C(M_{f}, p, q)}{\lambda_1 \ldots \lambda_p \mu_1 \ldots \mu_q}, 
\end{equation*}
since we work under $P_{f_1\otimes f_2}$. 

Hence, we can upper bound $M$ as follows, 
\begin{equation}
M \leq \frac{C(M_{f}, p, q) n}{\sqrt{\lambda_1 \ldots \lambda_p \mu_1 \ldots \mu_q}}.
\label{MM}
\end{equation}
Consequently, by combining Equations \eqref{2M} and \eqref{MM}, we obtain
\begin{equation}
q_{1-\alpha,2}^{\lambda,\mu} \leq \frac{C(M_{f}, p, q)}{n \sqrt{\lambda_1 \ldots \lambda_p \mu_1 \ldots \mu_q}} \log\left( \frac{1}{\alpha} \right).
\label{q2case1}
\end{equation}

\medskip
\textbf{Case 2.} If $\displaystyle \alpha =  A \exp \left(-n^{4/3} \varepsilon_{\alpha}^{2/3} / \cro{AL^{2/3}}\right)$, then $\varepsilon_{\alpha}$ verifies
\begin{equation*}
\varepsilon_{\alpha}^{2/3} = \displaystyle \frac{AL^{2/3}}{n^{4/3}} \left( \log \left(\frac{1}{\alpha}\right) + \log \left( A \right) \right).
\end{equation*}
Thus, since $\log(1/\alpha)>1$, $\varepsilon_{\alpha}$ can be upper bounded as
\begin{equation}
\varepsilon_{\alpha} \leq \displaystyle \frac{CL}{n^2} \log \left(\frac{1}{\alpha}\right)^{3/2},
\label{2L}
\end{equation}
Let us upper bound $L$ w.r.t $n$, $\lambda$ and $\mu$, where
\begin{equation*}
L^2 = n \sup_{z\in\R^p\times\R^q}\ac{\esp{h_{(2)}^2(Z_1,z)}}.
\end{equation*}

Yet, for all $z = (x, y) \in \mathbb{R}^{p}\times\mathbb{R}^{q}$, 
\begin{align*}
h_{(2)}^2 (Z_1,z) = \frac{1}{36} &\biggl\{ k_{\lambda}(X_1,x) - \esp{k_{\lambda}(X_1,X_2) \mid X_1} - \esp{k_{\lambda}(X_3,x)}  + \esp{k_{\lambda}(X_3,X_2)}\biggr\}^2 \\ 
\times &\biggl\{ l_{\mu}(Y_1,y) - \esp{l_{\mu}(Y_1,Y_2) \mid Y_1} - \esp{l_{\mu}(Y_3,y)}  + \esp{l_{\mu}(Y_3,Y_2)}\biggr\}^2 .   
\end{align*}
Therefore, we have the following inequality for $h_{(2)}^2 (Z_1,z)$,
\begin{align*}
h_{(2)}^2 (Z_1,z) \leq C &\biggl\{ k_{\lambda}^2(X_1,x) + \esp{k_{\lambda}(X_1,X_2) \mid X_1}^2 + \esp{k_{\lambda}(X_3,x)}^2  + \esp{k_{\lambda}(X_3,X_2)}^2\biggr\} \\ 
\times &\biggl\{ l_{\mu}^2(Y_1,y) + \esp{l_{\mu}(Y_1,Y_2) \mid Y_1}^2 + \esp{l_{\mu}(Y_3,y)}^2  + \esp{l_{\mu}(Y_3,Y_2)}^2\biggr\} .   
\end{align*}
Using that $(X_1, \ldots ,X_n)$ and $(Y_1, \ldots ,Y_n)$ are independent, and Jensen's inequality, 
\begin{align*}
\esp{h_{(2)}^2(Z_1,z)} \leq C \biggl\{ \esp{k_{\lambda}^2(X_1,x)} +  \esp{k_{\lambda}^2(X_1,X_2)} \biggr\} \times \biggl\{ \esp{l_{\mu}^2(Y_1,y)} +  \esp{l_{\mu}^2(Y_1,Y_2)} \biggr\}
\end{align*} 
Moreover, by similar arguments as in Section \ref{Annexe/Variance/snew}, one can prove that for all $x$ in $\R^p$, 
$$\esp{k_{\lambda}^2(X_1,x)} \leq \frac{C(\norm{f_1}_\infty,p)}{\lambda_1 \ldots \lambda_p}, \quad \quad \mbox{and}\quad \quad \esp{k_{\lambda}^2(X_1,X_2)} \leq \frac{C(\norm{f_1}_\infty,p)}{\lambda_1 \ldots \lambda_p},$$
and for all $y$ in $\R^q$, 
$$\esp{l_{\mu}^2(Y_1,y)}\leq \frac{C(\norm{f_2}_\infty,q)}{\mu_1\ldots\mu_q}\quad \quad \mbox{and}\quad \quad \esp{l_{\mu}^2(Y_1,Y_2)} \leq \frac{C(\norm{f_2}_\infty,q)}{\mu_1\ldots\mu_q}.$$
Hence, by taking the supremum over $z=(x,y)$ in $\R^p \times \R^q$, we obtain 
\begin{equation}
L^2 \leq  C(M_{f},p,q) \frac{n}{\lambda_1 \ldots \lambda_p \mu_1 \ldots \mu_q}.
\label{LL}
\end{equation}
By combining Equations \eqref{2L} and \eqref{LL}, we have 
\begin{equation*}
\varepsilon_{\alpha} \leq \displaystyle \frac{C(M_{f},p,q)}{n^{3/2}\sqrt{\lambda_1 \ldots \lambda_p \mu_1 \ldots \mu_q} } \left[ \log\left(\frac{1}{\alpha} \right) \right]^{3/2}.
\end{equation*}
Moreover, since from $\boldsymbol{\mathcal{A}_2(\alpha)}$ we have $\lambda_1 \ldots \lambda_p \mu_1 \ldots \mu_q < 1$, we obtain
\begin{equation}
\varepsilon_{\alpha} \leq \displaystyle \frac{C(M_{f})}{(n \sqrt{\lambda_1 \ldots \lambda_p \mu_1 \ldots \mu_q})^{3/2}} \left[ \log\left(\frac{1}{\alpha} \right) \right]^{3/2}.
\label{q2case2}
\end{equation}

\medskip
\textbf{Case 3.} If $\displaystyle \alpha =  A \exp \left(-n \varepsilon_{\alpha}^{1/2} / \cro{AK^{1/2}}\right)$, then $\varepsilon_{\alpha}$ is expressed as
\begin{equation*}
\varepsilon_{\alpha}^{1/2} =  \frac{AK^{1/2}}{n} \left( \log\left(\frac{1}{\alpha}\right) + \log\left( A \right) \right).
\end{equation*}
Using that, from $\boldsymbol{\mathcal{A}_2(\alpha)}$, $\log(1/\alpha) > 1$, we upper bound $\varepsilon_{\alpha}$ as  
\begin{equation}
\varepsilon_{\alpha} \leq \frac{CK}{n^2} \left[ \log\left(\frac{1}{\alpha}\right) \right]^2.
\label{2K}
\end{equation}
Moreover, we can easily show that 
\begin{equation}\label{KK}
K = \norm{ h_{(2)} }_{\infty} \quad \leq \quad C \pa{\sup_{x,x'\in\R^p} k_\lambda(x,x')} \pa{\sup_{y,y'\in\R^q} l_\mu(y,y')} = \frac{C(p,q)}{\lambda_1\ldots\lambda_p\mu_1\ldots\mu_q}.
\end{equation}
By combining Equations \eqref{2K} and \eqref{KK}, we obtain:
\begin{equation}
\varepsilon_{\alpha} \leq  \frac{C}{n^2\lambda_1 \ldots \lambda_p \mu_1 \ldots \mu_q} \cro{ \log\left(\frac{1}{\alpha}\right) }^2.
\label{q2case3}
\end{equation}

Finally, using Equations \eqref{eq:majorq2varepsalpha}, \eqref{q2case1}, \eqref{q2case2} and \eqref{q2case3} and the fact that, from Assumption $\boldsymbol{\mathcal{A}_2(\alpha)}$,
$$ \frac{1}{n\sqrt{\lambda_1 \ldots \lambda_p \mu_1 \ldots \mu_q}} \log\left(\frac{1}{\alpha} \right) < 1,$$ 
we have the following inequality  
\begin{equation}
q_{1-\alpha,2}^{\lambda,\mu} \leq \displaystyle \frac{C(M_f,p,q)}{n \sqrt{\lambda_1 \ldots \lambda_p \mu_1 \ldots \mu_q}} \log\left( \frac{1}{\alpha} \right).
\label{q2}
\end{equation}

\subsubsection{Upper bound of $q_{1-\alpha,3}^{\lambda,\mu}$}

In this part, we give an upper bound for the $(1-\alpha)$-quantile of 
$$\widehat{\HH}_{\lambda,\mu}^{(3,D)} = \frac{1}{n(n-1)(n-2)} \sum_{(i,j,q) \in \mathbf{i}^n_3} h_{(3)}(Z_i,Z_j,Z_q),$$
where $h_{(3)}$ is define by 
$$h_{(3)}(Z_i,Z_j,Z_q) = h_{i,j,q} = \esp{ h_{i,j,q,r} \mid Z_i , Z_j , Z_q} - \frac{1}{2!} \sum_{(t,u)}^{(i,j,q)} h_{t,u}.$$
For this, we use the concentration inequality (c), page 1501 of \cite{arcones1993limit}. We write for all $t > 0$,
\begin{equation}
\mathbb{P} \pa{ n^{-3/2} \abs{ \sum_{(i,j,q) \in i^n_3}  h_{(3)}(Z_i,Z_j,Z_k) } > t } \leq \displaystyle A \exp \pa{ - \frac{B t^{2/3}}{M^{2/3} + K^{1/2} t^{1/6} n^{-1/4}} },
\label{Arco}
\end{equation} 
where $K = \norm{ h_{(3)} }_{\infty}$, $M^2 = \esp{h_{1,2,3}^2}$ and $A>1$, $B>0$ are absolute constant. 
\newline \newline
By setting $\varepsilon = \displaystyle \frac{t}{n^{3/2}}$ and using Equation \eqref{Arco}, we have
\begin{equation*}
\proba{\frac{1}{n^3} \abs{ \displaystyle \sum_{(i,j,q) \in i_3^n} h_{i,j,q} } > \varepsilon } \leq  A \exp \pa{ - \frac{B n \varepsilon^{2/3}}{M^{2/3} + K^{1/2} \varepsilon^{1/6}} }.
\end{equation*} 
Moreover, by adjusting the value of $B$, we can write 
\begin{equation}
\proba{\abs{ \widehat{\HH}_{\lambda,\mu}^{(3,D)} } > \varepsilon} \leq A \exp \pa{ - \frac{B n \varepsilon^{2/3}}{M^{2/3} + K^{1/2} \varepsilon^{1/6}} }.
\label{Gine1993}
\end{equation}
Hence, if $\varepsilon_{\alpha}$ is a positive number verifying
\begin{equation}
A \exp \pa{ - \frac{B n \varepsilon_\alpha^{2/3}}{M^{2/3} + K^{1/2} \varepsilon_\alpha^{1/6}} } = \alpha,
\label{eq-eps-alph}
\end{equation}
then, we have the following inequality
\begin{equation*}
q_{1-\alpha,3}^{\lambda,\mu} \leq \varepsilon_\alpha.
\label{talpha3}
\end{equation*}
In order to upper bound $\varepsilon_\alpha$ in \eqref{eq-eps-alph}, we set $\gamma_\alpha = \varepsilon_\alpha^{1/6}$ and we obtain
\begin{equation}
B n \gamma_\alpha^4 = K^{1/2} \log\left( \frac{A}{\alpha} \right) \gamma_\alpha + M^{2/3}\log\left( \frac{A}{\alpha} \right).
\label{eq-gamma-alph}
\end{equation}    
The polynomial Equation \eqref{eq-gamma-alph} has no explicit solutions. However, it is possible to give an upper bound of its roots. Indeed, 
\begin{equation*}
B n \gamma_\alpha^4 \leq 2 \max\left\{ K^{1/2} \gamma_\alpha , M^{2/3} \right\} \log\left( \frac{A}{\alpha} \right).
\end{equation*}  

\medskip
\textbf{Case 1.} If $K^{1/2} \gamma_\alpha \geq M^{2/3}$, then, $\gamma_\alpha$ verifies the following inequality, 
\begin{equation*}
\gamma_\alpha^3 \quad\leq\quad \displaystyle \frac{2K^{1/2}}{Bn} \log\left( \frac{A}{\alpha} \right) \quad\leq\quad \frac{CK^{1/2}}{n}\log\pa{\frac{1}{\alpha}},
\end{equation*}
since $\log(1/\alpha)>1$ in $\boldsymbol{\mathcal{A}_2(\alpha)}$, one gets $\log(A/\alpha)\leq C\log(1/\alpha)$. 
Hence,
\begin{equation*}
\varepsilon_\alpha \quad\leq\quad \displaystyle \frac{CK}{n^2} \left( \log\left( \frac{1}{\alpha} \right) \right)^2 .
\end{equation*}
Moreover, once again, one can upper bound $K = \norm{h_{(3)}}_{\infty}$ by 
$$K \quad \leq \quad C \cro{\sup_{x,x'\in\R^p} k_\lambda(x,x')} \cro{\sup_{y,y'\in\R^q} l_\mu(y,y')} \quad =\quad \frac{C(p,q)}{\lambda_1\ldots\lambda_p\mu_1\ldots\mu_q}.$$
Hence, 
$$\varepsilon_\alpha \leq \frac{C(p,q)}{n^2 \lambda_1\ldots\lambda_p\mu_1\ldots\mu_q}\pa{\log\pa{\frac{1}{\alpha}}}^2,$$
and, since from Assumption $\boldsymbol{\mathcal{A}_2(\alpha)}$, 
$$ \frac{1}{n\sqrt{\lambda_1 \ldots \lambda_p \mu_1 \ldots \mu_q}} \log\left(\frac{1}{\alpha} \right) < 1,$$ 
we have the following inequality  
\begin{equation}\label{eq:majquant3cas1}
\varepsilon_\alpha \leq \displaystyle \frac{C}{n \sqrt{\lambda_1 \ldots \lambda_p \mu_1 \ldots \mu_q}} \log\left( \frac{1}{\alpha} \right).
\end{equation}

\medskip
\textbf{Case 2.} If $K^{1/2} \gamma_\alpha \leq M^{2/3}$, then, 
\begin{equation*}
\gamma_\alpha^4 \quad \leq\quad \frac{2M^{2/3}}{Bn} \log\pa{\frac{A}{\alpha}} \quad \leq \quad \frac{CM^{2/3}}{n} \log\pa{\frac{1}{\alpha}},
\end{equation*}
since $\log(1/\alpha)>1$ in $\boldsymbol{\mathcal{A}_2(\alpha)}$.
Therefore, $\varepsilon_\alpha$ can be upper bounded as
\begin{equation*}
\varepsilon_\alpha \leq \displaystyle \frac{CM}{n^{3/2}} \left[ \log\left( \frac{1}{\alpha} \right) \right]^{3/2}.
\end{equation*}
Moreover, using the law of total variance, 
one can upper bound $M^2 = \esp{h_{1,2,3}^2}$ by
\begin{equation}
\displaystyle M^2 = \Var \left( h_{1,2,3} \right) \leq C \Var \left( h_{1,2,3,4} \right).
\end{equation}
Then, according Equation \eqref{sineq} (see Section \ref{Annexe/Variance/snew}), under $P_{f_1\otimes f_2}$, $M$ can be upper bounded as
\begin{equation*}
\displaystyle M \leq \frac{C(M_{f_1\otimes f_2},p,q)}{\sqrt{\lambda_1 \ldots \lambda_p \mu_1 \ldots \mu_q}} \leq \frac{C(M_{f},p,q)}{\sqrt{\lambda_1 \ldots \lambda_p \mu_1 \ldots \mu_q}}.
\end{equation*}
Hence, 
$$\varepsilon_\alpha \leq \frac{C(M_{f},p,q)}{n^{3/2} \sqrt{\lambda_1 \ldots \lambda_p \mu_1 \ldots \mu_q}}  \cro{\log\pa{\frac{1}{\alpha}}}^{3/2},$$ 
Moreover, since both assumptions in $\boldsymbol{\mathcal{A}_2(\alpha)}$ imply that $n^{-1}\log(1/\alpha)<1$, we obtain 
\begin{equation}\label{eq:majquant3cas2}
\varepsilon_\alpha \leq \frac{C(M_f,p,q)}{n \sqrt{\lambda_1 \ldots \lambda_p \mu_1 \ldots \mu_q}}  \log\pa{\frac{1}{\alpha}}.
\end{equation}
Finally, both \eqref{eq:majquant3cas1} and \eqref{eq:majquant3cas2} lead to 
\begin{equation}\label{q3}
q_{1-\alpha,3}^{\lambda,\mu} \leq \frac{C(M_f,p,q)}{n \sqrt{\lambda_1 \ldots \lambda_p \mu_1 \ldots \mu_q}} \log\left( \frac{1}{\alpha} \right).
\end{equation}

\subsubsection{Upper bound of $q_{1-\alpha,4}^{\lambda,\mu}$}

In this part, we give an upper bound for the $(1-\alpha)$-quantile of 
$$\widehat{\HH}_{\lambda,\mu}^{(4,D)} = \frac{1}{n(n-1)(n-2)(n-3)} \sum_{(i,j,q,r) \in \mathbf{i}^n_4} h_{(4)}(Z_i,Z_j,Z_q,Z_r),$$
under $P_{f_1\otimes f_2}$ where $h_{(4)}$ is define by 
$$h_{(4)}(Z_i,Z_j,Z_q,Z_r) = \widetilde{h}_{i,j,q,r} = h_{i,j,q,r} - \displaystyle \frac{1}{3!} \sum_{(t,u,v)}^{(i,j,q,r)} h_{t,u,v} - \frac{1}{2!} \sum_{(t,u)}^{(i,j,q,r)} h_{t,u}. $$
For this, we use the concentration inequality (d), page 1501 of \cite{arcones1993limit}. We have for all $t > 0$, 
\begin{equation*}
\proba{\frac{1}{n^2} \abs{ \sum_{(i,j,q,r) \in i^n_4} \widetilde{h}_{i,j,q,r} } > t } \leq  A \exp\pa{ - B \sqrt{\frac{t}{K}}},
\end{equation*} 
where $A>1$ and $B>0$ are absolute constants and $K = \norm{ h_{(4)} }_{\infty}$. 
\newline \newline
By setting $\varepsilon = \displaystyle \frac{t}{n^2}$, we have
\begin{equation*}
\proba{\frac{1}{n^4} \abs{\sum_{(i,j,q,r) \in i^n_4} \widetilde{h}_{i,j,q,r} } > \varepsilon} \leq  A \exp  \left( - B n \sqrt{\frac{\varepsilon}{K}} \right).
\end{equation*} 
Furthermore, by adjusting the constant $B$, we can replace $\displaystyle \frac{1}{n^4} \displaystyle \sum_{(i,j,q,r) \in i^n_4} \widetilde{h}_{i,j,q,r}$ by $\widehat{\HH}_{\lambda,\mu}^{(4,D)}$ and obtain
\begin{equation}
\proba{\abs{ \widehat{\HH}_{\lambda,\mu}^{(4,D)} } > \varepsilon} \leq  A \exp  \left( - B n \sqrt{\frac{\varepsilon}{K}} \right).
\label{Gine1993d}
\end{equation}
Hence, if $\varepsilon_{\alpha}$ is a positive number verifying
\begin{equation}
A \exp\pa{ - B n \sqrt{\frac{\varepsilon_\alpha}{K}}} = \alpha,
\label{reso4}
\end{equation}
then
\begin{equation*}
q_{1-\alpha,4}^{\lambda,\mu} \leq \varepsilon_\alpha.
\label{talpha4}
\end{equation*}
By resolving Equation \eqref{reso4}, we obtain 
\begin{equation*}
\varepsilon_\alpha = \frac{BK}{n^2} \left[ \log\left( \frac{A}{\alpha} \right) \right]^2.
\end{equation*}
Therefore, since $\log(1/\alpha)>1$ in $\boldsymbol{\mathcal{A}_2(\alpha)}$, we can easily show that 
\begin{equation*}
\varepsilon_\alpha \leq \displaystyle \frac{CK}{n^2} \left[ \log\left( \frac{1}{\alpha} \right) \right]^2. 
\end{equation*}
Moreover, as above, one can upper bound $K = \norm{h_{(4)}}_{\infty}$ by 
$$K \quad \leq \quad C \cro{\sup_{x,x'\in\R^p} k_\lambda(x,x')} \cro{\sup_{y,y'\in\R^q} l_\mu(y,y')} \quad =\quad \frac{C(p,q)}{\lambda_1\ldots\lambda_p\mu_1\ldots\mu_q}.$$
Hence, 
\begin{equation*}
q_{1-\alpha,4}^{\lambda,\mu} \leq \displaystyle \frac{C(M_f,p,q)}{\lambda_1 \ldots \lambda_p  \mu_1 \ldots \mu_q n^2} \left( \log\left( \frac{1}{\alpha} \right) \right)^2.
\end{equation*}
Consequently, since from Assumption $\boldsymbol{\mathcal{A}_2(\alpha)}$,
$$ \frac{1}{n\sqrt{\lambda_1 \ldots \lambda_p \mu_1 \ldots \mu_q}} \log\left(\frac{1}{\alpha} \right) < 1,$$ 
one finally obtains \begin{equation}
q_{1-\alpha,4}^{\lambda,\mu} \leq \displaystyle \frac{C(M_f,p,q)}{n \sqrt{\lambda_1 \ldots \lambda_p  \mu_1 \ldots \mu_q} } \log\left( \frac{1}{\alpha} \right).
\label{q4}
\end{equation}

Finally, combining \eqref{eq:majq234}, \eqref{q2}, \eqref{q3} and \eqref{q4} ends the proof of Proposition \ref{LTQ}. 

\subsection{Proof of Corollary \ref{powerful-testparam}}

The proof of this corollary is immediately obtained from Lemma \ref{powerful-test}, Proposition \ref{LTV} and Proposition \ref{LTQ}. 

\subsection{Proof of Lemma \ref{Hpsi}}

Recalling the formulation of $\HH_{\lambda,\mu} (f)$ given in Equation \eqref{formHSIC} with $k = k_\lambda$ and $l = l_\mu$, we obtain    
\begin{align*}
\HH_{\lambda,\mu} (f) =  &\int_{(\R^p\times\R^q)^2}  k_{\lambda}(x , x') l_{\mu} (y , y') f(x,y) f(x',y') \mathrm{d}x \mathrm{d}y \mathrm{d}x' \mathrm{d}y' \\ 
& - 2\int_{(\R^p\times\R^q)^2} k_{\lambda}(x , x') l_{\mu} (y , y')  f(x,y)  f_1(x') f_2(y') \mathrm{d}x \mathrm{d}y \mathrm{d}x' \mathrm{d}y' \\ 
& +  \int_{(\R^p\times\R^q)^2} k_{\lambda}(x , x') l_{\mu} (y , y') f_1(x) f_2(y) f_1(x') f_2(y') \mathrm{d}x \mathrm{d}y \mathrm{d}x' \mathrm{d}y'.
\end{align*}
This expression can be factorized using the symmetry of the kernels $k_\lambda$ and $l_{\mu}$ as
\begin{align*}
\HH_{\lambda,\mu} (f) = \displaystyle &\int_{(\R^p\times\R^q)^2} k_{\lambda}(x , x') l_{\mu} (y , y') \biggl[ f(x,y) - f_1(x) f_2(y) \biggr]  \biggl[ f(x',y') - f_1(x') f_2(y') \biggr] \mathrm{d}x \mathrm{d}y \mathrm{d}x' \mathrm{d}y' \\ 
= &\int_{(\R^p\times\R^q)^2} k_{\lambda}(x , x') l_{\mu} (y , y') \psi(x,y) \psi(x',y') \mathrm{d}x \mathrm{d}y \mathrm{d}x' \mathrm{d}y',  
\end{align*}
where $\psi(x,y) = f(x,y) - f_1(x) f_2(y)$. 
\newline  \newline
Thereafter, by replacing $k_{\lambda}(x , x')$ with $\varphi_\lambda (x-x')$ and replacing $l_{\mu}(y , y')$ with $\phi_\mu (y-y')$, where $\varphi_\lambda$ and $\phi_\mu$ are respectively the functions defined in Equation \eqref{varphilambdaphimu}, one obtains
\begin{eqnarray*}
\HH_{\lambda,\mu} (f) &=& \int_{\R^p\times\R^q} \psi(x,y) \left[ \int_{\R^p\times\R^q} \psi(x',y') \varphi_\lambda(x-x') \phi_\mu (y-y') \mathrm{d}x' \mathrm{d}y' \right] \mathrm{d}x \mathrm{d}y \\ 
&=& \int_{\R^p\times\R^q} \psi(x,y) \left[ \psi \ast (\varphi_\lambda \otimes \phi_\mu) \right] (x,y) \; \mathrm{d}x \mathrm{d}y \\ 
&=& \langle \psi , \psi \ast (\varphi_\lambda \otimes \phi_\mu) \rangle_{2}.
\end{eqnarray*}

\subsection{Proof of Proposition \ref{prop:finemajvar}}

First recall that $\widehat{\HH}_{\lambda,\mu}$ can be written as a $U$-statistic of order 4, that is 
\begin{equation*}
\widehat{\HH}_{\lambda,\mu} = \frac{1}{n(n-1)(n-2)(n-3)} \sum_{(i,j,q,r) \in \mathbf{i}^n_4} h_{i,j,q,r},
\end{equation*}
where the general term $h_{i,j,q,r}$ of $\widehat{\HH}_{\lambda,\mu}$ is defined by 
\begin{equation}
\label{hijqr}
h_{i,j,q,r} =  \frac{1}{4!} \sum_{(t,u,v,w)}^{(i,j,q,r)} \pa{k_{t,u} l_{t,u} + k_{t,u} l_{v,w} - 2 k_{t,u} l_{t,v}}.
\end{equation}
where the sum represents all ordered quadruples $(t,u,v,w)$ drawn without replacement from $(i,j,q,r)$, and for all $t,u$ in $\{1, \ldots, n\}$, 
$$k_{t,u} = k_{\lambda}(X_t, X_u) \quad \quad \mbox{and}\quad \quad l_{t,u} = l_{\mu}(Y_t, Y_u).$$ 

According to Equations \eqref{ineq:var-s2-sigma2} and \eqref{sineq}, we already proved that 
\begin{equation}
\Var_f (\widehat{\HH}_{\lambda,\mu}) \leq  \frac{ C}{n} \Var_f \left( \mathbb{E}[h_{1,2,3,4} \mid Z_1] \right) + \frac{C(M_{f}, p, q)}{\lambda_1 \ldots \lambda_p \mu_1 \ldots \mu_q n^2}. 
\label{ineq:varinterm}
\end{equation}

To prove the intended result, we need a sharper control of $\Var_f \left( \mathbb{E}[h_{1,2,3,4} \mid Z_1] \right)$ in terms of $\norm{\psi \ast (\varphi_\lambda \otimes \phi_\mu)}_2^2$, which is provided in Lemma \ref{ineq:sigma2-normH2}. 

\begin{lemm}
For all $\lambda$ in $(0, +\infty)^p$ and $\mu$ in $(0, +\infty)^q$, we have
\begin{equation*}
\Var_f \left( \mathbb{E}[h_{1,2,3,4} \mid Z_1] \right) \leq C(M_{f}) \norm{\psi \ast (\varphi_\lambda \otimes \phi_\mu)}_2^2.
\end{equation*}
\label{ineq:sigma2-normH2}
\end{lemm}

Finally, both Equation \eqref{ineq:varinterm} and Lemma \ref{ineq:sigma2-normH2} end the proof of Proposition \ref{prop:finemajvar}. 

\begin{proof}[Proof of Lemma \ref{ineq:sigma2-normH2}]
The first step to upper bound $\Var_f \left( \mathbb{E}[h_{1,2,3,4} \mid Z_1] \right)$ is to rewrite $h_{1,2,3,4}$ by isolating all the terms depending on $Z_1$. 
\begin{align*}
h_{1,2,3,4} =& \frac{1}{4!} \sum_{(t,u,v,w)}^{(1,2,3,4)} \left[ k_{t,u} l_{t,u} + k_{t,u} l_{v,w} - 2 k_{t,u} l_{t,v} \right] \\
=& \frac{2}{4!} \sum_{(u,v,w)}^{(2,3,4)} \left[ k_{1,u} l_{1,u} + k_{1,u} l_{v,w} + k_{u,v} l_{1,w} - k_{w,v} l_{w,1} - k_{u,1} l_{u,v} - k_{1,u} l_{1,v} \right] + R (Z_2 , Z_3, Z_4),
\end{align*}
where the last sum represents all triplets $(u,v,w)$ drawn without replacement from $(2,3,4)$ and $R (Z_2 , Z_3, Z_4)$ is a random variable depending only on $Z_2$, $Z_3$ and $Z_4$. 
\newline \newline
Then, 		
\begin{equation*}
h_{1,2,3,4} = \frac{1}{12} \sum_{(u,v,w)}^{(2,3,4)} \cro{ k_{1,u} (l_{1,u} - l_{1,v}) - k_{u,1} (l_{u,v} - l_{v,w}) - (k_{w,v} - k_{u,v}) l_{1,w} } + R\pa{Z_2 , Z_3, Z_4}.
\end{equation*}
The random variable $R (Z_2 , Z_3, Z_4)$ being independent from $Z_1$, the variance of its expectation conditionally to $Z_1$ is equal to 0. It is then easy to see that $\Var_f \left( \mathbb{E}[h_{1,2,3,4} \mid Z_1] \right)$ can be upper bounded as follows:
\begin{align}
\Var_f \left( \esp{h_{1,2,3,4} \mid Z_1} \right) \leq C \big[ &\Var_f \left( \mathbb{E}[ k_{1,2} (l_{1,2} - l_{1,3}) \mid Z_1] \right) + \Var_f \left( \mathbb{E}[ k_{2,1} (l_{2,3} - l_{3,4}) \mid X_1] \right) \nonumber \\ + &\Var_f \left( \mathbb{E}[ (k_{2,3} - k_{4,3}) l_{1,2} \mid Y_1] \right) \big]. 
\label{varhcondZ1}
\end{align}
By now, we reformulate the function $\psi \ast (\varphi_\lambda \otimes \phi_\mu)$ in a simpler form in order to link its $\mathbb{L}_2$-norm with the upper bound given in Equation \eqref{varhcondZ1}. For notational convenience, we denote $G_{\lambda, \mu} = \psi \ast (\varphi_\lambda \otimes \phi_\mu)$. Then
\begin{eqnarray*}
G_{\lambda, \mu} (x,y) &=&  \int_{\R^p\times \R^q} \psi(x',y') k_{\lambda} (x,x') l_{\mu} (y,y') \; \mathrm{d}x' \mathrm{d}y' \\
&=& \int_{\R^p\times\R^q\times\R^q} k_{\lambda} (x,x') \bigg(l_{\mu} (y,y') - l_{\mu} (y,y'')\bigg)\; f(x',y')f_2(y'') \mathrm{d}x' \mathrm{d}y'\mathrm{d}y'' \\
&=& \esp{k_{\lambda} (x,X') \bigg(l_{\mu}(y,Y') - l_{\mu}(y,Y'')\bigg)},
\end{eqnarray*} 
where $(X',Y')$ and $Y''$ are independent random variables with respective densities $f$ and $f_2$. 
\newline \newline
Thereafter, the conditional expectations in Equation \eqref{varhcondZ1} can all be expressed as follows:   
\begin{eqnarray*}
\esp{k_{1,2} (l_{1,2} - l_{1,3}) \mid Z_1} &=& G_{\lambda, \mu} ( X_1,Y_1 ), \\
\esp{k_{2,1} (l_{2,3} - l_{3,4}) \mid X_1} &=& \esp{G_{\lambda, \mu} ( X_1,Y_3) \mid X_1}, \\
\esp{(k_{2,3} - k_{4,3}) l_{1,2} \mid Y_1} &=&  \esp{G_{\lambda, \mu} ( X_3,Y_1) \mid Y_1}.  
\end{eqnarray*}
Thus, using the law of total variance \cite{weiss2006course}, we have the following upper bound:
$$\Var_f \left( \mathbb{E}[h_{1,2,3,4} \mid Z_1] \right) \leq C \biggl[ \Var_f \left( G_{\lambda, \mu} (X_1,Y_1 ) \right) +  \Var_f \left( G_{\lambda, \mu} ( X_1,Y_3 ) \right) + \Var_f \left( G_{\lambda, \mu} ( X_3,Y_1 ) \right) \biggr].$$ 
On the other hand, it is straightforward to upper bound the three variances in the last equation as
\begin{eqnarray*}
\Var_f \left( G_{\lambda, \mu} (X_1,Y_1) \right) &\leq & \norm{f}_{\infty} \norm{ G_{\lambda, \mu}}_2^2, \\
\Var_f \left( G_{\lambda, \mu} ( X_1,Y_3 ) \right) &\leq & \norm{f_1 \otimes f_2}_{\infty}  \norm{G_{\lambda, \mu}}_2^2 , \\
\Var_f \left( G_{\lambda, \mu} ( X_3,Y_1 ) \right) &\leq & \norm{f_1 \otimes f_2}_{\infty} \norm{G_{\lambda, \mu}}_2^2.  
\end{eqnarray*}
Finally, combining these inequalities with Equation \eqref{varhcondZ1} allows to upper bound $\Var_f \left( \mathbb{E}[h_{1,2,3,4} \mid Z_1] \right)$ as
\begin{equation*}
\Var_f \left( \mathbb{E}[h_{1,2,3,4} \mid Z_1] \right) \leq C(M_{f}) \norm{\psi \ast (\varphi_\lambda \otimes \phi_\mu)}_2^2,
\end{equation*}
which ends the proof of Lemma \ref{ineq:sigma2-normH2}. 
\end{proof}

\subsection{Proof of Lemma \ref{TBSB}}
\label{Bias2}

Recall that for any bandwidths $\lambda = \left( \lambda_1,\ldots, \lambda_p \right)$ in $(0,+ \infty)^p$ and $\mu = \left( \mu_1,\ldots, \mu_q \right)$ in $(0,+ \infty)^q$, $\varphi_{\lambda}$ and $\phi_{\mu}$ are defined in Equation \eqref{varphilambdaphimu} for any $x$ in $\R^p$ and $y$ in $\R^q$, 
\begin{equation*}
\varphi_{\lambda}(x) = \frac1{\lambda_1 \ldots \lambda_p } g_p\pa{\frac{x^{(1)}}{\lambda_1},\ldots , \frac{x^{(p)}}{\lambda_p}}, \quad \phi_{\mu}(y) = \frac1{\mu_1 \ldots \mu_q } g_q\pa{\frac{y^{(1)}}{\mu_1},\ldots , \frac{y^{(q)}}{\mu_q}},
\end{equation*}
where $g_p$ and $g_q$ are the standard Gaussian density defined in Equation \eqref{gs}. 
The objective here is the provide an upper bound of the bias term $\norm{\psi - \psi \ast (\varphi_\lambda \otimes \phi_\mu)}_2^2$ w.r.t $\lambda$ and $\mu$. \\

First of all, since $\psi - \psi\ast (\varphi_\lambda \otimes \phi_\mu)$ belongs to $\mathbb{L}_2$, by Plancherel's theorem we obtain that 
\begin{eqnarray*}
(2\pi)^{p+q} \norm{\psi- \psi\ast (\varphi_\lambda \otimes \phi_\mu)}_2^2 
&=& \norm{\reallywidehat{\psi- \psi\ast (\varphi_\lambda \otimes \phi_\mu)}}_2^2 \\
&=& \norm{\widehat{\psi}\pa{1-\widehat{\varphi_\lambda\otimes \phi_\mu}}}_2^2 \\ 
&=& \int_{\R^p \times \R^q} \abs{1-\widehat{\varphi_\lambda\otimes\phi_\mu}(\xi,\zeta)}^2 \abs{\widehat{\psi}(\xi,\zeta)}^2 \mathrm{d}\xi \mathrm{d}\zeta.
\end{eqnarray*}
Moreover, by definition of $\varphi_\lambda$ and $\phi_\mu$ (see Equation \eqref{varphilambdaphimu}), 
$$\widehat{\varphi_\lambda\otimes\phi_\mu}(\xi,\zeta) = \widehat{g_p\otimes g_q}(\lambda\xi,\mu\zeta),$$
where $\lambda \xi = (\lambda_1\xi^{(1)},\ldots,\lambda_p\xi^{(p)})$ and $\mu \zeta = (\mu_1\zeta^{(1)},\ldots,\mu_q\zeta^{(q)})$. 
Besides, the Gaussian density satisfies for all $(u,v)$ in $\R^p\times\R^q$, 
$$\widehat{g_p\otimes g_q}(u,v) = (2\pi)^{(p+q)/2} g_p\otimes g_q(u,v) = \exp\pa{\frac{-\norm{(u,v)}^2}{2}}.$$
Hence, the bias term satisfies
$$(2\pi)^{p+q} \norm{\psi- \psi\ast (\varphi_\lambda \otimes \phi_\mu)}_2^2  = 
\int_{\R^p \times \R^q} \cro{1-\exp\pa{\frac{-\norm{(\lambda\xi,\mu\zeta)}^2}{2}}}^2 \abs{\widehat{\psi}(\xi,\zeta)}^2 \mathrm{d}\xi \mathrm{d}\zeta.$$

In addition, for all $\delta>0$, there exists $T_\delta$ in $[0,1]$ such that 
\begin{equation}\label{eq:Tdelta}
\forall x \geq T_\delta, \quad \quad 1-\exp\pa{-x^2/2} \leq x^\delta.
\end{equation}
Indeed, the function $g_\delta: x\mapsto 1-\exp(-x^2/2) - x^\delta$ is continuous on $\R_+$, satisfies $g_\delta(0) = 0$, and for all $x\geq 1$, $g_\delta(x) < 0$ since 
$$1-\exp\pa{-x^2/2} < 1 \leq x^\delta.$$ 
Note that if $\delta\leq 2$, then $T_\delta=0$ since, in addition, for all $x$ in $[0,1]$, 
$$1-\exp(-x^2/2) \leq \frac{x^2}{2} \leq \frac{x^\delta}{2}\leq x^\delta.$$

Therefore, one can split the integral as 
\begin{equation}\label{eq:biaisI1I2}
(2\pi)^{p+q} \norm{\psi- \psi\ast (\varphi_\lambda \otimes \phi_\mu)}_2^2  = I_1 + I_2,
\end{equation}
where 
\begin{eqnarray}
I_1 & = & \int_{\norm{(\lambda\xi,\mu\zeta)}<T_\delta} \cro{1-\exp\pa{\frac{-\norm{(\lambda\xi,\mu\zeta)}^2}{2}}}^2 \abs{\widehat{\psi}(\xi,\zeta)}^2 \mathrm{d}\xi \mathrm{d}\zeta \nonumber \\
 & \leq & \pa{1-e^{-T_\delta^2/2}}^2 \norm{\widehat{\psi}}_2^2 \nonumber \\
 & \leq & \pa{1-e^{-T_\delta^2/2}} (2\pi)^{p+q} \norm{\psi}_2^2, \quad \quad \text{since }\pa{1-e^{-T_\delta^2/2}}<1, \label{eq:I1} 
\end{eqnarray}
and
\begin{eqnarray*}
I_2 & = & \int_{\norm{(\lambda\xi,\mu\zeta)}\geq T_\delta} \cro{1-\exp\pa{\frac{-\norm{(\lambda\xi,\mu\zeta)}^2}{2}}}^2 \abs{\widehat{\psi}(\xi,\zeta)}^2 \mathrm{d}\xi \mathrm{d}\zeta \\
 & \leq & \int_{\R^p\times \R^q} \norm{(\lambda\xi,\mu\zeta)}^{2\delta} \abs{\widehat{\psi}(\xi,\zeta)}^2 \mathrm{d}\xi \mathrm{d}\zeta,
\end{eqnarray*}
by Equation \eqref{eq:Tdelta}. 
In addition, since for all $1\leq i \leq p$, $\lambda_i^2 \leq \norm{(\lambda,\mu)}^2$ and for all $1\leq j\leq q$, $\mu_j^2\leq \norm{(\lambda,\mu)}^2$, 
$$\norm{(\lambda\xi,\mu\zeta)}^{2\delta} = \pa{\sum_{i=1}^p\lambda_i^2\cro{\xi^{(i)}}^2+\sum_{j=1}^q \mu_j^2 \cro{\zeta^{(j)}}^2}^{\delta}\leq \norm{(\lambda,\mu)}^{2\delta}\norm{(\xi,\zeta)}^{2\delta}.$$
Thus, since $\psi$ belongs to $\mathcal{S}^{\delta}_{p+q} (R)$, 
\begin{equation}\label{eq:I2}
I_2 \leq (2\pi)^{p+q} R^2 \norm{(\lambda,\mu)}^{2\delta}.
\end{equation}
Thereafter, using H\"{o}lder's inequality if $\delta\geq 1$ and fact that $\norm{ \cdot }_{1/\delta} \leq \norm{ \cdot }_1$ if $\delta<1$, it is straightforward to see that 
\begin{equation}\label{eq:HolderConcav}
\norm{ (\lambda, \mu) }^{2 \delta} \leq \displaystyle C(p,q,\delta) \left[ \sum_{i=1}^p \lambda_i^{2 \delta} +  \sum_{j=1}^q \mu_j^{2 \delta} \right], 
\end{equation}

Finally, combining \eqref{eq:I1}, \eqref{eq:I2} and \eqref{eq:HolderConcav} in \eqref{eq:biaisI1I2} leads to 
$$\norm{\psi- \psi\ast (\varphi_\lambda \otimes \phi_\mu)}_2^2 \leq \pa{1-e^{-T_\delta^2/2}} \norm{\psi}_2^2 + C(p,q,\delta,R) \left[ \sum_{i=1}^p \lambda_i^{2 \delta} +  \sum_{j=1}^q \mu_j^{2 \delta} \right].$$ 

Note once again that if $\delta\leq 2,$ then $T_\delta=0$, and one directly obtains that 
$$\norm{\psi- \psi\ast (\varphi_\lambda \otimes \phi_\mu)}_2^2 \leq C(p,q,\delta,R) \left[ \sum_{i=1}^p \lambda_i^{2 \delta} +  \sum_{j=1}^q \mu_j^{2 \delta} \right].$$

\subsection{Proof of Theorem \ref{ThresholdSobolev}}
\label{Th-Sobolev}

Assume that $\psi$ belongs to the Sobolev balls $\mathcal{S}^{\delta}_{p+q} (R,R')$ with $\delta,R,R'>0$. 
One may notice that, by Lemma \ref{TBSB}, since, $T_\delta\leq 1$, then 
$$\norm{\psi- \psi\ast (\varphi_\lambda \otimes \phi_\mu)}_2^2 \leq \pa{1-e^{-1/2}} \norm{\psi}_2^2 + C(p,q,\delta,R) \left[ \sum_{i=1}^p \lambda_i^{2 \delta} +  \sum_{j=1}^q \mu_j^{2 \delta} \right].$$ 
Thus, since $M_f\leq R'$, one may easily deduce from Theorem \ref{th:powerful-testparam} that $P_{f} (\Delta_{\alpha}^{\lambda,\mu}=0) \leq \beta$ as soon as 
$$e^{-1/2}\norm{\psi}_2^2 > C(p,q,\delta,R) \left[ \sum_{i=1}^p \lambda_i^{2 \delta} +  \sum_{j=1}^q \mu_j^{2 \delta} \right] + \frac{C(R', p, q , \beta)}{n \sqrt{\lambda_1 \ldots \lambda_p \mu_1 \ldots \mu_q}} \log \left( \frac{1}{\alpha} \right),$$ 
that is, since constants may vary from line to line, 
$$\norm{\psi}_2^2 > C(p,q,\delta,R) \left[ \sum_{i=1}^p \lambda_i^{2 \delta} +  \sum_{j=1}^q \mu_j^{2 \delta} \right] + \frac{C(R', p, q , \beta)}{n \sqrt{\lambda_1 \ldots \lambda_p \mu_1 \ldots \mu_q}} \log \left( \frac{1}{\alpha} \right).$$ 
It now follows from the definition \eqref{seprate} of the uniform separation rate that 
$$\cro{\rho \left( \Delta_{\alpha}^{\lambda,\mu} , \mathcal{S}^{\delta}_{p+q} (R,R') , \beta \right)}^2 \leq \displaystyle C(p,q, \delta,R) \left[ \sum_{i=1}^p \lambda_i^{2 \delta} +  \sum_{j=1}^q \mu_j^{2 \delta} \right] + \frac{C(R',p, q, \beta)}{n\sqrt{\lambda_1 \ldots \lambda_p \mu_1 \ldots \mu_q}} \log\left( \frac{1}{\alpha} \right).$$

\subsection{Proof of Corollary \ref{cor:ThresholdSobolev}}

The objective here is to give the uniform separation rate having the smallest upper bound w.r.t. the sample size $n$, when $\psi$ belongs to a Sobolev ball $\mathcal{S}^{\delta}_{p+q} (R,R')$. For this, we recall that according to Theorem \ref{ThresholdSobolev}, we have 
$$\cro{\rho \left( \Delta_{\alpha}^{\lambda,\mu} , \mathcal{S}^{\delta}_{p+q} (R,R') , \beta \right)}^2 \leq \displaystyle C(p,q, \delta,R) \left[ \sum_{i=1}^p \lambda_i^{2 \delta} +  \sum_{j=1}^q \mu_j^{2 \delta} \right] + \frac{C(R', p, q, \beta)}{n\sqrt{\lambda_1 \ldots \lambda_p \mu_1 \ldots \mu_q}} \log\left( \frac{1}{\alpha} \right).$$ 
In order to have the smallest behavior of the right side of the last inequality w.r.t. $n$, one has then to choose bandwidths $\lambda^* = (\lambda^*_1, \ldots ,\lambda^*_p)$ and $\mu^* = (\mu^*_1, \ldots ,\mu^*_q)$ w.r.t. $n$ in such a way that 
$$ \cro{\sum_{i=1}^p \lambda_i^{*2 \delta} +  \sum_{j=1}^q \mu_j^{*2 \delta}} \quad \text{and} \quad \frac{1}{n\sqrt{\lambda_1^* \ldots \lambda_p^* \mu_1^* \ldots \mu_q^*}}$$ 
have the same order. Thereafter, it is clear that all $\lambda^*_i$'s and $\mu^*_j$'s have the same behavior w.r.t. $n$. 
It follows that for all $i$ in $\lbrace 1, \ldots , p \rbrace$ and all $j$ in $\lbrace 1, \ldots , q\rbrace$, we have 
$$ \lambda_i^* = \mu_j^* =   n^{- 2/( 4 \delta + p+q)}.$$ 
Consequently, the separation rate over $\mathcal{S}^{\delta}_{p+q} (R,R')$ can be upper bounded as 
$$\rho \left( \Delta_{\alpha}^{\lambda^*,\mu^*} , \mathcal{S}^{\delta}_{p+q} (R,R') , \beta \right) \leq \displaystyle C(p, q, \alpha, \beta, \delta,R,R') n^{- 2 \delta/(4 \delta + p+q )}.$$

\subsection{Proof of Lemma \ref{TBNB}}
\label{Bias}

The objective here is to give an upper bound of the bias term $\norm{\psi - \psi \ast (\varphi_\lambda \otimes \phi_\mu)}_2^2$ w.r.t. $\lambda$ and $\mu$, when $\psi$ belongs to a Nikol'skii-Besov ball $\mathcal{N}^{\delta}_{2,p+q} (R)$, with $\delta = (\nu_1, \ldots , \nu_p, \gamma_1, \ldots ,\gamma_q)$ in $(0,2]^{p+q}$. We first set $b = \psi \ast (\varphi_\lambda \otimes \phi_\mu) - \psi$ and we write   
\begin{eqnarray*}
b(x,y) &=& \psi \ast (\varphi_\lambda \otimes \phi_\mu) (x,y) - \psi(x,y) \\ 
&=& \displaystyle \int \psi(x',y') \varphi_\lambda(x-x') \phi_\mu (y-y') \mathrm{d}x' \mathrm{d}y' - \psi(x,y). 
\end{eqnarray*}
Moreover, using Equation \eqref{varphilambdaphimu}, the fonction $b$ can be written in terms of the functions $g_p$ and $g_q$ defined in Equation \eqref{gs} as
\begin{eqnarray*}
b(x,y) &=& \displaystyle \frac{1}{\lambda_1 \ldots \lambda_p \mu_1 \ldots \mu_q} \int \psi(x',y') g_p \left(\frac{x_1-x'_1}{\lambda_1}, \ldots ,\frac{x_p-x'_p}{\lambda_p} \right) g_q \left(\frac{y_1-y'_1}{\mu_1}, \ldots ,\frac{y_q-y'_q}{\mu_p} \right) \mathrm{d}x' \mathrm{d}y' \\ 
&&- \psi(x,y) \\
&=& \int \psi(x_1 + \lambda_1 u_1, \ldots , x_p + \lambda_p u_p, y_1 + \mu_1 v_1, \ldots ,y_q + \mu_q v_q) g_p (u_1, \ldots , u_p) g_q (v_1, \ldots , v_q) \; \mathrm{d}u \mathrm{d}v \\
&&- \psi(x,y). 
\end{eqnarray*}
Thereafter, using that $\displaystyle \int_{\mathbb{R}^p} g_p = \int_{\mathbb{R}^q} g_q = 1$, the function $b$ can be expressed as
\begin{equation*}
b(x,y) = \int g_p(u_1, \ldots ,u_p) g_q(v_1, \ldots ,v_q) \biggl[ \psi(x_1 + \lambda_1 u_1, \ldots ,y_q + \mu_q v_q) -  \psi(x, y) \biggr] \; \mathrm{d}u \mathrm{d}v. 
\end{equation*}
Let us from now define for all $i$ in $\{1, \ldots ,p\}$ and $j$ in $\{1, \ldots ,q \}$, the functions $b_{1,i}$ and $b_{2,j}$ by 
\begin{eqnarray*}
b_{1,i} (x,y) &=& \int g_p(u_1, \ldots ,u_p) g_q(v_1, \ldots ,v_q) \omega_{1,i} (x, y, u_1, \ldots , u_i) \; \mathrm{d}u \mathrm{d}v,\\
b_{2,j} (x,y) &=& \int g_p(u_1, \ldots ,u_p) g_q(v_1, \ldots ,v_q) \omega_{2,j} (x, y, u_1, \ldots , u_p, v_1, \ldots ,v_j) \; \mathrm{d}u \mathrm{d}v,
\end{eqnarray*}
where the function $\omega_{1,i}$ is defined as 
\begin{multline*}
\omega_{1,i} (x,y, u_1, \ldots , u_i) = \psi(x_1 + \lambda_1 u_1, \ldots , x_i + \lambda_i u_i, x_{i+1}, \ldots ,x_p, y) \\  
- \psi(x_1 + \lambda_1 u_1, \ldots , x_{i-1} + \lambda_{i-1} u_{i-1}, x_i, \ldots ,x_p, y),
\end{multline*}
while the function $\omega_{2,j}$ is defined as 
\begin{multline*}
\omega_{2,j} (x, y, u_1, \ldots , u_p, v_1, \ldots ,v_j) = \psi(x_1 + \lambda_1 u_1, \ldots ,x_p + \lambda_p u_p ,y_1 + \mu_1 v_1, \ldots , y_j + \mu_j v_j, y_{j+1}, \ldots ,y_q ) \\
- \psi(x_1 + \lambda_1 u_1, \ldots ,x_p + \lambda_p u_p ,y_1 + \mu_1 v_1, \ldots , y_{j-1} + \mu_{j-1} v_{j-1} , y_j, \ldots ,y_q). 
\end{multline*} 
It is then easy to see that the function $b$ is the sum of all the functions $b_{1,i}$ and $b_{2,j}$
$$b(x,y) =\sum_{i =1}^p b_{1,i} (x,y) + \sum_{j =1}^q b_{2,j} (x,y).$$ 
One can then deduce that it would be sufficient for the control of the $\mathbb{L}_2$-norm of $b$, to control the $\mathbb{L}_2$-normes of all the functions $b_{1,i}$ and $b_{2,j}$. Using the triangular inequality, we have
\begin{equation}
\norm{b}_2 \leq \sum_{i =1}^p \norm{b_{1,i}}_2 + \sum_{j =1}^q \norm{b_{2,j}}_2.
\label{nrmb1b2}
\end{equation}
By now, let us upper bound $\norm{b_{1,i}}_2^2$ and $\norm{b_{2,j}}_2^2$ for all $i$ in $\{1, \ldots ,p \}$ and $j$ in $\{1, \ldots ,q \}$. We distinguish two cases. 

\bigskip
\textbf{Case 1.} Assume that $0 < \nu_i \leq 1$. We first recall that $\norm{b_{1,i}}_2^2$ can be written as 
\begin{equation*}
\norm{b_{1,i}}_2^2 = \int \biggl[ \int g_p(u_1, \ldots ,u_p) g_q(v_1, \ldots ,v_q) \omega_{1,i} (x, y, u_1, \ldots , u_i) \; \mathrm{d}u \mathrm{d}v \biggr]^2 \; \mathrm{d}x \mathrm{d}y. 
\end{equation*}
We use the following lemma from page 13 of \cite{tsybakov2009introduction}. 
\begin{lemm}
Let $\rho: \mathbb{R}^d \times \mathbb{R}^{d'} \rightarrow \mathbb{R}$ be a Borel function, then we have the following inequality: 
\begin{equation*}
\displaystyle \int \left( \int \rho(\theta,z) d\theta \right)^2 dz \leq \left[ \int \left( \int \rho^2(\theta,z) dz \right)^{1/2} d\theta \right]^2.
\end{equation*}  
\label{Tsybakov}
\end{lemm}
By applying Lemma \ref{Tsybakov} to the function 
$$\left( (u,v),(x,y) \right) \mapsto g_p(u_1, \ldots ,u_p) g_q(v_1, \ldots ,v_q) \omega_{1,i} (x, y, u_1, \ldots , u_i),$$ 
we obtain
\begin{align}
\norm{b_{1,i}}_2^2 &\leq \biggl[ \int  \biggl( \int g_p^2(u_1, \ldots ,u_p) g_q^2(v_1, \ldots ,v_q) \omega_{1,i}^2 (x, y, u_1, \ldots , u_i) \; \mathrm{d}x \mathrm{d}y \biggr)^{1/2} \; \mathrm{d}u \mathrm{d}v \biggr]^2 \nonumber \\ 
&= \displaystyle \biggl[ \int g_p(u_1, \ldots ,u_p) g_q(v_1, \ldots ,v_q) \biggl( \int \omega_{1,i}^2 (x, y, u_1, \ldots , u_i) \; \mathrm{d}x \mathrm{d}y \biggr)^{1/2} \; \mathrm{d}u \mathrm{d}v \biggr]^2. \label{intb1i}
\end{align}
On the other hand, since $\psi$ belongs to the Nikol'skii-Besov ball $\mathcal{N}^{\delta}_{2,p+q} (R)$, we have 
$$\biggl( \int \omega_{1,i}^2 (x, y, u_1, \ldots , u_i) \;  \mathrm{d}x \mathrm{d}y \biggr)^{1/2} \leq R \lambda_i^{\nu_i} \abs{u_i}^{\nu_i}.$$ 
We then have by injecting this last inequation in Equation \eqref{intb1i}, that 
$$\norm{b_{1,i}}_2^2 \leq C(\nu_i,R) \lambda_i^{2 \nu_i}.$$

\bigskip
\textbf{Case 2.} Now assume that $1 < \nu_i \leq 2$. 
In this case the function $\psi$ has continuous first-order partial derivatives. Using Taylor expansion with integral form of the remainder w.r.t. the $i^{\text{th}}$ variable of $\psi$, we have 
$$\omega_{1,i} (x, y, u_1, \ldots , u_i) = \displaystyle \lambda_i u_i \int_0^1 (1-\tau) D^1_i \psi(x_1 + \lambda_1 u_1, \ldots ,x_i + \tau \lambda_i u_i, x_{i+1}, \ldots ,y) d\tau,$$ 
where $D^1_i$ denotes the first-order partial derivative of $\psi$ w.r.t. the $i^{\text{th}}$ variable. 
\newline \newline
Thereafter, by injecting the last equation in the expression of $b_{1,i}$, we obtain
\begin{multline*}
b_{1,i} (x,y) = \int \lambda_i u_i g_p(u_1, \ldots ,u_p) g_q (v_1, \ldots ,v_q) \\ 
\biggl[ \int_0^1 (1-\tau) D^1_i \psi(x_1 + \lambda_1 u_1, \ldots ,x_i + \tau \lambda_i u_i, x_{i+1}, \ldots ,y) d\tau \biggr] \; \mathrm{d}u \mathrm{d}v. 
\end{multline*}  
Furthermore, using the fact that $g_p$ is of order 2, we have that $\displaystyle \int u_i g_p(u_1, \ldots ,u_p) \mathrm{d}u_i = 0$. The function $b_{1,i}$ can then be written as
\begin{equation*}
b_{1,i} (x,y) =  \int \lambda_i u_i g_p(u_1, \ldots ,u_p) g_q (v_1, \ldots ,v_q) \biggl[  \int_0^1  (1-\tau) D^1_i \omega_{1,i} (x, y, u_1, \ldots , \tau u_i) \; d\tau \biggr] \; \mathrm{d}u \mathrm{d}v.   
\end{equation*}
We have then the following equation for the $\mathbb{L}_2$-norm of $b_{1,i}$:
\begin{equation*}
\norm{b_{1,i}}_2^2 = \int \biggl[ \int \lambda_i u_i g_p(u_1, \ldots ,u_p) g_q (v_1, \ldots ,v_q) \biggl( \int_0^1  (1-\tau) D^1_i \omega_{1,i} (x, y, u_1, \ldots , \tau u_i) \; d\tau \biggr) \; \mathrm{d}u \mathrm{d}v \biggr]^2 \; \mathrm{d}x \mathrm{d}y.
\end{equation*}
By now, we use as in Case 1 of Lemma \ref{Tsybakov} in order to upper bound $\norm{b_{1,i}}_2^2$. We then obtain:    
\begin{multline*}
\norm{b_{1,i}}_2^2  \leq \biggl( \int \biggl[ \int \biggl( \lambda_i u_i g_p(u_1, \ldots ,u_p) g_q (v_1, \ldots ,v_q) \\
\int_0^1 (1- \tau) D^1_i \omega_{1,i} (x, y, u_1, \ldots , \tau u_i) \; d\tau \biggr)^2 \mathrm{d}x \mathrm{d}y \biggr]^{1/2} \; \mathrm{d}u \mathrm{d}v \biggr)^2 
\end{multline*}
Then, 
\begin{multline*}
\norm{b_{1,i}}_2^2  \leq \biggl( \int \lambda_i u_i g_p(u_1, \ldots ,u_p) g_q (v_1, \ldots ,v_q) \\
 \biggl[  \int \biggl(  \int_0^1 (1- \tau) D^1_i \omega_{1,i} (x, y, u_1, \ldots , \tau u_i) \; d\tau \biggr)^2 \mathrm{d}x \mathrm{d}y \biggr]^{1/2} \; \mathrm{d}u \mathrm{d}v \biggr)^2
\end{multline*}
We apply a second time Lemma \ref{Tsybakov}. For this, consider the function 
$$\rho \colon \left( (x,y), \tau \right) \quad \mapsto\quad  (1- \tau) D^1_i \omega_{1,i} (x, y, u_1, \ldots , \tau u_i),$$ we then have
\begin{multline}
\norm{b_{1,i}}_2^2  \leq \biggl( \int \lambda_i u_i g_p(u_1, \ldots ,u_p) g_q (v_1, \ldots ,v_q) \\
\biggl[ \int_0^1 (1- \tau)  \biggl( \int \left( D^1_i \omega_{1,i} (x, y, u_1, \ldots , \tau u_i) \right)^2 \mathrm{d}x \mathrm{d}y \biggr)^{1/2} d\tau \biggr]  \mathrm{d}u \mathrm{d}v \biggr)^2.  \label{intbi}
\end{multline}
On the other hand, using that $\psi$ belongs to the Nikol'skii-Besov ball $\mathcal{N}^{\delta}_{2,p+q} (R)$,
$$\biggl( \int \left( D^1_i \omega_{1,i} (x, y, u_1, \ldots , \tau u_i) \right)^2 \mathrm{d}x \mathrm{d}y \biggr)^{1/2} \leq R \lambda_i^{\nu_i - 1} \abs{\tau u_i}^{\nu_i - 1}.$$ 
We then obtain by injecting this last inequation in Equation \eqref{intbi}, that 
$$\norm{b_{1,i}}_2^2 \leq C(\nu_i,R) \lambda_i^{2 \nu_i}.$$

Besides, for all $j$ in $\{1, \ldots ,q\}$, by similar arguments, one can prove that 
$$\norm{b_{2,j}}_2^2 \leq C(\gamma_j,R) \mu_j^{2 \gamma_j}.$$ 
\newline 
Consequently, according to Equation \eqref{nrmb1b2}, we have the following upper bound of $\norm{b}_2^2$ 
$$\norm{b}_2^2 \leq C(\delta,R) \left[ \sum_{i=1}^p \lambda_i^{2 \nu_i} +  \sum_{j=1}^q \mu_j^{2 \gamma_j} \right].$$

\subsection{Proof of Theorem \ref{ThresholdNikol}}

This proof is similar to the one of Theorem \ref{ThresholdSobolev}. 
Assume that $\psi$ belongs to the Sobolev balls $\mathcal{N}^{\delta}_{2,p+q} (R,R')$ with $\delta = (\nu_1, \ldots , \nu_p, \gamma_1, \ldots ,\gamma_q)$ in $(0,2]^{p+q}$ and $R,R'>0$. Then, since $M_f\leq R'$, one may easily deduce from Theorem \ref{th:powerful-testparam} and Lemma \ref{TBNB} that $P_{f} (\Delta_{\alpha}^{\lambda,\mu}=0) \leq \beta$ as soon as 
$$\norm{\psi}_2^2 > C(\delta,R) \left[ \sum_{i=1}^p \lambda_i^{2 \nu_i} +  \sum_{j=1}^q \mu_j^{2 \gamma_j} \right] + \frac{C(R', p, q , \beta)}{n \sqrt{\lambda_1 \ldots \lambda_p \mu_1 \ldots \mu_q}} \log \left( \frac{1}{\alpha} \right).$$ 

One can then conclude from the definition \eqref{seprate} of the uniform separation rate that 
$$\cro{\rho \left( \Delta_{\alpha}^{\lambda,\mu} , \mathcal{N}^{\delta}_{2,p+q} (R,R') , \beta \right)}^2 \leq C(\delta,R) \left[ \sum_{i=1}^p \lambda_i^{2 \nu_i} +  \sum_{j=1}^q \mu_j^{2 \gamma_j} \right] + \frac{C\left(R', p, q, \beta \right)}{n\sqrt{\lambda_1 \ldots \lambda_p \mu_1 \ldots \mu_q}} \log\left( \frac{1}{\alpha} \right).$$

\subsection{Proof of Corollary \ref{cor:ThresholdNikol}}
We aim here to give the uniform separation rate having the smallest upper bound w.r.t. the sample size $n$, when $\psi$ belongs to a Nikol'skii-Besov ball $\mathcal{N}^{\delta}_{2,p+q} (R,R')$, with $\delta = (\nu_1, \ldots , \nu_p, \gamma_1, \ldots ,\gamma_q)$ in $(0,2]^{p+q}$. 
We first recall that Theorem \ref{ThresholdNikol} shows that 
$$\cro{\rho \left( \Delta_{\alpha}^{\lambda,\mu} , \mathcal{N}^{\delta}_{2,p+q} (R,R') , \beta \right)}^2 \leq C(\delta,R) \left[ \sum_{i=1}^p \lambda_i^{2 \nu_i} +  \sum_{j=1}^q \mu_j^{2 \gamma_j} \right] + \frac{C(R', p, q, \beta)}{n\sqrt{\lambda_1 \ldots \lambda_p \mu_1 \ldots \mu_q}} \log\left( \frac{1}{\alpha} \right).$$ 
Hence, in order to minimize the right side of the last inequality w.r.t. $n$, we choose bandwidths $\lambda^* = (\lambda^*_1, \ldots ,\lambda^*_p)$ and $\mu^* = (\mu^*_1, \ldots ,\mu^*_q)$ w.r.t. $n$ such that
$$ \left[ \sum_{i=1}^p \lambda_i^{*2 \nu_i} +  \sum_{j=1}^q \mu_j^{*2 \gamma_j} \right] \quad \text{and} \quad \frac{1}{n\sqrt{\lambda_1^* \ldots \lambda_p^* \mu_1^* \ldots \mu_q^*}}$$ 
have the same order. Let us set for all $i$ in $\lbrace 1, \ldots , p \rbrace$ and all $j$ in $\lbrace 1, \ldots , q\rbrace$, $\lambda_i^* = n^{a_i}$ and $\mu_j^* = n^{b_j}$. Then, it is clear that for all $i$ and all $j$
\begin{equation}
2 a_i \nu_i = 2 b_j \gamma_j = - \frac{1}{2}  \left[ \sum_{r = 1}^p  a_r + \sum_{s = 1}^q b_s \right] - 1. 
\label{interm:expNiko}
\end{equation}
One can first express all $a_i$'s and all $b_j$'s w.r.t $a_1$ as 
$$a_i = a_1 \frac{\nu_1}{\nu_i} \quad \text{and} \quad b_j = a_1 \frac{\nu_1}{\gamma_j}.$$ 
Thereafter, using Equation \eqref{interm:expNiko} we have
$$2a_1 \nu_1 = \frac{-a_1 \nu_1}{2 \eta} - 1.$$
Thus, we first write that $a_1 = \displaystyle \frac{-2 \eta}{\nu_1 (4\eta + 1)}$. We next obtain for all $i$ and for all $j$ that 
$$a_i = \frac{-2 \eta}{\nu_i (4\eta + 1)} \quad \text{and} \quad b_j = \frac{-2 \eta}{\gamma_j (4\eta + 1)}.$$ 
Note that the condition  $n \geq \pa{\log(1/\alpha)}^{1+1/(4 \eta)}$ ensures that $( \lambda^*,\mu^*)$ satisfies Assumption $\boldsymbol{\mathcal{A}_2(\alpha)}$. 
Consequently, the separation rate over $\mathcal{N}^{\delta}_{2,p+q} (R,R')$ can be upper bound as $$\rho \left( \Delta_{\alpha}^{\lambda^*,\mu^*} , \mathcal{N}^{\delta}_{2,p+q} (R,R'), \beta \right) \leq C(p, q,  \alpha, \beta, \delta,R,R') n^{- \frac{2 \eta}{1 + 4 \eta}}.$$

\subsection{Proof of Lemma \ref{prop5}}

Let $\alpha$ be in $(0,1)$, we first prove that $u_\alpha \geq \alpha$. For this, we apply Bonferroni's inequality
\begin{eqnarray*}
P_{f_1 \otimes f_2} \pa{ \sup_{(\lambda,\mu) \in \W} \ac{ \widehat{\HH}_{\lambda,\mu} -  q_{1- \alpha e^{- \omega_{\lambda,\mu}}}^{\lambda,\mu}} > 0 } 
&=& P_{f_1 \otimes f_2} \pa{ \bigcup_{(\lambda,\mu) \in \W} \ac{ \widehat{\HH}_{\lambda,\mu} > q_{1- \alpha e^{- \omega_{\lambda,\mu}}}^{\lambda,\mu} } } \\
&\leq& \sum_{(\lambda,\mu) \in \W} P_{f_1 \otimes f_2} \pa{\widehat{\HH}_{\lambda,\mu} > q_{1- \alpha e^{- \omega_{\lambda,\mu}}}^{\lambda,\mu}} \\
&\leq& \sum_{(\lambda,\mu) \in \W}  \alpha e^{- \omega_{\lambda,\mu}} \\ 
&\leq& \quad \alpha.  
\end{eqnarray*}  
Then, by definition of $u_\alpha$ we have $u_\alpha \geq \alpha$. Thereafter, we obtain
\begin{eqnarray*}
P_{f} \pa{\Delta_\alpha = 0} 
&=& P_{f} \pa{ \bigcap_{(\lambda,\mu) \in \W} \ac{\widehat{\HH}_{\lambda,\mu} \leq q_{1- u_\alpha e^{- \omega_{\lambda,\mu}}}^{\lambda,\mu}} } \\
&\leq& \inf_{(\lambda,\mu) \in \W} P_{f} \pa{ \widehat{\HH}_{\lambda,\mu} \leq q_{1- u_\alpha e^{- \omega_{\lambda,\mu}}}^{\lambda,\mu} } \\
&\leq& \inf_{(\lambda,\mu) \in \W} P_{f} \pa{ \widehat{\HH}_{\lambda,\mu} \leq q_{1- \alpha e^{- \omega_{\lambda,\mu}}}^{\lambda,\mu} } \\
&=& \inf_{(\lambda,\mu) \in \W} \ac{P_{f} \pa{\Delta_{\alpha e^{- \omega_{\lambda,\mu}}}^{\lambda,\mu}=0}},
\end{eqnarray*}
which concludes the proof. 

\subsection{Proof of Theorems \ref{theorem3Sobol} and \ref{theorem3Nikol}}

Let $\alpha$ and $\beta$ be in $(0,1)$. According to Lemma \ref{prop5}, $P_{f} \pa{\Delta_\alpha = 0} \leq \beta$ as soon as there exists $(\lambda,\mu)$ in $\W$ such that 
$$P_{f} \pa{\Delta_{\alpha e^{- \omega_{\lambda,\mu}}}^{\lambda,\mu}=0} \leq \beta.$$ 

Then, according to Theorem \ref{ThresholdSobolev} if $\psi$ belongs to $\mathcal{S}^{\delta}_{p+q} (R,R')$, or Theorem \ref{ThresholdNikol} if $\psi$ belongs to $\mathcal{N}^{\delta}_{2,p+q} (R,R')$, one can take the infimum of the upper bounds for the uniform separation rates over $\mathcal{S}^{\delta}_{p+q} (R)$ (resp. over $\mathcal{N}^{\delta}_{2,p+q} (R,R')$) of the single tests over $\W$ while replacing $\log(1/\alpha)$ by $\log(1/\alpha) + \omega_{\lambda,\mu}$.

\subsection{Proof of Corollary \ref{corr1Sobol}}

Assume that $\psi$ belongs to $\mathcal{S}^{\delta}_{p+q} (R,R')$ with regularity parameter $\delta>0$ and positive radiuses $R,R'$. Let us first verify that  $\boldsymbol{\mathcal{A}_2(\alpha e^{-\omega_{\lambda,\mu}})}$ holds for all $(\lambda,\mu)$ in $\W$. 
Let $(\lambda,\mu)$ in $\W$. Then, by definition of $M_n^{p,q}$, 
$$n \sqrt{ \lambda_1 \ldots \lambda_p \mu_1 \ldots \mu_q}\ \geq\  n 2^{-M_n^{p,q} \pa{\frac{p+q}2}}\ \geq\ \log(n).$$
Moreover, 
 \begin{eqnarray*}
 \log\pa{\frac1\alpha} + \omega_{\lambda,\mu} &\leq& \log\pa{\frac1\alpha} + 2 \log \pa{M_n^{p,q}  \frac{ \pi}{\sqrt{6}}}\\
  &\leq&  \log\pa{\frac1\alpha}  + 2  \log \pa{ \frac{ \pi}{\sqrt{6}}} + \log \pa{\frac{ 2}{p+q}} + \log(\log_2(n)).
\end{eqnarray*} 
This implies that there exists $C(p,q,\alpha)$ such that for $n \geq C(p,q, \alpha)$, $\boldsymbol{\mathcal{A}_2(\alpha e^{-\omega_{\lambda,\mu}})}$ holds, and this for all $(\lambda,\mu)$ in $\W$.
Hence, using Theorem \ref{theorem3Sobol}, we have the following inequality
\begin{multline*}
\cro{\rho \left( \Delta_{\alpha} , \mathcal{S}^{\delta}_{p+q} (R,R') , \beta \right) }^2\leq \displaystyle C\pa{p, q, \beta , \delta,R,R'} \inf_{(\lambda,\mu) \in \W} \left\{ 
\left[ \sum_{i=1}^p \lambda_i^{2 \delta} + \sum_{j=1}^q \mu_j^{2 \delta} \right] + \right.\\ 
\left. + \frac{1}{n\sqrt{\lambda_1 \ldots \lambda_p \mu_1 \ldots \mu_q}} \left( \log\left( \frac{1}{\alpha} \right) + \omega_{\lambda,\mu} \right) \right\}.
\end{multline*}
Let us take $(\lambda^*, \mu^*) = 2^{-m^*} \boldsymbol{1}_{p+q}$ with $m^*$ satisfying the condition
\begin{equation*}
\pa{\frac n{\log\log (n)}}^{\frac2{p+q+4\delta}}  <  2^{m^* } \leq 2  \pa{\frac n{\log\log (n)}}^{\frac2{p+q+4\delta}}.
\end{equation*}
Note that there exists a positive constant $C(p,q, \delta)$ such that for $ n \geq C(p,q, \delta)$, $m^* \in \ac{1, \ldots , M_n^{p,q}}$, which implies that $ (\lambda^*, \mu^*) \in \W$. Noticing that
\begin{eqnarray*}
\left[ \sum_{i=1}^p (\lambda^*_i)^{2 \delta} + \sum_{j=1}^q (\mu^*_j)^{2 \delta} \right] &+ &
 \frac{1}{n\sqrt{\lambda^*_1 \ldots \lambda^*_p \mu^*_1 \ldots \mu^*_q}} \left( \log\left( \frac{1}{\alpha} \right) + \omega_{\lambda^*,\mu^*} \right) \\
 &&\leq   C\pa{p, q, \alpha,  \delta}\left( \frac{\log \log (n)}{n} \right)^{4\delta/( 4 \delta + p+q)},
 \end{eqnarray*}
and applying Theorem \ref{theorem3Sobol}, we obtain the desired result.

\subsection{Proof of Corollary \ref{corr1Nikol}}

Assume that $\psi$ belongs to $\mathcal{N}^{\delta}_{2,p+q} (R,R')$ with regularity parameter $\delta = (\nu_1, \ldots , \nu_p, \gamma_1, \ldots ,\gamma_q)$ in $(0,2]^{p+q}$ and positive radiuses $R,R'$. 
Let us first verify that 
$\boldsymbol{\mathcal{A}_2(\alpha e^{-\omega_{\lambda,\mu}})}$ holds for all $(\lambda,\mu)$ in $\W$.
The condition $\sum_{i=1}^p m_{1,i} + \sum_{j=1}^q m_{2,j} \leq  2 \log_2\cro{\pa{\frac n {\log(n)}}}$ implies that for all $(\lambda, \mu )$ in $\W$, 
$$ n \sqrt{ \lambda_1 \ldots \lambda_p \mu_1 \ldots \mu_q} \geq \log(n) .$$
Moreover, by definition of the weights $\omega_{\lambda,\mu}$, we have that for all $(\lambda,\mu)$ in $\W$, 
\begin{eqnarray*}
 \log\pa{\frac1\alpha}+ \omega_{\lambda,\mu} &\leq& \log\pa{\frac1\alpha} + 2(p+q) \log \pa{ \frac{ \pi}{\sqrt{6}}} + 2\sum_{i=1}^p \log(m_{1,i}) +  2\sum_{j=1}^q \log(m_{2,i})  \\
&\leq &  \log\pa{\frac1\alpha} + 2(p+q) \log \pa{ \frac{ 2\pi}{\sqrt{6}}} + 2(p+q) \log(\log_2(n)) . 
\end{eqnarray*}
This implies that there exists some constant $C(p,q,\alpha)$ such that for $n \geq C(p,q, \alpha)$, $\boldsymbol{\mathcal{A}_2(\alpha e^{-\omega_{\lambda,\mu}})}$ holds for all $(\lambda,\mu)$ in $\W$.
Hence, using Theorem \ref{theorem3Nikol}, we have the following inequality
\begin{multline*}
\cro{  \rho \left( \Delta_{\alpha}, \mathcal{N}^{\delta}_{2,p+q} (R,R') , \beta \right) }^2 \leq C(p, q, \beta, \delta,R,R') 
\inf_{(\lambda,\mu) \in \W} \Bigg\lbrace \cro{ \sum_{i=1}^p \lambda_i^{2 \nu_i} +  \sum_{j=1}^q \mu_j^{2 \gamma_j} } \\
+ \frac{1}{\sqrt{\lambda_1 \ldots \lambda_p \mu_1 \ldots \mu_q}n} \cro{ \log\pa{\frac{1}{\alpha}} + \omega_{\lambda,\mu}} \Bigg\rbrace.
\end{multline*}
Let us take $\lambda^* = (2^{-m^*_{1,1}}, \ldots ,2^{-m^*_{1,p}})$ and $\mu^* = (2^{-m^*_{2,1}}, \ldots ,2^{-m^*_{2,q}})$, where the integers $m^*_{1,1}, \ldots , m^*_{1,p}$, $m^*_{2,1}, \ldots , m^*_{2,q}$ are defined by the inequalities
$$\pa{ \frac{n}{\log \log (n)} }^{\frac{2 \eta}{\nu_i(1 + 4 \eta)}} < 2^{m^*_{1,i}} \leq   2 \pa{ \frac{n}{\log \log (n)} }^{\frac{2 \eta}{\nu_i(1 + 4 \eta)}}$$
and
$$\pa{ \frac{n}{\log \log (n)} }^{\frac{2 \eta}{\gamma_i(1 + 4 \eta)}} < 2^{m^*_{2,j}} \leq   2 \pa{ \frac{n}{\log \log (n)} }^{\frac{2 \eta}{\gamma_j (1 + 4 \eta)}},$$ 
where $\eta^{-1} = \sum_{i = 1}^p (\nu_i)^{-1} + \sum_{j = 1}^q (\gamma_j)^{-1}$.  Note that there exists a positive constant $C(\delta)$ such that for $ n \geq C(\delta)$, $(\lambda^*, \mu^*)$ belongs to $\W$. 
Then, we obviously have
\begin{multline*}
\cro{ \rho \left( \Delta_{\alpha}, \mathcal{N}^{\delta}_{2,p+q} (R,R') , \beta \right) }^2 \leq C(p, q, \beta, \delta,R,R') \Bigg\{ \cro{\sum_{i=1}^p (\lambda^*_i)^{2 \nu_i} +  \sum_{j=1}^q (\mu^*_j)^{2 \gamma_j}} \\
 + \frac{1}{\sqrt{\lambda^*_1 \ldots \lambda^*_p \mu^*_1 \ldots \mu^*_q}n} \left( \log\left( \frac{1}{\alpha} \right)  + \omega_{\lambda^*,\mu^*} \right) \Bigg\}.
\end{multline*}
By definition of  the integers $m^*_{1,1}, \ldots , m^*_{1,p}$, $m^*_{2,1}, \ldots , m^*_{2,q}$, we have 
\begin{equation*}
(\lambda^*_i)^{-1/2} = 2^{m_{1,i}^*/2} \leq \sqrt{2}  \left( \frac{n}{\log \log (n)} \right)^{\frac{\eta}{\nu_i(1 + 4 \eta)}} 
\quad \text{and} \quad 
(\mu^*_j)^{-1/2} = 2^{m_{2,j}^*/2} \leq \sqrt{2} \left( \frac{n}{\log \log (n)} \right)^{\frac{\eta}{\gamma_j (1 + 4 \eta)}}.
\end{equation*}
Therefore, we obtain 
\begin{equation}
(\lambda^*_1 \ldots \lambda^*_p \mu^*_1 \ldots \mu^*_q)^{-1/2} \leq  2^{(p+q)/2 } \left( \frac{n}{\log \log (n)} \right)^{\frac{1}{(1 + 4 \eta)}}. 
\label{corr1eq1}
\end{equation}
Let us now upper bound $\omega_{\lambda^*,\mu^*}$. We first write
\begin{eqnarray*}
\omega_{\lambda^*,\mu^*} &=& 2 \sum_{i=1}^p \log \pa{m^*_{1,i} \times \frac{\pi}{\sqrt{6}}} + 2 \sum_{j=1}^q \log \pa{m^*_{2,j} \times \frac{\pi}{\sqrt{6}}} \\ 
&=& 2 \log \pa{ m^*_{1,1} \ldots m^*_{1,p} m^*_{2,1} \ldots m^*_{2,q} } + 2(p+q) \log \pa{\frac{\pi}{\sqrt{6}}}.  
\end{eqnarray*}
Moreover, it is easy to see that for $n \geq C(\delta)$,
\begin{equation*}
m^*_{1,i} \leq \frac{2 \eta}{\nu_i(1 + 4 \eta)} \log_2(n) 
\quad \text{and} \quad 
m^*_{2,j} \leq \frac{2 \eta}{\gamma_j (1 + 4 \eta)} \log_2(n).
\end{equation*}
Then, for $n \geq C(\delta)$,
\begin{equation*}
\log( m^*_{1,1} \ldots m^*_{1,p} m^*_{2,1} \ldots m^*_{2,q}) \leq C(\delta) \log \log (n).
\end{equation*}
Thereafter, $\omega_{\lambda^*,\mu^*}$ can be upper bound as
\begin{equation}
\omega_{\lambda^*,\mu^*} \leq C(\delta,p,q) \log \log (n).
\label{corr1eq2}
\end{equation}
From Equations \eqref{corr1eq1} and \eqref{corr1eq2}, we have 
\begin{equation}
\displaystyle \frac{1}{n \sqrt{\lambda^*_1 \ldots \lambda^*_p \mu^*_1 \ldots \mu^*_q}} \left( \log\left( \frac{1}{\alpha} \right) + \omega_{\lambda^*,\mu^*} \right) \leq C(\alpha, \delta ,p,q) \left( \frac{\log \log (n)}{n} \right)^{\frac{4 \eta}{(1 + 4 \eta)}}.
\label{corr1eq3}
\end{equation}
We aim now to upper bound $\sum_{i=1}^p (\lambda^*_i)^{2 \nu_i} +  \sum_{j=1}^q (\mu^*_j)^{2 \gamma_j}$. 
By definition of  the integers $m^*_{1,1}, \ldots , m^*_{1,p}$, $m^*_{2,1}, \ldots , m^*_{2,q}$, 
\begin{equation*}
(\lambda^*_i)^{2 \nu_i} \leq \left( \frac{\log \log (n)}n \right)^{\frac{4 \eta}{1 + 4 \eta}} \quad \text{and} \quad (\mu^*_j)^{2 \gamma_j} \leq \left( \frac{\log \log (n)}n \right)^{\frac{4 \eta}{1 + 4 \eta}}.
\end{equation*}
Therefore, we obtain
\begin{equation}
\sum_{i=1}^p (\lambda^*_i)^{2 \nu_i} +  \sum_{j=1}^q (\mu^*_j)^{2 \gamma_j} \leq (p+q) \left( \frac{\log \log (n)}n \right)^{\frac{4 \eta}{1 + 4 \eta}}. 
\label{corr1eq4}
\end{equation} 
Consequently, from Equations \eqref{corr1eq3} and \eqref{corr1eq4}, 
\begin{equation*}
\rho \left( \Delta_{\alpha} , \mathcal{N}^{\delta}_{2,p+q} (R,R'), \beta \right) \leq \displaystyle C(p, q, \alpha, \beta, \delta,R,R') \left( \frac{\log \log (n)}{n} \right)^{\frac{2 \eta}{1 + 4 \eta}},
\end{equation*} 
which ends the proof of Corollary \ref{corr1Nikol}

\subsection{Proof of Lemma \ref{Lem:MethodLowerBound}} 

Assume there exists a distribution $f_0$ that satisfies $(\mathcal{H}_0)$ such that the probability measure $P_{\nu_{\rho_*}}$ is absolutely continuous w.r.t. $P_{f_0}$ and verifies Equation \eqref{Ineq:Lnu}.

Let us first lower bound $\beta \big[ \mathcal{F}_{\rho_*} (\regSp_\delta) \big]$ w.r.t. the distributions $P_{\nu_{\rho_*}}$ and $P_{f_0}$. 
We recall that
$$ \beta \big[ \mathcal{F}_{\rho_*} (\regSp_\delta) \big] = \underset{\Delta_\alpha}{\inf}
\underset{f \in  \mathcal{F}_{\rho_*} (\regSp_\delta) }{\sup} P_{f}  \left( \Delta_\alpha = 0 \right) .$$
Using the assumption $\nu_{\rho_*}(\mathcal{F}_{\rho_*} (\regSp_\delta)) \geq 1-\eta$, we obtain the following inequalities 
\begin{eqnarray*}
P_{\nu_{\rho_*}} \left( \Delta_\alpha = 0 \right)  &=& \int_{\mathbb{L}_2(\R^p\times \R^q)} P_{f}  \left( \Delta_\alpha = 0 \right) d\nu_{\rho_*}(f)  \\
 &\leq& \int_{\mathcal{F}_{\rho_*} (\regSp_\delta)} P_{f}  \left( \Delta_\alpha = 0 \right) d\nu_{\rho_*}(f)  + \eta\\
 &\leq& \sup_{f \in \mathcal{F}_{\rho_*} (\regSp_\delta)} P_{f}  \left( \Delta_\alpha = 0 \right)   + \eta. 
 \end{eqnarray*}
This leads to 
$$ \sup_{f \in \mathcal{F}_{\rho_*} (\regSp_\delta)} P_{f}  \left( \Delta_\alpha = 0 \right)  \geq  P_{\nu_{\rho_*}} \left( \Delta_\alpha = 0 \right) -\eta.$$
Hence, we have 
\begin{eqnarray*}
\beta \big[ \mathcal{F}_{\rho_*} (\regSp_\delta) \big]  &\geq&  \underset{\Delta_\alpha}{\inf}  P_{\nu_{\rho_*}} \left( \Delta_\alpha = 0 \right) -\eta \\
&\geq & 1 - \underset{\Delta_\alpha}{\sup}  P_{\nu_{\rho_*}} \left( \Delta_\alpha = 1 \right) -\eta  \\
&\geq& 1-\alpha - \underset{\Delta_\alpha}{\sup} \abs{ P_{\nu_{\rho_*}} \left( \Delta_\alpha = 1 \right) -P_{f_0} \left( \Delta_\alpha = 1 \right) } -\eta. 
\end{eqnarray*}
We denote by $\norm{ P_{\nu_{\rho_*}} - P_{f_0} }_{TV}$ the total variation distance between the distributions $P_{\nu_{\rho_*}}$ and $P_{f_0}$. We recall that, $$\norm{ P_{\nu_{\rho_*}} - P_{f_0} }_{TV} = \underset{E \in \mathcal{E}}{\sup} \abs{ P_{\nu_{\rho_*}} (E) - P_{f_0}(E) },$$ where $\mathcal{E}$ is the space of measurable sets in $\R^{n(p+q)}$. We then obtain $$\beta \big[ \mathcal{F}_{\rho_*} (\regSp_\delta) \big] \geq 1 - \alpha -\eta- \norm{ P_{\nu_{\rho_*}} - P_{f_0} }_{TV}.$$ Notice that, $$\norm{ P_{\nu_{\rho_*}} - P_{f_0} }_{TV} = \underset{E \in \mathcal{E}}{\sup}  \left[ P_{\nu_{\rho_*}} (E) - P_{f_0}(E) \right] = \underset{E \in \mathcal{E}}{\sup} \left[ P_{f_0}(E) - P_{\nu_{\rho_*}} (E) \right].$$ It is then straightforward to show that 
\begin{eqnarray*}
\norm{ P_{\nu_{\rho_*}} - P_{f_0} }_{TV} &=& \frac{1}{2} \int_{\R^{n(p+q)}} \abs{ L_{\nu_{\rho_*}} - 1 } \, \mathrm{d} P_{f_0} \\
&=& \frac{1}{2} \mathbb{E}_{P_{f_0}} \left[ \abs{ L_{\nu_{\rho_*}} (\mathbb{Z}_n) - 1 } \right] \\
&\leq& \frac{1}{2} \left( \mathbb{E}_{P_{f_0}} \left[ L^2_{\nu_{\rho_*}} (\mathbb{Z}_n) \right] - 1 \right)^{1/2},
\end{eqnarray*}
where the last inequality holds by applying Cauchy-Schwarz and the fact that $\mathbb{E}_{P_{f_0}} \left[ L_{\nu_{\rho_*}} (\mathbb{Z}_n) \right] = 1$. Thus, 
$$\beta \big[ \mathcal{F}_{\rho_*} (\regSp_\delta) \big] \geq  1 - \alpha -\eta- \frac{1}{2} \left( \mathbb{E}_{P_{f_0}} \left[ L^2_{\nu_{\rho_*}} (\mathbb{Z}_n) \right] - 1 \right)^{1/2}.$$ 
If the condition \eqref{Ineq:Lnu} holds, we then obtain 
$$\beta \big[ \mathcal{F}_{\rho_*} (\regSp_\delta) \big] > \beta.$$ 
Furthermore, using that $\mathcal{F}_{\rho_*} (\regSp_\delta) \subset \mathcal{F}_{\rho} (\regSp_\delta)$ for all $\rho \leq \rho_*$, we have 
$$\beta \big[ \mathcal{F}_{\rho} (\regSp_\delta) \big] > \beta.$$

\paragraph{}
Let us now prove that this implies the lower bound 
\begin{equation}
\rho \left(\regSp_\delta, \alpha, \beta \right) = \inf_{\Delta_\alpha}\rho \left( \Delta_\alpha, \regSp_\delta, \beta \right) \geq \rho_*.
\label{eq:borneinfgene}
\end{equation}
Assume 
$\beta \big[ \mathcal{F}_{\rho_*} (\regSp_\delta) \big] > \beta$, then
$$\forall \Delta_\alpha,\quad \underset{f \in \mathcal{F}_{\rho_*} \pa{\regSp_\delta}}{\sup} P_f \left( \Delta_\alpha = 0 \right)>\beta.$$
In particular, since the family $\ac{\mathcal{F}_{\rho}\pa{\regSp_\delta}}_{\rho>0}$ is non increasing for the inclusion,
$$\forall \Delta_\alpha,\quad \rho \left( \Delta_\alpha, \regSp_\delta, \beta \right) = \inf\ac{\rho>0\ ;\ \underset{f \in \mathcal{F}_{\rho} \pa{\regSp_\delta}}{\sup} P_f \left( \Delta_\alpha = 0 \right)\leq\beta} > \rho_*,$$
which directly implies \eqref{eq:borneinfgene}.

\subsection{Proof of Lemma \ref{prop:alternftheta}} 



\paragraph{Proof of \ref{borninfproba}} Assume that 
\begin{equation}\label{eq:cond1C0}
C_0 \leq \min\{1,R'-1\}e^{p+q}.
\end{equation}
Let us first prove that $f_\theta$ is a density function. First, it is obvious from Equation \eqref{ftheta} that $$\displaystyle \int_{\mathbb{R}^{p+q}} f_\theta (x,y) \, \mathrm{d}x \, \mathrm{d}y = 1,$$ since $\mathds{1}_{[0,1]^{p+q}}$ is a probability density function and that $\displaystyle \int_{\mathbb{R}} G(x) \, \mathrm{d}x = 0$. 
It remains to check that $f_\theta$ is a non-negative function under Assumption \eqref{eq:cond1C0}. 

Let $j=(j_1,\ldots,j_p)$ in $\{1,\ldots,M_n\}^p$ and $l=(l_1,\ldots,l_q)$ in $\{1,\ldots,M_n\}^q$. Knowing that for all $1\leq r\leq p$ and all $1\leq s\leq q$, the supports of the functions $G_{h_n} (\cdot - j_r h_n)$  and $G_{h_n} (\cdot - l_s h_n)$ are respectively the intervals $\Big((j_r-1)h_n, j_r h_n \Big]$ and $\Big((l_s-1)h_n, l_s h_n \Big]$, we deduce that the support of the function 
\begin{equation}
\label{def:gnpq}
g_{n,j,l}:(x,y) \mapsto \prod_{r = 1}^p G_{h_n} (x_r - j_r h_n) \prod_{s = 1}^q G_{h_n}(y_s - l_s h_n)
\end{equation}
is the set 
\begin{equation}
D_{(j,l)}=\prod_{r = 1}^p \Big( (j_r -1)h_n, j_r h_n \Big] \times \prod_{s = 1}^q \Big( (l_s -1)h_n, l_s h_n \Big]. 
\label{eq:support}
\end{equation}
These supports are then disjoint for different 
multi-indexes $(j,l)$ in $\I_{n,p,q}$ and have as union set $(0,1]^{p+q}$ (since $M_n h_n = 1$). In particular, for all $(x,y)$ in $(0,1]^{p+q}$, 
\begin{eqnarray}
\abs{ \sum_{(j,l)\in \I_{n,p,q}} \theta_{(j,l)} \prod_{r = 1}^p G_{h_n} (x_r - j_r h_n) \prod_{s = 1}^q G_{h_n}(y_s - l_s h_n) } &\leq& \frac{1}{h_n^{p+q}}  \left(\sup_{t \in [-1,0]} |G(t)| \right)^{p+q} \nonumber\\
&=& \frac{1}{(e h_n)^{p+q}}.\label{eq:majresteftheta}
\end{eqnarray}
Hence, if $(x,y)$ belongs to $[0,1]^{p+q}$, then since $h_n\leq 1$,
$$f_\theta(x,y) \geq 1 - C_0\frac{h_n^\delta}{e^{p+q}} \geq 1 - \frac{C_0}{e^{p+q}} \geq 0,$$
by equation \eqref{eq:cond1C0}. 
Otherwise, if $(x,y)\notin[0,1]^{p+q}$, then $f_\theta(x,y)=0$. 
In particular, for all $(x,y)$ in $\R^{p+q}$, $f_\theta(x,y)\geq 0$. \\

Remains to prove that $\max\{\norm{f_\theta}_\infty,\norm{f_{\theta,1}}_\infty,\norm{f_{\theta,2}}_\infty\}\leq R'$. 
On the one hand, since $f_{\theta,1}=\mathds{1}_{[0,1]^p}$ and $f_{\theta,2}=\mathds{1}_{[0,1]^q}$, we directly obtain 
$$\norm{f_{\theta,1}}_\infty = \norm{f_{\theta,2}}_\infty = 1 \leq R'.$$
On the other hand, by \eqref{eq:majresteftheta}, for all $(x,y)$ in $\R^{p+q}$, 
$$\abs{f_\theta(x,y)} \leq 1 + C_0\frac{h_n^\delta}{e^{p+q}} \leq 1+\frac{C_0}{e^{p+q}}.$$
Hence, assuming \eqref{eq:cond1C0} directly leads to $\norm{f_{\theta}}_\infty\leq R'$, which ends the proof of this point. 

\paragraph{Proof of \ref{borninfdistH0}} 

Let us prove that, for all $\theta$ in $\{-1,1\}^{M_n^{p+q}}$, $f_\theta$ satisfies 
$$\norm{f_{\theta} - f_{\theta,1} \otimes f_{\theta,2}}_2 = C(p,q,\delta,R,R',\eta)h_n^\delta.$$
Since, $\int_\R G(t) \mathrm{d}t =0$, we know that $f_{\theta,1}=\mathds{1}_{[0,1]^p}$ and $f_{\theta,2}=\mathds{1}_{[0,1]^q}$, thus $f_{\theta,1} \otimes f_{\theta,2} = \mathds{1}_{[0,1]^{p+q}}$ and
$$f_{\theta} - f_{\theta,1} \otimes f_{\theta,2} = C_0 h_n^{\delta + (p+q)}  \sum_{(j,l)\in \I_{n,p,q}} \theta_{(j,l)} g_{n,j,l}(x,y),$$
where the functions $g_{n,j,l}$ are defined in \eqref{def:gnpq}, with disjoint supports. \\
In particular, 
$$\norm{f_{\theta} - f_{\theta,1} \otimes f_{\theta,2}}_2^2 = C_0^2 h_n^{2\delta + 2(p+q)}  \sum_{(j,l)\in \I_{n,p,q}} \norm{g_{n,j,l}}_2^2.$$
Moreover, for all $(j,l)\in \I_{n,p,q}$, 
\begin{eqnarray*}
\norm{g_{n,j,l}}_2^2 &=& \int_{\R^{p+q}} \cro{\prod_{r = 1}^p G_{h_n}^2 (x_r - j_r h_n) \prod_{s = 1}^q G_{h_n}^2(y_s - l_s h_n)} \mathrm{d}x_1 \ldots \mathrm{d}x_p \mathrm{d}y_1\ldots \mathrm{d}y_q \\
&=& \cro{\prod_{r = 1}^p \pa{\int_\R G_{h_n}^2(x_r - j_r h_n) \mathrm{d}x_r}}\times \cro{ \prod_{s = 1}^q \pa{\int_\R G_{h_n}^2(y_s - l_s h_n) \mathrm{d}y_s}},
\end{eqnarray*}
and for all $k$ in $\{1,\ldots,M_n\}$, a simple change of variables implies that 
$$\int_\R G_{h_n}^2(t-k h_n) \mathrm{d}t \quad =\quad  \frac{1}{h_n^2}\int_\R G^2\pa{\frac{t-k h_n}{h_n}} \mathrm{d}t 
\quad =\quad \frac{1}{h_n}\int_\R G^2(t) \mathrm{d}t\quad =\quad \frac{\norm{G}_2^2}{h_n},$$
since $G$ belongs to $\mathbb{L}_2(\R)$. 
We thus deduce that 
\begin{equation}
\norm{g_{n,j,l}}_2^2= \frac{\norm{G}_2^{2(p+q)}}{h_n^{p+q}},
\label{eq:norm2g}
\end{equation} and that, since the cardinality of $I_{n,p,q}$ equals $M_n^{p+q}$, recalling that $M_nh_n=1$,
$$\norm{f_{\theta} - f_{\theta,1} \otimes f_{\theta,2}}_2^2 =  C_0^2 \norm{G}_2^{2(p+q)} h_n^{2\delta + 2(p+q)}  \times\frac{M_n^{p+q}}{h_n^{p+q}} = C_0^2 \norm{G}_2^{2(p+q)} h_n^{2\delta}.$$

\subsection{Proof of Lemma \ref{lemma:alternfthetaSobolev}}

Let us prove that there exists a positive constant $C(p,q,\delta,\eta)$ such that, if $C_0^2 \leq (2\pi)^{p+q} R^2 / [2C(p,q,\delta,\eta)]$, then the random function $f_\Theta - f_{\Theta,1} \otimes f_{\Theta,2}$ belongs to the Sobolev ball $\mathcal{S}^\delta_{p+q} (R)$ with probability greater that $1-\eta$. 
This point relies on Lemma \cite[Lemma 2]{butucea2007goodness} recalled below. 
\begin{lemm}[\cite{butucea2007goodness}]
Let $G$ be the function defined in Equation \eqref{Eq:functionG}. Then $G$ is an infinitely differentiable function such that $\displaystyle \int_{\mathbb{R}} G(x) \, \mathrm{d}x =0$. Its Fourier transform verifies 
\begin{equation*}
\abs{\widehat{G} (u)} \leq C \exp \pa{- a \sqrt{\abs{u}}} \quad \text{as} \quad \abs{u} \rightarrow \infty,
\end{equation*}
for some positive constants $C$ and $a$. Moreover, $\widehat{G}$ is an infinitely differentiable and bounded function.
\label{lemm:Butucea}
\end{lemm}

According to the Fourier transform properties, we write, for all $\theta$ in $\{-1,1\}^{M_n^{p+q}}$,
 $$\widehat{f}_\theta (u,v) = \widehat{f}_{\theta,1} \otimes \widehat{f}_{\theta,2} (u,v) + C_0\displaystyle h_n^{\delta + (p+q)} \sum_{(j,l)\in \I_{n,p,q}} \theta_{j,l} \prod_{r = 1}^p \exp(i u_r j_r h_n) \widehat{G} (h_n u_r) \prod_{s = 1}^q \exp(i v_s l_s h_n) \widehat{G}(h_n v_s).$$
Then, 
\begin{equation} \label{Eq:Fourierftheta}
\abs{ \widehat{f}_\theta (u,v) - \widehat{f}_{\theta,1} \otimes \widehat{f}_{\theta,2} (u,v) }^2 = H_{1,n} (u,v) + H_{2,n} (u,v,\theta),
\end{equation}
where the functions $H_{1,n}$ and $H_{2,n}$ are respectively defined by
\begin{align}
H_{1,n} (u,v) &= C_0^2 M_n^{p+q} h_n^{2\delta + 2(p+q)}\pa{\prod_{r = 1}^p \abs{ \widehat{G} (h_n u_r) }^2} \pa{\prod_{s = 1}^q \abs{ \widehat{G} (h_n v_s) }^2} \label{H1n}, \\
H_{2,n} (u,v,\theta) &= C_0^2 h_n^{2\delta + 2(p+q)} \sum_{\substack{(j_1, l_1)\in \I_{n,p,q} \\ (j_2, l_2) \in \I_{n,p,q} \\ (j_1, l_1) \neq (j_2, l_2) }} \theta_{j_1,l_1} \theta_{j_2,l_2} \mathcal{G}_{j_1,l_1,j_2,l_2} (h_n u, h_n v), \label{H2n}.
\end{align}
and, the function $\mathcal{G}_{j_1,l_1,j_2,l_2}$ is defined for all indexes $j_k = (j_{k,1}, \ldots, j_{k,p})$ and $l_k = (l_{k,1}, \ldots, l_{k,q})$, $k=1,2$, as 
\begin{equation} \label{Gj1k1j2k2}
\mathcal{G}_{j_1,l_1,j_2,l_2}: (u,v) \mapsto \pa{\prod_{r = 1}^p \exp\left(i u_r (j_{1,r}-j_{2,r}) \right) \abs{ \widehat{G} (u_r) }^2} \pa{\prod_{s = 1}^q \exp\left( i v_s (l_{1,s}-l_{2,s}) \right) \abs{ \widehat{G} (v_s) }^2}.
\end{equation} 
By now, our aim is to prove that
 $$ \mathbb{P} \pa{ \int_{\R^{p+q}} \norm{(u,v)}^{2 \delta} \abs{ \widehat{f}_\Theta (u,v) - \widehat{f}_{\Theta,1} \otimes \widehat{f}_{\Theta,2} (u,v) }^2 \mathrm{d}u \mathrm{d}v \leq (2\pi)^{p+q}R^2} \geq 1-\eta.$$ First, by Equation \eqref{eq:HolderConcav}, we have 
\begin{equation} \label{Ineq:norm_uv}
\norm{(u,v)}^{2 \delta} \leq C(p,q,\delta)\left[ \displaystyle \sum_{i = 1}^p \abs{u_i}^{2 \delta} + \sum_{j=1}^q \abs{v_j}^{2 \delta}\right].
\end{equation}
We then obtain from Equations \eqref{H1n} and \eqref{Ineq:norm_uv} the following result,  
\begin{multline*}
\int_{\R^{p+q}} \norm{(u,v)}^{2 \delta} H_{1,n} (u,v) \, \mathrm{d}u \, \mathrm{d}v \\ 
\leq  C(p,q,\delta) C_0^2 M_n^{p+q} h_n^{2\delta + 2(p+q)} \left( \int_{\mathbb{R}} |t|^{2 \delta} \abs{ \widehat{G} (h_n t) }^2 \, \mathrm{d}t  \right) \left( \int_{\mathbb{R}}  \abs{ \widehat{G} (h_n z) }^2 \, \mathrm{d}z  \right)^{p+q-1} \\
=  C(p,q,\delta) C_0^2(M_n h_n)^{p+q} \left( \int_{\mathbb{R}} |t|^{2 \delta} \abs{ \widehat{G} (t) }^2 \, \mathrm{d}t  \right) \left( \int_{\mathbb{R}} \abs{ \widehat{G} (z) }^2 \, \mathrm{d}z  \right)^{p+q-1}.
\end{multline*}
The functions $t \mapsto |t|^{2 \delta} \abs{ \widehat{G} (t) }^2$ and $z \mapsto \abs{ \widehat{G} (z) }^2$ being integrable according to Lemma \ref{lemm:Butucea}, we have
\begin{eqnarray} \label{Ineq:IntH1n}
\int_{\R^{p+q}} \norm{(u,v)}^{2 \delta} H_{1,n} (u,v) \, \mathrm{d}u \, \mathrm{d}v &\leq&   C(p,q,\delta) C_0^2 (M_n h_n)^{p+q} \nonumber \\ 
&\leq&    (2\pi)^{p+q}R^2 /2
\end{eqnarray}
provided that $C(p,q,\delta) C_0^2 \leq (2\pi)^{p+q}R^2 /2$. \\

To complete  the proof, let us now consider the random part. Starting from the expression of $H_{2,n}$ in \eqref{H2n}, we write $$\int_{\R^{p+q}} \norm{(u,v)}^{2 \delta} H_{2,n} (u,v,\Theta) \, \mathrm{d}u \, \mathrm{d}v = C_0^2 h_n^{p+q} \sum_{\substack{(j_1, l_1)\in \I_{n,p,q} \\ (j_2, l_2) \in \I_{n,p,q} \\ (j_1, l_1) \neq (j_2, l_2) }} \Theta_{j_1,l_1} \Theta_{j_2,l_2} \int_{\mathbb{R}^{p + q}} \norm{(u,v)}^{2 \delta} \mathcal{G}_{j_1,l_1,j_2,l_2} (u,v) \, \mathrm{d}u \, \mathrm{d}v.$$
Noting that $  \mathcal{G}_{j_2,l_2,j_1,l_1} = \overline{ \mathcal{G}_{j_1,l_1,j_2,l_2}}$, we get 
$$
  \sum_{\substack{(j_1, l_1)\in \I_{n,p,q} \\ (j_2, l_2) \in \I_{n,p,q} \\ (j_1, l_1) \neq (j_2, l_2)  }} \Theta_{j_1,l_1} \Theta_{j_2,l_2}  \mathcal{G}_{j_1,l_1,j_2,l_2} (u,v) =\sum_{(j_1, l_1)\prec (j_2, l_2) \in \I_{n,p,q}} \Theta_{j_1,l_1} \Theta_{j_2,l_2} \pa{ \mathcal{G}_{j_1,l_1,j_2,l_2} (u,v) + \overline{ \mathcal{G}_{j_1,l_1,j_2,l_2}}} 
$$
where $ (j_1,l_1) \prec (j_2,l_2)$ means that $ (j_1,l_1)$ is strictly smaller that $ (j_2,l_2)$ in the lexicographic order. 
In the following, we prove that there exists some positive constant $C(p,q,\delta, \eta)$ such that, with probability greater that $1-\eta$, we have 
\begin{equation}\label{eq:controlH2}
 \sum_{(j_1, l_1)\prec (j_2, l_2) \in \I_{n,p,q}}  \Theta_{j_1,l_1} \Theta_{j_2,l_2} \int_{\mathbb{R}^{p + q}} \norm{(u,v)}^{2 \delta} \pa{\mathcal{G}_{j_1,l_1,j_2,l_2} (u,v)  + \overline{ \mathcal{G}_{j_1,l_1,j_2,l_2} }}\, \mathrm{d}u \, \mathrm{d}v \leq C(p,q,\delta,\eta)  M_n^{p+q} .
\end{equation}
Since $M_n h_n =1$, this implies that 
$$ \proba{ \int_{\R^{p+q}} \norm{(u,v)}^{2 \delta} H_{2,n} (u,v,\Theta) \, \mathrm{d}u \, \mathrm{d}v  \leq (2\pi)^{p+q}R^2 /2} \geq 1-\eta, $$
choosing $C_0$ such that $C(p,q,\delta,\eta) C_0^2 \leq (2\pi)^{p+q}R^2 /2$, and concludes the proof. \\

In order to show that \eqref{eq:controlH2} holds with probability greater that $1-\eta$, we use Hoeffding's inequality. For the sake of completeness, let us first recall this inequality. 
\begin{lemm}\label{lem:Hoeffding} \cite{Hoeffding1963}
Let $Z_1, \ldots, Z_n$ be independent real random variables such that for all $i$, $a\leq Z_i \leq b$. Then we have, for all $x>0$, 
$$ \proba{ \abs{Z_1+ \ldots+Z_n} \geq x} \leq 2 \exp \pa{ - \frac{2x^2}{n(b-a)^2}}.$$
\end{lemm}
We apply Hoeffding's inequality to the variables $ (Z_{j_1,l_1,j_2,l_2})_{ (j_1, l_1) \in \I_{n,p,q}, (j_2, l_2) \in \I_{n,p,q},  (j_1, l_1)\prec (j_2, l_2)}$, where
$$ Z_{j_1,l_1,j_2,l_2} = \Theta_{j_1,l_1} \Theta_{j_2,l_2} \int_{\mathbb{R}^{p + q}} \norm{(u,v)}^{2 \delta}\pa{ \mathcal{G}_{j_1,l_1,j_2,l_2} (u,v) +  \overline{ \mathcal{G}_{j_1,l_1,j_2,l_2} }(u,v) }\, \mathrm{d}u \, \mathrm{d}v.$$
One easily verifies that the variables $(Z_{j_1,l_1,j_2,l_2})_{ (j_1, l_1)\in \I_{n,p,q},  (j_2, l_2) \in \I_{n,p,q} , (j_1, l_1) \prec (j_2, l_2) }$ are independent. Furthermore,
$$ |Z_{j_1,l_1,j_2,l_2} | \leq  2 \int_{\mathbb{R}^{p + q}} \norm{(u,v)}^{2 \delta}  \prod_{r=1}^p |\widehat{G}(u_r)|^2  \prod_{s=1}^q |\widehat{G}(v_s)|^2  \leq C(p,q,\delta )$$
by Lemma \ref{lemm:Butucea}. Hence, we obtain from Hoeffding's inequality that for all $x>0$, 
$$ \proba{\abs{\sum_{\substack{(j_1, l_1)\in \I_{n,p,q} \\ (j_2, l_2) \in \I_{n,p,q} \\ (j_1, l_1) \neq (j_2, l_2) }} Z_{j_1,l_1,j_2,l_2}} \geq x} \leq 2 \exp \pa{ - \frac{x^2}{2 C^2(p,q,\delta ) M_n^{2p+2q}}}.$$
We deduce from the above inequality that
$$  \proba{ \abs{\sum_{\substack{(j_1, l_1)\in \I_{n,p,q} \\ (j_2, l_2) \in \I_{n,p,q} \\ (j_1, l_1) \neq (j_2, l_2) }} Z_{j_1,l_1,j_2,l_2}} \geq C(p,q,\delta, \eta) M_n^{p+q}} \leq \eta,$$
which yields \eqref{eq:controlH2} and concludes the proof.

\subsection{Proof of Proposition \ref{prop:rapportvrais}}

Let $\mathbb{Z}_n = (X_i,Y_i)_{1 \leq i \leq n}$ be an i.i.d sample with common uniform distribution $P_{f_0}$ on $[0,1]^{p+q}$.

For simplicity, denote for all $1\leq i\leq n$ and all $(j,l)$ in $\I_{n,p,q}$, 
$$a_{i,j,l}\ =\ C_0 h_n^{\delta + (p+q)} g_{n,j,l}(X_i,Y_i)\ =\
C_0 h_n^{\delta + (p+q)}  \prod_{r = 1}^p G_{h_n} (X^{(r)}_i - j_r h_n) \prod_{s = 1}^q G_{h_n}(Y^{(s)}_i - l_s h_n),$$ 
where $g_{n,j,l}$ is defined in Equation \eqref{def:gnpq}, such that $f_\theta(X_i,Y_i)= 1 + \sum_{(j,l)\in \I_{n,p,q}} \theta_{(j,l)} a_{i,j,l}.$
Note that $a_{i,j,l}\neq 0$ if and only if $(X_i,Y_i)$ belongs to the set $D_{(j,l)}$ defined in Equation \eqref{eq:support}.

\paragraph{}
Then, since $f_0 = \mathds{1}_{[0,1]^{p+q}}$, the likelihood ratio equals 
\begin{align*}
L_{\nu_{\rho_n^*}}(\mathbb{Z}_n)  &= \frac{\mathrm{d} P_{\nu_{\rho_n^*}}}{\mathrm{d} P_{f_0}} (\mathbb{Z}_n) = \int \prod_{i = 1}^n \frac{f_\theta }{f_0} (X_i ,Y_i) \pi(\mathrm{d} \theta) \\
& = \esps{\Theta}{ \prod_{i=1}^n \pa{1 + \sum_{(j,l)\in \I_{n,p,q}} \Theta_{(j,l)} a_{i,j,l}} },
\end{align*} 
where $\Theta = (\Theta_{(j,l)})_{(j,l)\in \I_{n,p,q}}$ has i.i.d. Rademacher components $\Theta_{(j,l)}$, and $\esps{\Theta}{\cdot}$ denotes the expectation w.r.t. $\Theta$. 

\paragraph{}
Noticing that for all $1\leq i\leq n$, there exists a unique $(j,l)$ in $\I_{n,p,q}$ such that $a_{i,j,l}\neq 0$, we obtain
$$1+\sum_{(j,l)\in \I_{n,p,q}}\Theta_{(j,l)} a_{i,j,l} = \prod_{(j,l)\in \I_{n,p,q}} \pa{1+\Theta_{(j,l)} a_{i,j,l}}.$$
Thus, 
\begin{align*}
L_{\nu_{\rho_n^*}} (\mathbb{Z}_n) &= \esps{\Theta}{\prod_{(j,l)\in \I_{n,p,q}} \prod_{i=1}^n \pa{1+\Theta_{(j,l)} a_{i,j,l}}} \\
&=\prod_{(j,l)\in \I_{n,p,q}}  \cro{\frac{1}{2}\prod_{i=1}^n \pa{1- a_{i,j,l}} + \frac{1}{2}\prod_{i=1}^n \pa{1+ a_{i,j,l}}}. 
\end{align*}
Moreover, for $\varepsilon$ in $\{-1,1\}$, 
$$\prod_{i=1}^n \pa{1+ \varepsilon a_{i,j,l}} = 1+\sum_{k=1}^n \varepsilon^k \pa{\sum_{1\leq i_1<\ldots<i_k\leq n} a_{i_1,j,l}\ldots a_{i_k,j,l}}.$$
Hence, by cancelling the odd terms, we obtain 
\begin{align*}
\frac{1}{2}\prod_{i=1}^n \pa{1- a_{i,j,l}} + \frac{1}{2}\prod_{i=1}^n \pa{1+ a_{i,j,l}} &= 1+\sum_{k=1}^{[n/2]} \sum_{1\leq i_1<\ldots<i_{2k}\leq n} a_{i_1,j,l}\ldots a_{i_{2k},j,l} \\
&= 1+\sum_{k=1}^{[n/2]} A_{k,j,l}, 
\end{align*}
where $[\cdot]$ denotes the integer part, and 
\begin{equation}
A_{k,j,l} = \sum_{1\leq i_1<\ldots<i_{2k}\leq n} a_{i_1,j,l}\ldots a_{i_{2k},j,l}.
\label{eq:defAkjl}
\end{equation}
Thus, 
$$\cro{L_{\nu_{\rho_n^*}} (\mathbb{Z}_n)}^2 = \prod_{(j,l)\in \I_{n,p,q}}  \pa{1+\sum_{k=1}^{[n/2]} A_{k,j,l}}^2 = \prod_{(j,l)\in \I_{n,p,q}}  \pa{1+ B_{j,l}},$$
where 
\begin{equation}
B_{j,l} = 2\sum_{k=1}^{[n/2]} A_{k,j,l} + \sum_{k,k'=1}^{[n/2]} A_{k,j,l} A_{k',j,l}.
\label{eq:defBjl}
\end{equation}
Then, 
\begin{equation}
\label{esp0Lnu}
\cro{L_{\nu_{\rho_n^*}} (\mathbb{Z}_n)}^2 = 1+ \sum_{m=1}^{M_n^{p+q}} \frac{1}{m!} \sum_{(j_1,l_1),\ldots,(j_m,l_m)}^{\neq} B_{j_1,l_1}  \ldots B_{j_m,l_m},
\end{equation}
where $\displaystyle\sum^{\neq}$ means that the indexes are all distinct. 

After tedious computations, up to a possible permutation of the indexes $(j_1,l_1),\ldots,(j_m,l_m)$, we can express the product $B_{j_1,l_1}  \ldots B_{j_m,l_m}$ as a sum of terms of the form 
\begin{equation}
\label{sommehorriblesimple} 
2^P\sum_{k_1=1}^{[n/2]}\sum_{k_2'',\ldots,k_{M}''=1}^{[n/2]} 2A_{k_1,j_1,l_1}\times A_{k_2'',j_2',l_2'}\times \ldots \times A_{k_M'',j_M',l_M'}
\end{equation}
or
\begin{equation}
\label{sommehorribledouble}
2^Q\sum_{k_1,k_1'=1}^{[n/2]}\sum_{k_2'',\ldots,k_{M}''=1}^{[n/2]} A_{k_1,j_1,l_1}A_{k_1',j_1,l_1}\times A_{k_2'',j_2',l_2'}\times \ldots \times A_{k_M'',j_M',l_M'}
\end{equation}
where $P$ and $Q$ are integers, $M\in\{m-1,\ldots,2m-2\}$ and $(j_2',l_2'),\ldots,(j_M',l_M')$ are drawn in $(j_2,l_2),\ldots,(j_m,l_m)$ such that each $(j_r,l_r)$ for $2\leq r\leq m$ appears exactly once or twice. To be more precise, $P$ and $Q$ count the number of indexes $(j_r,l_r)_r$, appearing exactly once in the product. \\

First note that in Equation \eqref{sommehorriblesimple}, the index $(j_1,l_1)$ appears only once. 
Moreover
\begin{multline*}
\esps{f_0}{A_{k_1,j_1,l_1}\times A_{k_2'',j_2',l_2'}\times \ldots \times A_{k_M'',j_M',l_M'}} \\
 = \sum_{i_{1,1}<\ldots<i_{1,2k_1}} \sum_{i_{2,1}<\ldots<i_{2,2k''_2}} \ldots \sum_{i_{M,1}<\ldots<i_{M,2k''_M}} \mathbb{E}_{f_0}\big[
 a_{i_{1,1},j_1,l_1} \times \ldots \times a_{i_{1,2k_1},j_1,l_1} \times \\
 \times a_{i_{2,1},j_2',l_2'} \times \ldots \times a_{i_{2,2k_2''},j_2',l_2'} \times 
 \ldots 
 \times a_{i_{M,1},j_M',l_M'} \times \ldots \times a_{i_{M,2k_M''},j_M',l_M'}\big]
\end{multline*}
If $i_{1,1}$ appears at least twice in the sums, that is there exists $2\leq r\leq M$ and $1\leq s\leq 2k_r''$ such that $i_{1,1}=i_{r,s}$, then, 
$a_{i_{1,1},j_1,l_1} a_{i_{1,1},j_r,l_r} = 0$ since $D_{j_1,l_1}\cap D_{j_r,l_r}=\emptyset$. 
Otherwise, if $i_{1,1}$ appears only once, by independence between the $(X_i,Y_i)_i$, we obtain that 
\begin{multline*}
\esps{f_0}{a_{i_{1,1},j_1,l_1} \times \ldots \times a_{i_{1,2k_1},j_1,l_1} \times 
  a_{i_{2,1},j_2',l_2'} \times \ldots \times a_{i_{M,2k_M''},j_M',l_M'}} \\
=\esps{f_0}{a_{i_{1,1},j_1,l_1}} \times \esps{f_0}{a_{i_{1,2},j_1,l_1}\ldots \times a_{i_{1,2k_1},j_1,l_1} \times 
  a_{i_{2,1},j_2',l_2'} \times \ldots \times a_{i_{M,2k_M''},j_M',l_M'}} =0
\end{multline*}
Hence, 
\begin{equation*}\label{sommesimplenulle}
\mathbb{E}_{f_0}[A_{k_1,j_1,l_1}\times A_{k_2'',j_2',l_2'}\times \ldots \times A_{k_M'',j_M',l_M'}]=0.
\end{equation*}
and thus, all the terms of the form \eqref{sommehorriblesimple} have a null expectation. \\
Let us now consider Equation \eqref{sommehorribledouble} (where the index $(j_1,l_1)$ appears twice). 
\begin{multline*}
\mathbb{E}_{f_0}[A_{k_1,j_1,l_1}\times A_{k_1',j_1,l_1}\times A_{k_2'',j_2',l_2'}\times \ldots \times A_{k_M'',j_M',l_M'}] 
 = \\
 \sum_{i_{1,1}<\ldots<i_{1,2k_1}}\sum_{i_{1,1}'<\ldots<i_{1,2k_1'}'} \sum_{i_{2,1}<\ldots<i_{2,2k''_2}} \ldots \sum_{i_{M,1}<\ldots<i_{M,2k''_M}} 
 \mathbb{E}_{f_0}\big[
 a_{i_{1,1},j_1,l_1} \times \ldots \times a_{i_{1,2k_1},j_1,l_1} \times a_{i_{1,1}',j_1,l_1} \times  \\
\times \ldots \times a_{i_{1,2k_1'}',j_1,l_1}  \times a_{i_{2,1},j_2',l_2'} \times \ldots \times a_{i_{2,2k_2''},j_2',l_2'} \times 
 \ldots 
 \times a_{i_{M,1},j_M',l_M'} \times \ldots \times a_{i_{M,2k_M''},j_M',l_M'}\big]
\end{multline*}

If there exists at least one index $i_{1,\cdot}$ or $i_{1,\cdot}'$ that can be isolated, then by independence,  
\begin{multline*}
\mathbb{E}_{f_0}\big[
 a_{i_{1,1},j_1,l_1} \times \ldots \times a_{i_{1,2k_1},j_1,l_1} \times a_{i_{1,1}',j_1,l_1} \times \ldots \times a_{i_{1,2k_1'}',j_1,l_1} \times \\
 \times a_{i_{2,1},j_2',l_2'} \times \ldots \times a_{i_{2,2k_2''},j_2',l_2'} \times 
 \ldots 
 \times a_{i_{M,1},j_M',l_M'} \times \ldots \times a_{i_{M,2k_M''},j_M',l_M'}\big] = 0.
\end{multline*}

Hence, the remaining terms are obtained for $k_1=k_1'$ and $i_{1,s}=i_{1,s}'$ for all $1\leq s\leq 2k_1$. 
These arguments are being valid for any index $(j_r,l_r)$, we obtain that $Q=0$ and 

\begin{multline*}
\mathbb{E}_{f_0} \big[B_{j_1,l_1}  \ldots B_{j_m,l_m}\big] \\
 = \sum_{k_1,\ldots,k_m=1}^{[n/2]} \sum_{i_{1,1}<\ldots<i_{1,2k_1}}\ldots \sum_{i_{m,1}<\ldots<i_{m,2k_m}} 
\mathbb{E}_{f_0} \Bigg[ a_{i_{1,1},j_1,l_1}^2 \times \ldots \times a_{i_{1,2k_1},j_1,l_1}^2 \times \ldots \times\\ 
\times a_{i_{m,1},j_m,l_m}^2 \times \ldots \times a_{i_{m,2k_m},j_m,l_m}^2 \Bigg] \\
 = \sum_{k_1,\ldots,k_m=1}^{[n/2]} \sum_{\underset{\Card(I_r)=2k_r}{I_1,\ldots,I_m \subset\{1,\ldots n\}}}
\esps{f_0}{\prod_{i_1\in I_1}a_{i_1,j_1,l_1}^2 \times \ldots \times \prod_{i_m\in I_m}  a_{i_m,j_m,l_m}^2}. 
\end{multline*}

If the subsets $I_r$ are not pairwise disjoints, the product $\prod_{i_1\in I_1}a_{i_1,j_1,l_1}^2 \times \ldots\times \prod_{i_m\in I_m} a_{i_m,j_m,l_m}^2=0$, since the supports $D_{(j_r,l_r)}$ are disjoint. 

Thus, 
\begin{eqnarray*}
\esps{f_0}{B_{j_1,l_1}  \ldots B_{j_m,l_m}} &=& \sum_{k_1,\ldots,k_m=1}^{[n/2]} \sum_{\underset{\underset{\Card(I_r)=2k_r}{I_r\cap I_s = \emptyset, \forall r\neq s}}{{I_1,\ldots,I_m \subset\{1,\ldots n\}}}}\esps{f_0}{\prod_{i_1\in I_1}a_{i_1,j_1,l_1}^2 \times \ldots\times \prod_{i_m\in I_m} a_{i_m,j_m,l_m}^2} \\
&=& \sum_{k_1,\ldots,k_m=1}^{[n/2]} \sum_{\underset{\underset{\Card(I_r)=2k_r}{I_r\cap I_s = \emptyset, \forall r\neq s}}{{I_1,\ldots,I_m \subset\{1,\ldots n\}}}}\prod_{i_1\in I_1} \esps{f_0}{a_{i_1,j_1,l_1}^2} \times\ldots\times \prod_{i_m\in I_m} \esps{f_0}{a_{i_m,j_m,l_m}^2},\\
\end{eqnarray*}
by independence of the $(X_i,Y_i)_{1\leq i\leq n}$. 

Besides, for all $1\leq i\leq n$, 
$$\esps{f_0}{a^2_{i,j,l}}\ =\ C_0^2 h_n^{2\delta + 2p + 2q}\norm{g_{n,j,l}}_2^2\ =\ C(p,q,\delta,R,R',\eta) h_n^{2\delta + p+q},$$ 
since $C_0$ depends on $p,q,\delta,R,R'$ and $\eta$, and by Equation \eqref{eq:norm2g}. 
Thus, 
\begin{eqnarray*}
\esps{f_0}{B_{j_1,l_1}  \ldots B_{j_m,l_m}} &=& \sum_{k_1,\ldots,k_m=1}^{[n/2]} \sum_{\underset{\underset{\Card(I_r)=2k_r}{I_r\cap I_s = \emptyset, \forall r\neq s}}{{I_1,\ldots,I_m \subset\{1,\ldots n\}}}} \pa{C(p,q,\delta,R,R',\eta)h_n^{2\delta + p+q}}^{2k_1+\ldots+2k_m} \\
&=& \sum_{k_1,\ldots,k_m=1}^{[n/2]} \binom{n}{2k_1,\ldots,2k_m,n-\sum_{r=1}^m 2k_r} \pa{C(p,q,\delta,R,R',\eta)h_n^{2\delta + p+q}}^{2k_1+\ldots+2k_m}. 
\end{eqnarray*}
Moreover, the multinomial coefficient can be upper bounded as follows
$$\binom{n}{2k_1,\ldots,2k_m,n-\sum_{r=1}^m 2k_r} \leq n^{2k_1+\ldots+2k_m}.$$ 
Hence, 
\begin{eqnarray*}
\esps{f_0}{B_{j_1,l_1}  \ldots B_{j_m,l_m}} &\leq& \sum_{k_1=1}^{[n/2]}\ldots \sum_{k_m=1}^{[n/2]} \pa{C(p,q,\delta,R,R',\eta) n \times h_n^{2\delta + p+q}}^{2k_1+\ldots+2k_m} \\
&=& \cro{\sum_{k=1}^{[n/2]} \pa{C(p,q,\delta,R,R',\eta) n \times h_n^{2\delta + p+q}}^{2k}}^{m}.
\end{eqnarray*}
Furthermore, for $h_n$ defined in \eqref{eq:defhn} we have
$$h_n \leq C(p,q,\alpha,\beta,\delta,R,R',\eta)n^{-2/(4\delta+p+q)},$$
and thus, whatever the constant $C(p,q,\alpha,\beta,\delta,R,R',\eta)$ is, 
$$C(p,q,\delta,R,R',\eta)\ n\times h_n^{2\delta + p+q} \leq C'(p,q,\alpha,\beta,\delta,R,R',\eta) n^{-(p+q)/(4\delta+p+q)} < 1/2$$ 
for $n$ large enough. Thus, by property of geometric series, we get 
\begin{eqnarray*}
\sum_{k=1}^{[n/2]}\ac{\cro{C(p,q,\delta,R,R',\eta)\ n\times h_n^{2\delta + p+q}}^{2}}^{k} &\leq& \frac{\cro{C(p,q,\delta,R,R',\eta)\ n\times h_n^{2\delta + p+q}}^{2}}{1-\cro{C(p,q,\delta,R,R',\eta)\ n\times h_n^{2\delta + p+q}}^{2}} \\&\leq& \frac{4}{3}\cro{C(p,q,\delta,R,R',\eta)\ n\times h_n^{2\delta + p+q}}^{2}.
\end{eqnarray*}
\paragraph{}
We recall that the constants $C(\cdot)$ may vary from line to line. 
This being true for all $(j,l)$ in $\I_{n,p,q}$, from Equation \eqref{esp0Lnu}, we deduce that 
\begin{align*}
\esps{f_0}{\ac{L_{\nu_{\rho_n^*}} (\mathbb{Z}_n)}^2} 
& \leq 1+\sum_{m=1}^{M_n^{p+q}}  \binom{M_n^{p+q}}{m} \cro{C(p,q,\delta,R,R',\eta)\ n\times h_n^{2\delta + p+q}}^{2m} \\
& \leq 1 + \sum_{m=1}^{M_n^{p+q}} \ac{M_n^{p+q}\cro{C(p,q,\delta,R,R',\eta)\ n\times h_n^{2\delta + p+q}}^{2}}^m \\
& \leq 1 + \sum_{m=1}^{M_n^{p+q}} \cro{C(p,q,\delta,R,R',\eta)\ n^2\times h_n^{4\delta + p+q}}^m,
\end{align*}
since $\binom{M_n^{p+q}}{m}\leq \cro{M_n^{p+q}}^{m}$ and $M_nh_n=1$.
\paragraph{}
Finally, for $h_n$ defined in \eqref{eq:defhn}, with 
$$C(p,q,\alpha,\beta,\delta,R,R',\eta) = \pa{\frac{1}{C(p,q,\delta,R,R',\eta)} \times \frac{4(1 - \alpha - \beta-\eta )^2}{1+4(1 - \alpha - \beta-\eta)^2}}^{1/(4\delta+p+q)},$$
we directly obtain that 
$$C(p,q,\delta,R,R',\eta)\ n^2\times h_n^{4\delta + p+q} \leq \frac{4(1 - \alpha - \beta -\eta)^2}{1+4(1 - \alpha - \beta-\eta)^2} < 1.$$ 
Hence, by property of the geometric series we obtain, 
\begin{eqnarray*}
\esps{f_0}{\pa{L_{\nu_{\rho_n^*}} (\mathbb{Z}_n)}^2} 
&<& 1 + \frac{\cro{C(p,q,\delta,R,R',\eta)\ n^2\times h_n^{4\delta + p+q}}}{1-\cro{C(p,q,\delta,R,R',\eta)\ n^2\times h_n^{4\delta + p+q}}}\\
& < & 1+4(1-\alpha-\beta-\eta)^2,
\end{eqnarray*}
which ends the proof of Proposition \ref{prop:rapportvrais}.


\end{document}